\documentclass[a4paper]{amsart}

\usepackage{amsmath}
\usepackage{amssymb}
\usepackage{latexsym}
\usepackage[mathscr]{eucal}
\usepackage{mathrsfs}
\usepackage{graphics}
\usepackage{newtxtext, newtxmath}
\usepackage{framed}
\usepackage{ulem}
\usepackage{cancel}
\usepackage{arydshln}
\usepackage{bm}
\usepackage{longtable}
\allowdisplaybreaks

\usepackage{type1cm}

\theoremstyle{plain}
\newtheorem{theorem}{Theorem}[section]
\newtheorem{lemma}[theorem]{Lemma}
\newtheorem{proposition}[theorem]{Proposition}
\newtheorem{corollary}[theorem]{Corollary}

\usepackage {anyfontsize}

\newcommand{\relmiddle}[1]{\mathrel{}\middle#1\mathrel{}}

\makeatletter
\def\thefootnote{\ifnum\c@footnote>\z@\leavevmode\lower.5ex%
      \hbox{$^{\@arabic\c@footnote)}$}\fi}
\makeatother

\theoremstyle{definition}

\theoremstyle{remark}

\def\Aut{\mbox{\rm {Aut}}}
\def\det{\mbox{\rm {det}}}

\def\diag{\mbox{\rm {diag}}}

\def\Der{\mbox{\rm {Der}}}

\def\ad{\mbox{\rm {ad}}}

\def\Hom{\mbox{\rm {Hom}}}

\def\Iso{\mbox{\rm {Iso}}}

\def\Ker{\mbox{\rm {Ker}}}

\def\tr{\mbox{\rm {tr}}}

\def\Re{\mbox{\rm {Re}}}

\def\ov{\overline}

\def\ti{\tilde}

\def\dsum{\displaystyle \sum}

\def\dfrac#1#2{\displaystyle \frac{#1}{#2}}

\def\C{\mbox{\boldmath $C$}}

\def\H{\mbox{\boldmath $H$}}

\def\R{\mbox{\boldmath $R$}}

\def\Z{\mbox{\boldmath $Z$}}

\def\sR{\mbox{\boldmath $\scriptstyle{R}$}}
\def\sC{\mbox{\boldmath $\scriptstyle{C}$}}
\def\sH{\mbox{\boldmath $\scriptstyle{H}$}}

\def\0{\mbox{\boldmath {0}}}
\def\1{\mbox{\boldmath {1}}}
\def\2{\mbox{\boldmath {2}}}
\def\3{\mbox{\boldmath {3}}}
\def\4{\mbox{\boldmath {4}}}
\def\5{\mbox{\boldmath {5}}}
\def\6{\mbox{\boldmath {6}}}
\def\7{\mbox{\boldmath {7}}}
\def\8{\mbox{\boldmath {8}}}
\def\9{\mbox{\boldmath {9}}}

\def\a{\mbox{\boldmath $a$}}

\def\m{\mbox{\boldmath $m$}}

\def\tr{\mbox{\rm {tr}}}
\def\ad{\mbox{\rm {ad}}}

\def\det{\mbox{\rm {det}}}
\def\diag{\mbox{\rm {diag}}}
\def\Iso{\mbox{\rm {Iso}}}
\def\Hom{\mbox{\rm {Hom}}}
\def\Ker{\mbox{\rm {Ker}}}
\def\ov{\overline}

\def\dsum{\displaystyle \sum}

\def\sR{\mbox{\boldmath $\scriptstyle{R}$}}
\def\sC{\mbox{\boldmath $\scriptstyle{C}$}}
\def\sH{\mbox{\boldmath $\scriptstyle{H}$}}
\def\dfrac#1#2{\displaystyle \frac{#1}{#2}}

\def\C{\mbox{\boldmath $C$}}

\def\H{\mbox{\boldmath $H$}}

\def\R{\mbox{\boldmath $R$}}

\def\Z{\mbox{\boldmath $Z$}}
\def\sR{\mbox{\boldmath $\scriptstyle{R}$}}

\def\0{\mbox{\boldmath {0}}}
\def\1{\mbox{\boldmath {1}}}
\def\2{\mbox{\boldmath {2}}}
\def\3{\mbox{\boldmath {3}}}
\def\4{\mbox{\boldmath {4}}}
\def\5{\mbox{\boldmath {5}}}
\def\6{\mbox{\boldmath {6}}}
\def\7{\mbox{\boldmath {7}}}
\def\8{\mbox{\boldmath {8}}}
\def\9{\mbox{\boldmath {9}}}
\def\a{\mbox{\boldmath $a$}}

\def\m{\mbox{\boldmath $m$}}

\def\v{\mbox{\boldmath $v$}}

\usepackage{epic,eepic}

\begin{document}

\title[On realizations of the complex Lie groups $ (F_{4,\sR})^C, (E_{6,\sR})^C, (E_{7,\sR})^C ,(E_{8,\sR})^C $  ]
{On realizations of the complex Lie groups $ (F_{4,\sR})^C, (E_{6,\sR})^C, (E_{7,\sR})^C ,(E_{8,\sR})^C$ \\
and those compact real forms $ F_{4,\sR},E_{6,\sR},E_{7,\sR},E_{8,\sR} $}

\author[Toshikazu Miyashita]{Toshikazu Miyashita}


\begin{abstract}
In order to define the complex exceptional Lie groups $ {F_4}^C, {E_6}^C, {E_7}^C, {E_8}^C $ and these compact real forms $ F_4,E_6,E_7,E_8 $, we usually use the Cayley algebra $ \mathfrak{C} $. In the present article, we consider replacing the Cayley algebra $ \mathfrak{C} $ with the field of real numbers $ \R $ in the definition of the groups above, and these groups are denoted as in title above. Our aim is to determine the structure of these groups. We call realization to determine the structure of the groups.
\end{abstract}

\subjclass[2010]{ 53C30, 53C35, 17B40.}

\keywords{exceptional Lie groups}

\address{1365-3 Bessho onsen      \endgraf
	     Ueda City                \endgraf
	     Nagano Pref. 386-1431    \endgraf
	     Japan}
\email{anarchybin@gmail.com}

\maketitle

\setcounter{section}{0}

\section{Introduction}
The realizations of exceptional simple Lie groups of the type $ E_6, E_7 $ and $ E_8 $ had been completed by Ichiro Yokota and the members of his school thirty-five years ago. Until now, their works has yielded the valuable and useful results, and those are summarized in \cite{iy0} and \cite{iy11}.

In order to define the complex exceptional Lie groups $ {F_4}^C, {E_6}^C, {E_7}^C, {E_8}^C $ and these compact real forms $ F_4,E_6,E_7,E_8 $, we usually use the Cayley algebra $ \mathfrak{C} $. We consider replacing $ \mathfrak{C} $ with the field of real numbers $ \R $ in the definition of the  groups above. These groups are denoted by $ (F_{4,\sR})^C, (E_{6,\sR})^C, (E_{7,\sR})^C ,(E_{8,\sR})^C$ as complex Lie groups
and $ F_{4,\sR},E_{6,\sR},E_{7,\sR},E_{8,\sR} $ as these compact real forms. In the present article, our aim is to determine the structure of those groups and to investigate the connectedness of them using the results obtained and the ideas of several proofs in the exceptional Lie groups.

Here, we will sketch out the contents of this article. In Section $ 2 $ (Preliminaries), we describe the definitions of the exceptional Jordan algebras $ \mathfrak{J}(3,\mathfrak{C}), \mathfrak{J}(3,\mathfrak{C}^C) $, the Freudenthal $ C $-vector space $ \mathfrak{P}^C $, the $ 248 $-dimensional $ C $-vector space $ {\mathfrak{e}_8}^C $ and the exceptional Lie groups mentioned above, and moreover the properties related to algebras, vector spaces or the results obtained from the groups. In Section $ 3 $, we define the groups $ (F_{4,\sR})^C $ and $ F_{4,\sR} $ as follows:
\begin{align*}
(F_{4,\sR})^C&=\left\lbrace \alpha \in \Iso_{C}((\mathfrak{J}_{\sR})^C)\relmiddle{|} \alpha(X \circ Y)=\alpha X \circ \alpha Y\right\rbrace,
\\[1mm]
F_{4,\sR}&=\left\lbrace \alpha \in \Iso_{\sR}(\mathfrak{J}_{\sR})\relmiddle{|} \alpha(X \circ Y)=\alpha X \circ \alpha Y\right\rbrace,
\end{align*}
where $ \mathfrak{J}_{\sR} $ is the algebra replaced $ \mathfrak{C} $ with $ \R $ in $ \mathfrak{J}(3,\mathfrak{C}) $, and $ (\mathfrak{J}_{\sR})^C $ is the complexification of $ \mathfrak{J}_{\sR} $. In order to determine the structure of the group $ (F_{4,\sR})^C $, we concretely construct the isomorphism $ SO(3,C) \to (F_{4,\sR})^C $, and using this isomorphism, we determine the structure of the real form $ F_{4,\sR} $ of $ (F_{4,\sR})^C $. Moreover we also describe the uniqueness of the  expression $ \delta=\tilde{A}_1(c_1)+\tilde{A}_2(c_2)+\tilde{A}_3(c_3),c_i \in C $ for any element $ \delta \in (\mathfrak{f}_{4,\sR})^C $ on the Lie algebra $ (\mathfrak{f}_{4,\sR})^C $ of the group $ (F_{4,\sR})^C $.
In addition, in order to determine the type of the group $ (E_{8,\sR})^C $ as Lie algebras, we will determine the root system and the Dynkin diagram of the complex Lie algebras $ (\mathfrak{f}_{4,\sR})^C, (\mathfrak{e}_{6,\sR})^C $ and $ (\mathfrak{e}_{7,\sR})^C $ of the group $ (F_{4,\sR})^C, (E_{6,\sR})^C $ and $ (E_{7,\sR})^C $ through Sections $ 3,4$ and $ 5 $.
Well, in Section $ 4 $, the structure of contents is almost the same as that in Section $ 3 $. Besides, we prove the connectedness of the group $ (E_{6,\sR})^C $ by constructing a homogeneous space, and its result
plays the important role in order to determine to the structure of its group.
In Section $ 5 $, we can not concretely construct a mapping between the group $ (E_{7,\sR})^C $(resp. $ E_{7,\sR} $) and the group $ Sp(3,\H^C) $ (resp. $ Sp(3) $). However, by constructing a mapping between the Lie algebras $ (\mathfrak{e}_{7,\sR})^C $ (resp. $ \mathfrak{e}_{7,\sR} $) and $ \mathfrak{sp}(3,\H^C) $ (resp. $ \mathfrak{sp}(3) $), concretely,
together with the connectedness of the group $ (E_{7,\sR})^C $ (resp. $ E_{7,\sR} $) proved later, the structure of these groups are determined.
In the final section (Section $ 6 $), we can not even construct a mapping between the Lie algebras $ (\mathfrak{e}_{8,\sR})^C $ and $ {\mathfrak{f}_4}^C $. Using the results related to Lie algebras in the previous sections, we can draw the Dynkin Diagram of the Lie algebra $(\mathfrak{e}_{8,\sR})^C  $ of the group $ (E_{8,\sR})^C $, so that we obtain that the group $(E_{8,\sR})^C $ is the type $ F_4 $ as Lie algebras. Moreover, since the connectedness of $(E_{8,\sR})^C $ is proved, the structure of its group is determined. Since the group $ E_{8,\sR} $ is the compact real form of the group $ (E_{8,\sR})^C $,
we can determine the structure of the group $ E_{8,\sR} $.

Our results are given as follows.
\begin{center}
\begin{tabular}{ll}
    $(F_{4,\sR})^C \cong SO(3,C)$, & $ F_{4,\sR} \cong SO(3), $
    \\[3mm]
    $(E_{6,\sR})^C \cong SL(3,C)$, & $E_{6,\sR} \cong SU(3), $
    \\[3mm]
    $(E_{7,\sR})^C \cong Sp(3,\H^C)$, & $ E_{7,\sR} \cong Sp(3), $
    \\[3mm]
    $ (E_{8,\sR})^C \cong {F_4}^C $, & $E_{8,\sR} \cong F_4$.
\end{tabular}
\end{center}
\vspace{2mm}

It is trivial that $ (G_{2,\sR})^C=G_{2,\sR}=\{1\} $, so it is omitted in the present article. As a vision for our future, we will consider replacing $ \mathfrak{C} $ with the fields of complex numbers $ \C $ or quaternion  numbers $ \H $ as mentioned above. To our knowledge, the structures of the groups $  E_{7,\sH}, E_{8, \sC}, E_{8,\sH} $ and these complexification have not been determined. As for the groups which have been already resolved, see \cite{iy11}, \cite{iy0} or \cite{iy9}.

We use the same notations as that in \cite{iy7}, \cite{iy2}, \cite{iy1} or \cite{iy0}. Finally, the author would like to say that this article is to give elementary proofs of the isomorphism and connectedness of the groups.

\section{Preliminaries}

In this section, we will focus on the definitions of the exceptional complex and compact Lie groups of type $ F_4,E_6, E_7 $ and $ E_8 $, the structure of those Lie algebras and several results related to them. Since the proofs are omitted, if necessary, refer to \cite{iy7}, \cite{iy2}, \cite{iy1}, \cite{iy3} for the complex Lie groups and \cite{iy0} for the compact Lie groups.

\subsection{The Lie groups $ {F_4}^C $ and $ F_4 $}

In $\mathfrak{J}(3,\mathfrak{C}^C )$ (resp.\;$ \mathfrak{J}(3,\mathfrak{C}) $), the Jordan multiplication $X \circ Y$, the
inner product $(X,Y)$ and the cross multiplication $X \times Y$, called the Freudenthal multiplication, are defined by
$$
\begin{array}{c}
X \circ Y = \dfrac{1}{2}(XY + YX), \quad (X,Y) = \tr(X \circ Y),
\vspace{1mm}\\
X \times Y = \dfrac{1}{2}(2X \circ Y-\tr(X)Y - \tr(Y)X + (\tr(X)\tr(Y)
- (X, Y))E),
\end{array}
$$
respectively, where $E$ is the $3 \times 3$ unit matrix. Moreover, we define the trilinear form $(X, Y, Z)$, the determinant $\det \,X$ by
$$
(X, Y, Z)=(X, Y \times Z),\quad \det \,X=\dfrac{1}{3}(X, X, X),
$$
respectively, and briefly denote $\mathfrak{J}(3, \mathfrak{C}^C)$ (resp. $ \mathfrak{J}(3,\mathfrak{C}) $)
by $\mathfrak{J}^C$ (resp. $ \mathfrak{J} $).

In $\mathfrak{J}^C$ (resp. $ \mathfrak{J} $), we often use the following nations:
\begin{align*}
E_1 &= \begin{pmatrix}
1 & 0 & 0 \\
0 & 0 & 0 \\
0 & 0 & 0
\end{pmatrix},  \,\,\,\,\,\,\,\,
E_2 = \begin{pmatrix}
0 & 0 & 0 \\
0 & 1 & 0 \\
0 & 0 & 0
\end{pmatrix},  \,\,\,\,\,\,\,\,\,\,
E_3 =\begin{pmatrix}
0 & 0 & 0 \\
0 & 0 & 0 \\
0 & 0 & 1
\end{pmatrix},
\\[2mm]
F_1 (x) &= \begin{pmatrix}
0 &      0 & 0 \\
0 &      0 & x \\
0 & \ov{x} & 0
\end{pmatrix},  \,\,
F_2(x) = \begin{pmatrix}
0 & 0 & \ov{x} \\
0 & 0 & 0 \\
x & 0 & 0
\end{pmatrix},  \,\,
F_3 (x) = \begin{pmatrix}
0 & x & 0 \\
\ov{x} & 0 & 0 \\
0 & 0 & 0
\end{pmatrix}.
\end{align*}

The simply connected complex Lie group ${F_4}^C$ and its compact real form $ F_4 $ are defined by
\begin{align*}
{F_4}^C
&= \{\alpha \in \Iso_C(\mathfrak{J}^C) \, | \, \alpha(X \circ Y) = \alpha X \circ \alpha Y \},
\\[1mm]
F_4
&= \{\alpha \in \Iso_{\sR}(\mathfrak{J}) \, | \, \alpha(X \circ Y) = \alpha X \circ \alpha Y \}.
 \end{align*}

The Lie algebra ${\mathfrak{f}_4}^C$
of the group ${F_4}^C$
is given by
\begin{align*}
{\mathfrak{f}_4}^C=\{\delta \in \mathrm{Hom}_{C}({\mathfrak{J}}^C)\,|\, \delta(X \circ Y)=\delta X \circ Y + X \circ \delta Y  \},
\end{align*}
and any element $\delta$ of the Lie algebras ${\mathfrak{f}_4}^C$ can be uniquely expressed by
\begin{align*}
\delta =(D_1,D_2,D_3) +\ti{A}_1(a_1) +\ti{A}_2(a_2) +\ti{A}_3(a_3), D_i \in \mathfrak{so}(8,C), a_k \in \mathfrak{C}^C,
\end{align*}
where $D_2, D_3 \in \mathfrak{so}(8,C)$ are uniquely determined by the Principle of triality on $ \mathfrak{so}(8,C) $: $(D_1 x)y+x(D_2 y)=\overline{D_3 (\overline{xy})},\, x,y \in \mathfrak{C}^C$ for a given
$D_1 \in \mathfrak{so}(8,C)$ and $\ti{A}_k(a_k)$ is the $C$-linear mapping of $\mathfrak{J}^C$.

\subsection{The Lie groups $ {E_6}^C $ and $ E_6 $}

The simply connected complex Lie group ${E_6}^C$ and its compact real form $ E_6 $ are defined by
\begin{align*}
{E_6}^C &= \{\alpha \in \Iso_C(\mathfrak{J}^C) \, | \, \det \, \alpha X = \det \, X \},
\\[1mm]
E_6 &= \{\alpha \in \Iso_C(\mathfrak{J}^C) \, | \, \det \, \alpha X = \det \, X, \langle \alpha X, \alpha Y \rangle=\langle X, Y \rangle \},
\end{align*}
where the Hermite inner product $ \langle X, Y \rangle $ is defined by $ (\tau X,Y) $: $ \langle X, Y \rangle=(\tau X,Y) $, here $ \tau $ is the complex conjugation in $ \mathfrak{J}^C $.

Then we have naturally the inclusion ${F_4}^C \hookrightarrow {E_6}^C$ and $ ({E_6}^C)_E \cong  {F_4}^C$ (resp. $ F_4 \hookrightarrow E_6$ and $ (E_6)_E \cong F_4 $).

The Lie algebra ${\mathfrak{e}_6}^C$ of the group ${E_6}^C$
is given by
\begin{align*}
{\mathfrak{e}_6}^C=\{\phi \in \mathrm{Hom}_{C}({\mathfrak{J}}^C)\,|\, (\phi X,X,X)=0 \},
\end{align*}
and any element $ \phi $ of Lie algebras $ {\mathfrak{e}_6}^C$ can be uniquely expressed by
\begin{align*}
\phi=\delta+\tilde{T},\,\,\delta \in {\mathfrak{f}_4}^C, \, T \in (\mathfrak{J}^C)_0,
\end{align*}
where $(\mathfrak{J}^C)_0=\{X \in \mathfrak{J}^C \,|\,\tr(X)=0 \}$ and
the $C$-linear mapping $\tilde{T}$ of $\mathfrak{J}^C$ is defined by $\tilde{T}X=T \circ X, X \in \mathfrak{J}^C$.

\subsection{The Lie groups $ {E_7}^C $ and $ E_7 $}

Let $\mathfrak{P}^C$ be the $56$-dimensional Freudenthal $C$-vector space
\begin{align*}
\mathfrak{P}^C = \mathfrak{J}^C \oplus \mathfrak{J}^C \oplus C \oplus C
\end{align*}
with the inner product $ (P,Q) $, the Hermite inner product $ \langle P,Q \rangle $ and the skew-symmetric inner product $ \{ P,Q \} $ which are defined by
\begin{align*}
(P,Q)&=(X,Z)+(Y,W)+\xi\zeta+\eta\omega
\\
\langle P,Q \rangle&=\langle X,Z \rangle+\langle Y,W \rangle+(\tau \xi)+(\tau \eta)\omega,
\\
\{P, Q\}&=(X,W)-(Y,Z)+\xi\omega-\eta\zeta,
\end{align*}
respectively, where $ P=(X,Y,\xi,\eta), Q=(Z,W,\zeta,\omega) \in \mathfrak{P}^C $.

In $ \mathfrak{P}^C $, we often use the following notations:
\begin{align*}
\dot{X}=(X,0,0,0),\,\,\text{\d{$ Y $}}=(0,Y,0,0),\,\dot{1}=(0,0,1,0),\,\,\text{\d{$ 1 $}}=(0,0,0,1).
\end{align*}

For $\phi \in {\mathfrak{e}_6}^C$, $A, B \in \mathfrak{J}^C$ and $\nu \in C$, we define a $C$-linear mapping  $\varPhi(\phi, A, B, \nu) : \mathfrak{P}^C \to \mathfrak{P}^C$ by
\begin{align*}
\varPhi(\phi, A, B, \nu) \begin{pmatrix}X \vspace{1.5mm}\\
Y \vspace{1.5mm}\\
\xi \vspace{.5mm}\\
\eta
\end{pmatrix}
=  \begin{pmatrix}\phi X - \dfrac{1}{3}\nu X + 2B
\times Y + \eta A \vspace{-0.5mm}\\
2A \times X - {}^t\phi Y + \dfrac{1}{3}\nu Y +
\xi B \vspace{1mm}\\
(A, Y) + \nu\xi \vspace{1.5mm}\\
(B, X) - \nu\eta
\end{pmatrix},
\end{align*}
where ${}^t\phi \in {\mathfrak{e}_6}^C$ is the transpose of $\phi$ with respect to the
inner product $(X, Y)$: $({}^t\phi X, Y) = (X, \phi Y)$.

For $ P = (X, Y, \xi, \eta), Q = (Z, W, \zeta, \omega) \in \mathfrak{P}^C $, we define a $C$-linear
mapping $P \times Q: \mathfrak{P}^C \to \mathfrak{P}^C$, called the Freudenthal cross operation, by
\begin{align*}
P \times Q = \varPhi(\phi, A, B, \nu), \quad
\left \{ \begin{array}{l}
\vspace{1mm}
\phi = - \dfrac{1}{2}(X \vee W + Z \vee Y) \\
\vspace{1mm}
A = - \dfrac{1}{4}(2Y \times W - \xi Z - \zeta X) \\
\vspace{1mm}
B =  \dfrac{1}{4}(2X \times Z - \eta W - \omega Y) \\
\vspace{1mm}
\nu = \dfrac{1}{8}((X, W) + (Z, Y) - 3(\xi\omega + \zeta\eta)),
\end{array} \right.
\end{align*}
where $X \vee W \in {\mathfrak{e}_6}^C$ is defined by
\begin{align*}
X \vee W = [\tilde{X}, \tilde{W}] + (X \circ W - \dfrac{1}{3}(X, W)E)^{\sim},
\end{align*}
here the $C$-linear mappings $\tilde{X}, \tilde{W}$ of $\mathfrak{J}^C$ are the same ones as that in ${E_6}^C$.
\vspace{1mm}

The simply connected complex Lie group ${E_7}^C$ and its compact real form $ E_7 $ are defined by
\begin{align*}
{E_7}^C &=  \{\alpha \in \Iso_C(\mathfrak{P}^C) \, | \, \alpha(P \times Q)\alpha^{-1}=\alpha P \times \alpha Q\},
\\[1mm]
E_7 &=  \{\alpha \in \Iso_C(\mathfrak{P}^C) \, | \, \alpha(P \times Q)\alpha^{-1}=\alpha P \times \alpha Q, \langle \alpha P. \alpha Q \rangle=\langle P,Q \rangle \}.
\end{align*}

For $\alpha \in {E_6}^C$, the mapping $\ti{\alpha}: \mathfrak{P}^C \to \mathfrak{P}^C$ is defined by
$$
\ti{\alpha}(X, Y, \xi, \eta)=(\alpha X, {}^t \alpha^{-1}Y, \xi, \eta),
$$
then we have $\ti{\alpha} \in {E_7}^C$, and so $\alpha$ and $\ti{\alpha}$ will be identified. The group ${E_7}^C$ contains ${E_6}^C$ as a subgroup by
$$
{E_6}^C=({E_7}^C)_{\dot{1}, \underset{\dot{}}{1}}(=\{ \alpha \in {E_7}^C \,|\, \alpha \dot{1}=\dot{1}, \alpha \text{\d{$ 1 $}}=\text{\d{$ 1 $}} \}\,).
$$
Hence we have the inclusion $ {F_4}^C \hookrightarrow {E_6}^C \hookrightarrow {E_7}^C$.
For $ \alpha \in E_6 $, suppose $ \alpha \underset{\dot{}}{1}=\underset{\dot{}}{1} $, we have $ \alpha \dot{1}=\dot{1} $, so $ E_7 $ contains $ E_6 $ as a subgroup by
\begin{align*}
  E_6=(E_7)_{\underset{\dot{}}{1}}(:=\{ \alpha \in E_7 \,|\, \alpha \dot{1}=\dot{1} \}\,).
\end{align*}
Hence we have the inclusion $ F_4 \hookrightarrow E_6 \hookrightarrow E_7 $.
Note that $ \alpha \in {E_7}^C $ leaves the skew-symmetric inner product invariant: $ \{\alpha P, \alpha Q \}=\{P,Q \}, P,Q, \in \mathfrak{P}^C $.

The Lie algebra ${\mathfrak{e}_7}^C$ of the group ${E_7}^C$ and the Lie algebra $ \mathfrak{e}_7 $ of the group $ E_7 $ are given by
\begin{align*}
{\mathfrak{e}_7}^C &= \{\varPhi(\phi, A, B, \nu) \in \Hom_C(\mathfrak{P}^C) \, | \, \phi \in {\mathfrak{e}_6}^C, A, B \in \mathfrak{J}^C, \nu \in C\},
\\[1mm]
\mathfrak{e}_7 &= \{\varPhi(\phi, A, -\tau A, \nu) \in \Hom_C(\mathfrak{P}^C) \, | \, \phi \in \mathfrak{e}_6, A \in \mathfrak{J}^C, \nu \in i\R \}.
\end{align*}

\subsection{The Lie groups $ {E_8}^C $ and $ E_8 $ }

Let ${\mathfrak{e}_8}^C$ be the $248$-dimensional $C$-vector space
$$
{\mathfrak{e}_8}^C = {\mathfrak{e}_7}^C \oplus \mathfrak{P}^C \oplus \mathfrak{P}^C \oplus C \oplus C \oplus C.
$$
We define a Lie bracket $[R_1, R_2], R_1\!=\!(\varPhi_1, P_1, Q_1, r_1, s_1, t_1), R_2\!=\!(\varPhi_2, P_2, Q_2, r_2, s_2, t_2)$, by
\begin{align*}
[R_1, R_2]:= (\varPhi, P, Q, r, s, t),\,\,\left\{\begin{array}{l}
\varPhi = {[}\varPhi_1, \varPhi_2] + P_1 \times Q_2 - P_2 \times Q_1
\vspace{1mm} \\
P = \varPhi_1P_2 - \varPhi_2P_1 + r_1P_2 - r_2P_1 + s_1Q_2 - s_2Q_1
\vspace{1mm} \\
Q = \varPhi_1Q_2 - \varPhi_2Q_1 - r_1Q_2 + r_2Q_1 + t_1P_2 - t_2P_1
\vspace{1mm} \\
r = - \dfrac{1}{8}\{P_1, Q_2\} + \dfrac{1}{8}\{P_2, Q_1\} + s_1t_2 - s_2t_1\vspace{1mm} \\
s = \,\,\, \dfrac{1}{4}\{P_1, P_2\} + 2r_1s_2 - 2r_2s_1
\vspace{1mm} \\
t = - \dfrac{1}{4}\{Q_1, Q_2\} - 2r_1t_2 + 2r_2t_1.
\end{array} \right.
\end{align*}
Then the $C$-vector space ${\mathfrak{e}_8}^C$ becomes a complex simple Lie algebra of type $E_8$.

In ${\mathfrak{e}_8}^C$, we often use the following notations:
\begin{align*}
\begin{array}{c}
\varPhi = (\varPhi, 0, 0, 0, 0, 0), \quad P^- = (0, P, 0, 0, 0, 0), \quad Q_- = (0, 0, Q, 0, 0, 0),
\vspace{1mm}\\
\tilde{r} = (0, 0, 0, r, 0, 0), \quad s^- = (0, 0, 0, 0, s, 0), \quad t_- = (0, 0, 0, 0, 0, t).
\end{array}
\end{align*}

We define a $C$-linear transformation $\lambda_\omega$ of ${\mathfrak{e}_8}^C$
by
\begin{align*}
\lambda_\omega(\varPhi, P, Q, r, s, t) = (\lambda\varPhi\lambda^{-1}, \lambda Q, - \lambda P, -r, -t, -s),
\end{align*}
where a $C$-linear transformation $\lambda$ of $\mathfrak{P}^C$ on the right-hand side is defined by
$ \lambda(X, Y, \xi, \eta)$ $ = (Y, - X, \eta, -\xi)$.
As in $\mathfrak{J}^C$, the complex conjugation in ${\mathfrak{e}_8}^C$ is denoted by $\tau$:
\begin{align*}
\tau(\varPhi, P, Q, r, s, t) = (\tau\varPhi\tau, \tau P, \tau Q, \tau r, \tau s, \tau t).
\end{align*}

The simply connected complex Lie group ${E_8}^C$ and it compact real form $E_8$ are defined by
\begin{align*}
{E_8}^C &=\Aut({\mathfrak{e}_8}^C)= \{\alpha \in \Iso_C({\mathfrak{e}_8}^C) \,|\, \alpha[R, R'] = [\alpha R, \alpha R']\},
\\[1mm]
E_8 &= \{\alpha \in {E_8}^C \, | \, \langle \alpha R_1, \alpha R_2 \rangle=\langle R_1,R_2 \rangle \} =({E_8}^C)^{\tau\lambda_\omega },
\end{align*}
where the Hermite inner product $ \langle R_1,R_2 \rangle $ is defined by $ (-1/15)B_8(\tau\lambda_\omega R_1,R_2) $ using the Killing form $ B_8 $ of $ {\mathfrak{e}_8}^C $: $ \langle R_1,R_2 \rangle=(-1/15)B_8(\tau\lambda_\omega R_1,R_2) $.

For $\alpha \in {E_7}^C$, the mapping $\tilde{\alpha} : {\mathfrak{e}_8}^C \to {\mathfrak{e}_8}^C$ is defined by
\begin{align*}
\tilde{\alpha}(\varPhi, P, Q, r, s, t) = (\alpha\varPhi\alpha^{-1}, \alpha P, \alpha Q, r, s, t),
\end{align*}
then we have $\tilde{\alpha} \in {E_8}^C$, and so $\alpha$ and $\tilde{\alpha}$ will be identified. The group ${E_8}^C$ contains the group ${E_7}^C$ as a subgroup defined by
\begin{align*}
{E_7}^C
=({E_8}^C)_{\tilde{1},1^-,1_-}(=\{ \alpha \in ({E_8}^C)\,|\,\alpha \tilde{1}=\tilde{1},\alpha 1^-=1^-,\alpha 1_-=1_- \}\,).
\end{align*}
Hence we have the inclusion ${F_4}^C \hookrightarrow {E_6}^C \hookrightarrow {E_7}^C \hookrightarrow {E_8}^C$. For $ \alpha \in E_8 $, suppose $ \alpha 1_-=1_- $, we have $ \alpha \tilde{1}=\tilde{1} $ and $ \alpha 1^-=1^- $, so $ E_8 $ contains $ E_7 $ as a subgroup defined by
\begin{align*}
E_7=(E_8)_{1_-}(=\left\lbrace \alpha \in E_8 \relmiddle{|}\alpha 1_-=1_-\right\rbrace ).
\end{align*}
Hence we have the inclusion $F_4 \hookrightarrow E_6 \hookrightarrow E_7 \hookrightarrow E_8$.

The Lie algebra $ {\mathfrak{e}_8}^C $ of the group $ {E_8}^C $ and the Lie algebra $\mathfrak{e}_8$ of the group $E_8$ are given by
\begin{align*}
{\mathfrak{e}_8}^C&=\{(\varPhi, P, Q, r, s, t)\,|\,\varPhi \in {\mathfrak{e}_7}^C, P,Q \in \mathfrak{J}^C, r,s,t \in C  \},
\\[1mm]
\mathfrak{e}_8&=\{(\varPhi, P, -\tau\lambda P, r, s, -\tau s)\,|\,\varPhi \in \mathfrak{e}_7, P \in \mathfrak{J}^C, r \in i\R, s \in C  \}.
\end{align*}

\subsection{Useful lemma}

\begin{lemma}{\rm \cite[Proposition 8.7]{iy10}}\label{lemma 2.1}
	For Lie groups $G, G' $,  let a mapping $\varphi : G \to G'$ be a homomorphism of Lie groups. When $G'$ is connected, if $\Ker\,\varphi$ is discrete and $\dim(\mathfrak{g})=\dim(\mathfrak{g}')$, $\varphi$ is surjective.
\end{lemma}

\begin{lemma}{\rm \cite[Lemma 0.7]{iy7}[E. Cartan-Ra\v{s}evskii]}\label{lemma 2.2}
	Let $G$ be a simply connected Lie group with a finite order automorphism $\sigma$
	of $G$. Then $G^\sigma$ is connected.
\end{lemma}
After this, using these lemmas without permission each times, we often prove lemma, proposition or theorem.

\section{The groups $ (F_{4,\sR})^C, F_{4,\sR }$ and the root system, the Dynkin diagram of the Lie algebra $ (\mathfrak{f}_{4,\sR})^C $ }

Let $ \mathfrak{J}(3,\mathfrak{C}^C), \mathfrak{J}(3,\mathfrak{C})  $ be the exceptional Jordan algebras.
We consider the algebras $ \mathfrak{J}(3,\R^C),\allowbreak \mathfrak{J}(3,\R) $ which are defined by replacing $ \mathfrak{C} $ with $ \R $ in $ \mathfrak{J}(3,\mathfrak{C}^C), \mathfrak{J}(3,\mathfrak{C}) $, and we briefly denote $ \mathfrak{J}(3,\R^C),\mathfrak{J}(3,\R) $ by $ (\mathfrak{J}_{\sR})^C, \mathfrak{J}_{\sR} $, respectively.

We consider the following groups $ (F_{4,\sR})^C $ and $F_{4,\sR} $ which are  respectively defined by replacing $ \mathfrak{C} $ with $ \R $ in the groups $ {F_4}^C $ and $ F_4 $:
\begin{align*}
(F_{4,\sR})^C&=\left\lbrace \alpha \in \Iso_{C}((\mathfrak{J}_{\sR})^C)\relmiddle{|} \alpha(X \circ Y)=\alpha X \circ \alpha Y\right\rbrace,
\\[1mm]
F_{4,\sR}&=\left\lbrace \alpha \in \Iso_{\sR}(\mathfrak{J}_{\sR})\relmiddle{|} \alpha(X \circ Y)=\alpha X \circ \alpha Y\right\rbrace.
\end{align*}
The group $ F_{4,\sR} $ is a compact group as a closed subgroup of the orthogonal group $ O(6)=O(\mathfrak{J}_{\sR})=\{\alpha \in \Iso_{\sR}(\mathfrak{J}_{\sR}) \,|\,(\alpha X,\alpha Y)=(X,Y)\} $ and the group $ (F_{4,\sR})^C $ is the complexification of $ F_{4,\sR} $.

We will prove lemma needed in the proof of Theorem \ref{theorem 3.2} below.

\begin{lemma}\label{lemma 3.1}
Any element $ X \in (\mathfrak{J}_{\sR})^C $ such that $ X^2=X, \tr(X)=1 $ can be transformed to any $ E_i,\allowbreak i=1,2,3 $ by some $ B \in O(3,C) ${\rm :} $ {}^t BXB=E_i $.
\end{lemma}
\begin{proof}
It is well-known that any element $ X \in \mathfrak{J}_{\sR} $ can be transformed to diagonal form by a certain $ B \in O(3) $. Hence, since $ O(3,C) $ contains the subgroup $ O(3) $, we may assume $ X \in (\mathfrak{J}_{\sR})^C $ as
\begin{align*}
X=\begin{pmatrix}
\xi_1 & ix_3 & ix_2 \\
ix_3 & \xi_2 & ix_1  \\
ix_2 & ix_1 & \xi_3
\end{pmatrix},\;\;\begin{array}{l}
\xi_i \in C, \xi_1+\xi_2+\xi_3=1,\\
x_i \in \R.
\end{array}
\end{align*}
Subsequently, the computation of $ X^2 $ is obtained as follows:
\begin{align*}
X^2=\begin{pmatrix}
{\xi_1}^2-{x_2}^2-{x_3}^2 & -x_2x_1+i(\xi_1+\xi_2)x_3 &  * \\
* & {\xi_2}^2-{x_3}^2-{x_1}^2 & -x_3x_2+i(\xi_2+\xi_3)x_1 \\
-x_1x_3+i(\xi_3+\xi_1)x_2 &  *  & {\xi_3}^2-{x_1}^2-{x_2}^2
\end{pmatrix} \in (\mathfrak{J}_{\sR})^C.
\end{align*}
Then we compare the diagonals of both of $ X^2=X $, so that we have that each $ \xi_i $ is real numbers. Indeed, for instance, since $ {\xi_1}^2-{x_2}^2-{x_3}^2=\xi_1 $, that is, $  {\xi_1}^2-\xi_1-{x_2}^2-{x_3}^2=0 $, we have $ \xi_1=(1/2)(1\pm\sqrt{1+({x_2}^2+{x_3}^2)}) \in \R$, so are $ \xi_2,\xi_3 $. Hence by comparing $ F_k(x_k),k=1,2,3 $-part, we have
\begin{align*}
x_1x_2=x_2x_3=x_3x_1=0, \;\; \xi_1x_1=\xi_2x_2=\xi_3x_3=0 \;\cdots \,(*).
\end{align*}

In the case where $ x_1=x_2=x_3=0 $. Since $ X $ is diagonal form, we have $ \xi_i=0 $ or $ \xi_i=1 $ from $ X^2=X $, and together with $ \tr(X)=1 $, we see $ X=E_1, X=E_2 $ or $ X=E_3 $. Set $ C_2:=\begin{pmatrix}
0 & 1 & 0 \\
1 & 0 & 0 \\
0 & 0 & 1
\end{pmatrix}, C_3:=\begin{pmatrix}
1 & 0 & 0 \\
0  & 0 & 1 \\
0  & 1 & 0
\end{pmatrix} $, so we easily see $ C_k \in O(3) \subset O(3,C),\vspace{1mm} k=2,3 $. Then we have
\begin{align*}
{}^t \! C_2E_1C_2=E_2,\;\; {}^t \! C_3E_2C_3=E_3, \;\; {}^t\!(C_3C_2)^{-1}E_3(C_3C_2)^{-1}=E_1.
\end{align*}

In the case where $ x_1\not=0 $. Then we have $ x_2=x_3=0, \xi_1=0, \xi_2+\xi_3=1 $ from the formulas $ (*) $ above. Hence $ X $ is of the form $ \begin{pmatrix}
0 & 0 & 0 \\
0 & \xi_2 & ix_1 \\
0 & ix_1 & \xi_3
\end{pmatrix}$ with $ {x_1}^2=-\xi_2\xi_3 $. Here, note that $ \xi_2\xi_3 <0 $, if $ \xi_2 >0 $ and $ \xi_3 <0 $, $ X $ can be transformed to $ E_2 $ by $ B_1:=\begin{pmatrix}
1 & 0 & 0 \\
0 & x_1/\sqrt{-\xi_3} & -ix_1/\sqrt{\xi_2} \\
0 & i\sqrt{-\xi_3} & \sqrt{\xi_2}
\end{pmatrix} \in O(3,C)$. Indeed, first it follows  from
\begin{align*}
B_1{}^tB_1=\begin{pmatrix}
1 & 0 & 0 \\
0 & x_1/\sqrt{-\xi_3} & -ix_1/\sqrt{\xi_2} \\
0 & i\sqrt{-\xi_3} & \sqrt{\xi_2}
\end{pmatrix}
\begin{pmatrix}
1 & 0 & 0 \\
0 & x_1/\sqrt{-\xi_3} & i\sqrt{-\xi_3} \\
0 & -ix_1/\sqrt{\xi_2} & \sqrt{\xi_2}
\end{pmatrix}
=E
\end{align*}
that $ B_1 \in O(3,C) $. By straightforward computation, we have $ {}^tB_1XB_1=E_2 $.
If $ \xi_2 <0 $ and $ \xi_3 >0 $, $ X $ can be also transformed to $ E_2 $ by $ B_2:=\begin{pmatrix}
1 & 0 & 0 \\
0 & i\sqrt{-\xi_2} & \sqrt{\xi_3} \\
0 & x_1/\sqrt{-\xi_2} & -ix_1/\sqrt{\xi_3}
\end{pmatrix} \in U(3,\C^C) $. Indeed, it follows from
\begin{align*}
B_2{}^tB_2=\begin{pmatrix}
1 & 0 & 0 \\
0 & i\sqrt{-\xi_2} & \sqrt{\xi_3} \\
0 & x_1/\sqrt{-\xi_2} & -ix_1/\sqrt{\xi_3}
\end{pmatrix}\begin{pmatrix}
1 & 0 & 0 \\
0 & i\sqrt{-\xi_2} & x_1/\sqrt{-\xi_2} \\
0 & \sqrt{\xi_3} & -ix_1/\sqrt{\xi_3}
\end{pmatrix}=E
\end{align*}
that $ B_2 \in O(3,C) $. As in the case above, we have $ {}^tB_2XB_2=E_2 $. 
Hence, as in the case where $ x_1=x_2=x_3=0 $, by operating $ C_2,C_3 $ or $ C_2C_3 $ on the both sides, $ X $ can be transformed to any $ E_i, i=1,2,3 $.

 In the case where $ x_2\not=0 $. As in the case where $ x_1\not=0 $, $ X $ is of the form $ \begin{pmatrix}
\xi_1 & 0 & i\ov{x}_2 \\
0 & 0 & 0 \\
ix_2 & 0 & \xi_3
\end{pmatrix} $ with $ {x_2}^2=-\xi_3\xi_1, \xi_3+\xi_1=1 $. Let $ C_2 $. Then we have
\begin{align*}
C_2X{}^t\!C_2=\begin{pmatrix}
0 & 0 & 0  \\
0 & \xi_1 & ix_2 \\
0 & ix_2 & \xi_3
\end{pmatrix}.
\end{align*}
Hence this case is reduced to the case where $ x_1\not=0 $.

In the case where $ x_3\not=0 $. As in the case where $ x_1\not=0 $, $ X $ is of the form $ \begin{pmatrix}
\xi_1 & ix_3 & 0 \\
ix_3 & \xi_2 & 0 \\
0 & 0 & 0
\end{pmatrix} $ with $ {x_3}^2=-\xi_1\xi_2, \xi_1+\xi_2=1 $. Let $ C_3 $ Then we have
\begin{align*}
C_3X{}^t\!C_3=\begin{pmatrix}
\xi_1 & 0 & ix_3  \\
0 & 0 & 0 \\
ix_3 & 0 & \xi_2
\end{pmatrix}.
\end{align*}
Hence this case is also reduced to the case where $ x_2\not=0 $.

 With above, the proof of this lemma is completed.
\end{proof}

We will determine the structure of the group $ (F_{4,\sR})^C $.

\begin{theorem}\label{theorem 3.2}
  The group $ (F_{4,\sR})^C $ is isomorphic to the group $ SO(3,C) $ {\rm :} $ (F_{4,\sR})^C \cong SO(3,C) $.
\end{theorem}
\begin{proof}
 Let the group $ SO(3,C)=\{ A \in M(3,C) \,|\,A\,{}^t\! A=E,\det \,A=1\} $. Then we define a mapping $ f_{4,C}: SO(3,C) \to (F_{4,\sR})^C $ by
 \begin{align*}
 f_{4,C}(A)X=AX{}^t\!A,\,\, X \in (\mathfrak{J}_{\sR})^C.
 \end{align*}

 First, we will prove that $ f_{4,C} $ is well-defined. It is clear $ f_{4,C}(A) \in \Iso_{C}((\mathfrak{J}_{\sR})^C)$, and we can easy show $ f_{4,C}(A)X \circ f_{4,C}(A)Y=f_{4,C}(A)(X \circ Y) $. Indeed, it follows that
 \begin{align*}
 f_{4,C}(A)X \circ f_{4,C}(A)Y
 &=(AX{}^t\!A)\circ (AY{}^t\!A)
 \\
 &=(1/2)\left((AX{}^t\!A)(AY{}^t\!A) +(AY{}^t\!A)(AX{}^t\!A) \right)
 \\
 &=(1/2)\left((AXY{}^t\!A) +(AYX{}^t\!A) \right)
 \\
 &=\left(A(1/2)(XY+YX){}^t\!A \right)
 \\
 &=A(X \circ Y){}^t\!A
 \\
 &=f_{4,C}(A)(X \circ Y).
 \end{align*}
 Hence $ f_{4,C} $ is well-defined. Subsequently, we will prove that $ f_{4,C} $ is a homomorphism. For $ A,B \in SO(3,C) $, it follows that
\begin{align*}
f_{4,C}(AB)X=(AB)X{}^t(AB)=A(BX{}^t\!B){}^t\!A=f_{4,C}(A)f_{4,C}(B)X,\,\,X \in (\mathfrak{J}_{\sR})^C,
\end{align*}
that is, $ f_{4,C}(AB)=f_{4,C}(A)f_{4,C}(B) $.

Next, we will prove that $ f_{4,C} $ is surjective. Let $ \alpha \in (F_{4,\sR})^C $. Then, since $ E_i \circ E_i=E_i $ and $ \tr(\alpha E_i)=\tr(E_i),i=1,2,3 $, we have $ (\alpha E_i)^2=\alpha E_i $ and $ \tr(\alpha E_i)=1 $.  Hence there exists $ A_i \in O(3,C) $ such that $ E_i={}^t \! A_i(\alpha E_i)A_i, i=1,2,3 $ from Lemma \ref{lemma 3.1}, that is, $ \alpha E_i=A_iE_i{}^t \! A_i $.

Here, set $ \a_i:=\begin{pmatrix}
a_{1i} \\
a_{2i} \\
a_{3i}
\end{pmatrix} $ be the $ i $-th column vector of $ A_i $, then we have
\begin{align*}
\alpha E_i=A_iE_i{}^t\!A_i=
\begin{pmatrix}
{a_{1i}}^2 & a_{1i}a_{2i} & a_{1i}a_{3i} \\
a_{2i}a_{1i} & {a_{2i}}^2 & a_{2i}a_{3i} \\
a_{3i}a_{1i} & a_{3i}a_{2i} & {a_{3i}}^2
\end{pmatrix},\,\, i=1,2,3.
\end{align*}
We construct a matrix $ A:=(\a_1,\a_2,\a_3) $, then we have
\begin{align*}
\alpha E_i=AE_i{}^t\!A,\,\,i=1,2,3.
\end{align*}
Hence we have $ A \in O(3,C) $. Indeed, it follows that
\begin{align*}
A{}^t\!A&=AE{}^t\!A=A(E_1+E_2+E_3){}^t\!A=AE_1{}^t\!A+AE_2{}^t\!A+AE_3{}^t\!A
\\
&=\alpha E_1+\alpha E_2+\alpha E_3=\alpha(E_1+E_2+E_3)=\alpha E
\\
&=E.
\end{align*}
Note that we may assume $ A \in SO(3,C) $, if necessary, replace $ \a_1 $ with $ \a_1/\det A $ for $ A $ constructed above, so hereafter, we regard $ A \in SO(3,C) $.

Set $ \beta:={f_{4,C}(A)}^{-1}\alpha \in (F_{4,\sR})^C $, then $ \beta $ satisfies
\begin{align*}
\beta E_i=E_i, \,\,i=1,2,3.
\end{align*}
Indeed, it follows that
\begin{align*}
\beta E_i={f_{4,C}(A)}^{-1}\alpha E_i=f_{4,C}(A^{-1})(AE_i{}^t\!A)=A^{-1}(AE_i{}^t\!A){}^t\!(A^{-1})=E_i.
\end{align*}

We consider a space $ \mathfrak{J}_i:=\{X \in (\mathfrak{J}_{\sR})^C \,|\, 2E_{i+1}\circ X=2E_{i+2} \circ X=X \}=\{F_i(x)\,|\, x \in C\},i=1,2,3 $. Then, since $ \beta $ satisfies $ \beta X \in \mathfrak{J}_i $ for $ X \in \mathfrak{J}_i $, $ \beta $ induces a $ C $-linear transformation $ \beta_i:C \to C $ such that
\begin{align*}
\beta F_i(x)=F_i(\beta_i x),\,\,i=1,2,3.
\end{align*}
Moreover, apply $ \beta $ on the both sides of $ F_i(x) \circ F_i(y)=(x,y)(E_{i+1}+E_{i+2}) $, then since it follows that
\begin{align*}
\beta(F_i(x) \circ F_i(y))&=\beta F_i(x) \circ \beta F_i(y)=F_i(\beta_ix) \circ F_i(\beta_i y)=(\beta_i x,\beta_i y)(E_{i+1}+E_{i+2}),
\\
\beta(x,y)(E_{i+1}+E_{i+2})&=(x,y)(E_{i+1}+E_{i+2}),
\end{align*}
we have $ (\beta_i x,\beta_i y)=(x,y) $, that is, $ (\beta_i x)(\beta_i y)=xy $. Hence, since $ \beta_i $ is the $ C $-linear transformation of $ C $, we have $ (\beta_i 1)^2=1 $, that is, $ \beta_i 1=1 $ or $ \beta_i 1=-1 $.

\noindent In addition, apply $ \beta $ on the both sides of $ 2F_1(x) \circ F_2(y)=F_3(xy) $, then it follows that
\begin{align*}
(\beta_1 x)(\beta_2 y)=\beta_3(xy).
\end{align*}
Hence, together with $ \beta_i=1 $ or $ \beta_i=-1 $, we have the following
\begin{align*}
\left\lbrace
\begin{array}{l}
\beta_1=1 \\
\beta_2=1 \\
\beta_3=1,
\end{array}
\right.\,\,
\left\lbrace
\begin{array}{l}
\beta_1=1 \\
\beta_2=-1 \\
\beta_3=-1,
\end{array}
\right. \,\,
\left\lbrace
\begin{array}{l}
\beta_1=-1 \\
\beta_2=1 \\
\beta_3=-1,
\end{array}
\right. \,\,
\left\lbrace
\begin{array}{l}
\beta_1=-1 \\
\beta_2=-1 \\
\beta_3=1.
\end{array}
\right.
\end{align*}
Here, we construct a matrix $ B \in SO(3) \subset SO(3,C) $ corresponding to each cases above as follows:
\begin{align*}
B=\begin{pmatrix}
1 & & \\
& 1 & \\
& & 1
\end{pmatrix},\,\,
\begin{pmatrix}
1 & & \\
& -1 & \\
& & -1
\end{pmatrix},\,\,
\begin{pmatrix}
-1 & & \\
& 1 & \\
& & -1
\end{pmatrix},\,\,
\begin{pmatrix}
-1 & & \\
& -1 & \\
& & 1
\end{pmatrix}.
\end{align*}
Hence, from $ \beta E_i=E_i, \beta F_i(x)=F_i(\beta_i x) $, we have
\begin{align*}
\beta X=BX{}^t\!B,\,\,X \in (\mathfrak{J}_{\sR})^C.
\end{align*}
Thus there exists $ B \in SO(3) \subset SO(3,C) $ such that $ \beta=f_{4,C}(B) $.

\noindent With above,  we obtain
\begin{align*}
\alpha=f_{4,C}(A)\beta=f_{4,C}(A)f_{4,C}(B)=f_{4,C}(AB), AB \in SO(3,C).
\end{align*}
This shows that $ f_{4,C} $ is surjective.

Finally, we will determine $ \Ker\,f_{4,C} $. From the definition of kernel,
we have the following
\begin{align*}
\Ker\,f_{4,C}&=\left\lbrace A \in SO(3,C) \relmiddle{|} f_{4,C}(A)=E \right\rbrace
\\
&=\left\lbrace A \in SO(3,C) \relmiddle{|} AX{}^t\!A=X \,\,{\text{for all}}\,\,X \in (\mathfrak{J}_{\sR})^C \right\rbrace
\\
&=\left\lbrace A \in SO(3,C) \relmiddle{|} AX=XA, \,\,{\text{for all}}\,\,X \in (\mathfrak{J}_{\sR})^C \right\rbrace
\\
&=\left\lbrace E \right\rbrace.
\end{align*}
Indeed, for $ X=E_i, i=1,2,3 $, we have that $ A $ is of form $ \diag(a_{11},a_{22},a_{33}), {a_{ii}}^2=1 $. In addition, for $ X=F_i(1),i=1,2 $, we have that $ A $ is of form $ \diag(a,a,a),a \in C, a^2=1 $: $ A=\diag(a,a,a) $. Hence, from $ \det\,A=1 $, we obtain $ \Ker\,f_{4,C}=\{E\} $.

Therefore we have the required isomorphism
\begin{align*}
(F_{4,\sR})^C \cong SO(3,C).
\end{align*}
\end{proof}

We consider the following subgroup $ ((F_{4,\sR})^C)^\tau $ of $ (F_{4,\sR})^C $:
\begin{align*}
((F_{4,\sR})^C)^\tau=\left\lbrace \alpha \in (F_{4,\sR})^C \relmiddle{|}\tau\alpha=\alpha\tau \right\rbrace,
\end{align*}
where $ \tau $ is the complex conjugation of $ (\mathfrak{J}_{\sR})^C $.

Then we will prove the following proposition which is needed to determine the structure of the group $ F_{4,\sR} $.

\begin{proposition}\label{proposition 3.3}
  The group $ F_{4,\sR} $ coincides with the group $ ((F_{4,\sR})^C)^\tau ${\rm :} $ F_{4,\sR}=((F_{4,\sR})^C)^\tau $.
\end{proposition}
\begin{proof}
Let $ \alpha \in F_{4,\sR} $. Then its complexification mapping $ \alpha^C: (\mathfrak{J}_{\sR})^C \to (\mathfrak{J}_{\sR})^C, \alpha^C(X_1+iX_2)=\alpha X_1 \allowbreak +i\alpha X_2, X_i \in \mathfrak{J}_{\sR} $ , belongs to $ (F_{4,\sR})^C $: $ \alpha^C \in (F_{4,\sR})^C $. Hence we regard $ F_{4,\sR} $ as a subgroup $ (F_{4,\sR})^C $ identifying $ \alpha $ with $ \alpha^C $: $ F_{4,\sR} \subset (F_{4,\sR})^C $.

Conversely, let $ \beta \in ((F_{4,\sR})^C)^\tau $. For $ X \in \mathfrak{J}_{\sR} $, we have $ \tau(\beta X)=\beta(\tau X)=\alpha X $ from $ \tau\beta=\beta\tau $, that is, $ \alpha X \in \mathfrak{J}_{\sR} $. Hence, since $ \alpha $ induces $ \R $-linear isomorphism of $ \mathfrak{J}_{\sR} $, we have $ \beta \in F_{4,\sR} $, that is, $ ((F_{4,\sR})^C)^\tau \subset F_{4,\sR} $.

   With above, the proof of this proposition is completed.
\end{proof}

Now, we will determine the structure of the group $ F_{4,\sR} $.

\begin{theorem}\label{theorem 3.4}
  The group $ F_{4,\sR} $ is isomorphic to the group $ SO(3) ${\rm :} $ F_{4,\sR} \cong SO(3) $.
\end{theorem}
\begin{proof}
   Let the group $ F_{4,\sR} $ as the group $ ((F_{4,\sR})^C)^\tau $ (Proposition \ref{proposition 3.3}). We define a mapping $ f_4: SO(3) \to ((F_{4,\sR})^C)^\tau $ by the restriction of the mapping $ f_{4,C} $: $ f_4(P)=f_{4,C}(P) $.

First, we will prove that $ f_4 $ is well-defined. However, since $ f_4 $ is the restriction of the mapping $ f_{4,C} $, we easily see $ f_4(P) \in (F_{4,\sR})^C $. Moreover, it follows that
\begin{align*}
\tau f_4(P)\tau X=\tau f_{4,C}(P)\tau X=\tau(P(\tau X){}^t\!P)
=(\tau P)X(\tau{}^t\!P)=PX{}^t\!P=f_4(P)X,\,\,X \in (\mathfrak{J}_{\sR}),
\end{align*}
that is, $ \tau f_4(P)=f_4(P)\tau $. Hence we have $ f_4(P) \in ((F_{4,\sR})^C)^\tau $. With above, $ f_4 $ is well-defined. Subsequently, we will prove that $ f_4 $ is a homomorphism. However, since $ f_4 $ is the restriction of the mapping $ f_{4,C} $, it is clear.

Next, we will prove that $ f_4 $ is surjective. Let $ \alpha \in ((F_{4,\sR})^C)^\tau  \subset (F_{4,\sR})^C$. Then there exists $ A \in SO(3,C) $ such that $ \alpha=f_{4,C}(A) $ (Theorem \ref{theorem 3.2}). Moreover, since $ \alpha $ satisfies the condition $ \tau\alpha=\alpha\tau $, it follows from
\begin{align*}
\tau\alpha\tau=\tau f_{4,C}(A)\tau=f_{4,C}(\tau A)
\end{align*}
that $ \tau A=A $, that is, $ A \in (SO(3,C))^\tau=\{A \in SO(3,C)\,|\, \tau A=A \}=SO(3) $. Hence there exists $ P \in SO(3) $ such that $ \alpha=f_{4,C}(P)=f_4(P) $, so that $ f_4 $ is surjective.

Finally, we will determine $ \Ker\,f_4 $. Using the result of $ \Ker\,f_{4,C} $, we easily see $ \Ker\,f_4=\{E\} $.

Therefore, from $ F_{4,\sR}=((F_{4,\sR})^C)^\tau $(Proposition \ref{proposition 3.3}), we have the required isomorphism
\begin{align*}
  F_{4,\sR} \cong SO(3).
\end{align*}
\end{proof}

We will study the Lie algebra $ (\mathfrak{f}_{4,\sR})^C $ of the group $ (F_{4,\sR})^C $.
\vspace{1mm}

The Lie algebra $ (\mathfrak{f}_{4,\sR})^C $ of the group $ (F_{4,\sR})^C $ is given by
\begin{align*}
(\mathfrak{f}_{4,\sR})^C=\left\lbrace \delta \in \Hom_{C}((\mathfrak{J}_{\sR})^C)\relmiddle{|} \delta (X \circ Y)=\delta X \circ Y +X \circ \delta Y \right\rbrace.
\end{align*}

For $ c \in C $, we adopt the following notations:
\begin{align*}
A_1(c)
=\begin{pmatrix}
0 & 0 & 0 \\
0 & 0 & c \\
0 & -c & 0
\end{pmatrix},\quad
A_2(c)
=\begin{pmatrix}
0 & 0 & -c \\
0 & 0 & 0 \\
c & 0 & 0
\end{pmatrix},\quad
A_3(a)
=\begin{pmatrix}
0 & c & 0 \\
-c & 0 & 0 \\
0 & 0 & 0
\end{pmatrix},
\end{align*}
so using $ A_i(c),i=1,2,3 $ above, we define a $ C $-linear transformation $ \tilde{A}_i(c) $ of $ (\mathfrak{J}_{\sR})^C $ by
\begin{align*}
\tilde{A}_i(c)X=\dfrac{1}{2}\left( A_i(c)X-XA_i(c)\right) ,\,\,X \in (\mathfrak{J}_{\sR})^C.
\end{align*}
Then we have $ \tilde{A}_i(c) \in (\mathfrak{f}_{4,\sR})^C $ with the following properties
\begin{align*}
(\ast) \,\,\left\{\begin{array}{l}
\tilde{A}_i(c)E_{i} = 0
\vspace{1mm}\\
\tilde{A}_i(c)E_{i+1} = - \dfrac{1}{2}F_i(c)
\vspace{1mm}\\
\tilde{A}_i(c)E_{i+2} = \dfrac{1}{2}F_i(c),
\end{array} \right.
\quad
\left\{\begin{array}{l}
\tilde{A}_i(c)F_i(x) = (cx)(E_{i+1} - E_{i+2})
\vspace{1mm}\\
\tilde{A}_i(c)F_{i+1}(x) = \dfrac{1}{2}F_{i+2}(cx)
\vspace{1mm}\\
\tilde{A}_i(c)F_{i+2}(x) = - \dfrac{1}{2}F_{i+1}(cx),
\end{array} \right.
\end{align*}
where the indexes are considered as mod $ 3 $.

The differential mapping $ f_{4,C_*}:\mathfrak{so}(3,C) \to (\mathfrak{f}_{4,\sR})^C$ of the mapping $ f_{4,C} $ is given by
\begin{align*}
f_{4,C_*}(D)X=DX+X{}^t\!D,\,\,X \in (\mathfrak{J}_{\sR})^C.
\end{align*}

Then, using the differential mapping $ f_{4,C_*} $, we have the following theorem.

\begin{theorem}\label{theorem 3.5}
 Any element $ \delta \in (\mathfrak{f}_{4,\sR})^C $ is uniquely expressed by
\begin{align*}
 \delta=\tilde{A}_1(c_1)+\tilde{A}_2(c_2)+\tilde{A}_3(c_3),\,\,c_i \in C.
\end{align*}

In particular, we have $ \dim_C((\mathfrak{f}_{4,\sR})^C )=3 $.
\end{theorem}
\begin{proof}
First, we see that the differential mapping $ f_{4,C_*} $ induces the isomorphism $ (\mathfrak{f}_{4,\sR})^C \cong \mathfrak{so}(3,C) $ from Theorem \ref{theorem 3.2}.

Set $ D:=
    \begin{pmatrix}
    0 & d_3 & d_2  \\
    -d_3 & 0 & d_1 \\
    -d_2 & -d_1 & 0
    \end{pmatrix} \in \mathfrak{so}(3,C)$, then by doing straightforward computation of $ f_{4,C_*}(D)X $ as $ X=
    \begin{pmatrix}
    \xi_1 & x_3 & x_2 \\
    x_3 & \xi_2 & x_1 \\
    x_2 & x_1 & \xi_3
    \end{pmatrix} \in (\mathfrak{J}_{\sR})^C$, we have
    \begin{align*}
    f_{4,C_*}(D)X=2(\tilde{A}_1(d_1)+\tilde{A}_2(d_2)+\tilde{A}_3(d_3))X.
    \end{align*}
    Hence $ \delta_C $ can be expressed
    $ \delta_C=\tilde{A}_1(c_1)+\tilde{A}_2(c_2)+\tilde{A}_3(c_3), c_i \in C $. In order to prove the uniqueness of the expression, it is sufficient to show that
    \begin{align*}
    \tilde{A}_1(c_1)+\tilde{A}_2(c_2)+\tilde{A}_3(c_3)=0, \,\,{\text{then}}\,\,A_1(c_1)=A_2(c_2)=A_3(c_3)=0.
    \end{align*}
    Indeed, if we apply it on $ E_1 $, then we have $ A_2(c_2)=A_3(c_3)=0 $, and so $ A_1(c_1)=0 $. Hence we easily see $ \dim_C((\mathfrak{f}_{4,\sR})^C )=3 $.
\end{proof}

We move the determination of the root system and Dynkin diagram of the Lie algebra $ (\mathfrak{f}_{4,\sR})^C $.

We define a Lie subalgebra $ \mathfrak{h}_4 $ of $ (\mathfrak{f}_{4,\sR})^C $ by
\begin{align*}
\mathfrak{h}_4=\left\lbrace \delta_4=\tilde{A}_1(-ia) \relmiddle{|} a \in C \right\rbrace.
\end{align*}
Note that the Lie brackets $ [\tilde{A}_i(a),\tilde{A}_i(b)]=0,[\tilde{A}_i(a),\tilde{A}_{i+1}(b)]=-(1/2)\tilde{A}_{i+2}(ab), a,b \in C $ hold, where the indexes $ i $ are considered as mod $ 3 $. Then $ \mathfrak{h}_4 $ is a Cartan subalgebra of $ (\mathfrak{f}_{4,\sR})^C $.

\begin{theorem}\label{theorem 3.6}
The roots $ \varDelta $ of $ (\mathfrak{f}_{4,\sR})^C $ relative to $ \mathfrak{h}_4 $ are given by
\begin{align*}
\varDelta=\left\lbrace \pm \dfrac{1}{2}a \right\rbrace.
\end{align*}
\end{theorem}
\begin{proof}
Let $ \delta_4=\tilde{A}_1(-ia) \in \mathfrak{h}_4 $ and $ \tilde{A}_2(i)-\tilde{A}_3(1) \in (\mathfrak{f}_{4,\sR})^C $. Then it follows that
\begin{align*}
[\delta_4,\tilde{A}_2(i)-\tilde{A}_3(1)]&=[\tilde{A}_1(-ia),\tilde{A}_2(i)-\tilde{A}_3(1) ]=[-\dfrac{1}{2}\tilde{A}_3(a)-\dfrac{1}{2}\tilde{A}_2(-ia)]
\\
&=\dfrac{1}{2}a[\tilde{A}_2(i)-\tilde{A}_3(1)],
\end{align*}
that is, $ [\delta_4,\tilde{A}_2(i)-\tilde{A}_3(1)]=(1/2)a[\tilde{A}_2(i)-\tilde{A}_3(1)] $. Hence we see that $ (1/2)a $ is a root of $ (\mathfrak{f}_{4,\sR})^C $ and $ \tilde{A}_2(i)-\tilde{A}_3(1) $ is a root vector associated with its root. Similarly, let $ \tilde{A}_2(1)-\tilde{A}_3(i) \in (\mathfrak{f}_{4,\sR})^C $. Then it follows that
\begin{align*}
[\delta_4,\tilde{A}_2(1)-\tilde{A}_3(i)]&=[\tilde{A}_1(-ia),\tilde{A}_2(1)-\tilde{A}_3(i) ]=[-\dfrac{1}{2}\tilde{A}_3(-ia)-\dfrac{1}{2}\tilde{A}_2(a)]
\\
&=-\dfrac{1}{2}a[\tilde{A}_2(1)-\tilde{A}_3(i)],
\end{align*}
that is, $ [\delta_4,\tilde{A}_2(1)-\tilde{A}_3(i)]=(-1/2)a[\tilde{A}_2(1)-\tilde{A}_3(i)] $. Hence we see that $ (-1/2)a $ is another root of $ (\mathfrak{f}_{4,\sR})^C $ and $ \tilde{A}_2(1)-\tilde{A}_3(i) $ is a root vector associated with its root.

Thus we have $ (\mathfrak{f}_{4,\sR})^C=\{\tilde{A}_1(-ia),\tilde{A}_2(i)-\tilde{A}_3(1),\tilde{A}_2(1)-\tilde{A}_3(i)  \}_{span} $ over $ C $, so that these roots are all.
\end{proof}


In $ (\mathfrak{f}_{4,\sR})^C $, we define an inner product $ (\delta_1, \delta_2)_4 $ by
\begin{align*}
(\delta_1, \delta_2)_4=(\delta_1,\displaystyle{\sum_{i=1}^{3}(-2)[\tilde{E}_i,\tilde{F}_{i+1}(c_{i+1})]})_4:=(-2)\displaystyle{\sum_{i=1}^{3}}(\delta_1E_i,F_{i+1}(c_{i+1})).
\end{align*}
We do the supplementary explanation of the inner product above.
In general, as in $ {\mathfrak{f}_4}^C $, we have that  $ [\tilde{A},\tilde{B}] \in (\mathfrak{f}_{4,\sR})^C $ for $ A, B \in (\mathfrak{J}_{\sR})^C $, where the $ C $-linear transformation $ \tilde{A}:(\mathfrak{J}_{\sR})^C \to (\mathfrak{J}_{\sR})^C $ is defined by $ \tilde{A}X=(1/2)[A,X] $(cf. \cite[Proposition 2.4.1]{iy0}).
Here, since $ \delta_C $ is expressed as $ \tilde{A}_1(c_1)+\tilde{A}_2(c_2)+\tilde{A}_3(c_3) $ (Theorem \ref{theorem 3.5}),
we have $ \delta_2= \displaystyle{\sum_{i=1}^{3}(-2)[\tilde{E}_i,\tilde{F}_{i+1}(c_{i+1})]} $
using $ \tilde{A}_i(c_i)=-2[\tilde{E}_i,\tilde{F}_{i+1}(c_{i+1})] $ (the indexes are considered as mod $ 3 $).
\vspace{1mm}

Here, we will determine the Killing form of $ (\mathfrak{f}_{4,\sR})^C $ needed later.

\begin{theorem}\label{theorem 3.7}
    The Killing form $ B_{4,\sR} $ of $(\mathfrak{f}_{4,\sR})^C $ is given by
    \begin{align*}
    B_{4,\sR}(\delta_1, \delta_2)
    =\dfrac{1}{4}(\delta_1, \delta_2)_4
    =\dfrac{1}{5}\tr(\delta_1 \delta_2).
    \end{align*}
\end{theorem}
\begin{proof}
The Lie algebra $ (\mathfrak{f}_{4,\sR})^C $ is simple because of $ (\mathfrak{f}_{4,\sR})^C \cong \mathfrak{so}(3,C) $ as mentioned in the proof of Theorem \ref{theorem 3.5}. By the analogues argument as that in $ {\mathfrak{f}_4}^C $, we will determine the values $ k, k' \in C $
such that
\begin{align*}
     B_{4,\sR}(\delta_1, \delta_2)
    =k (\delta_1, \delta_2)_4=k' \tr(\delta_1 \delta_2).
\end{align*}
    First, in order to determine $ k $, let $ \delta_1=\delta_2=\tilde{A}_1(1) $. Then, using the properties $ (\ast) $ above, we have
    \begin{align*}
   (\delta_1, \delta_2)_4&=(\tilde{A}_1(1), \tilde{A}_1(1) )_4=-2(\tilde{A}_1(1),[\tilde{E}_3,\tilde{F}_1(1)])
   \\
   &=-2(\tilde{A}_1(1)E_3, F_1(1))=-(F_1(1),F_1(1))
   \\
   &=-2.
    \end{align*}
    On the other hand, it follows from
    \begin{align*}
    \ad(\tilde{A}_1(1))\ad(\tilde{A}_1(1))\tilde{A}_1(1)&=[\tilde{A}_1(1),[\tilde{A}_1(1),\tilde{A}_1(1)]]=0,
    \\
    \ad(\tilde{A}_1(1))\ad(\tilde{A}_1(1))\tilde{A}_2(1)&=[\tilde{A}_1(1),[\tilde{A}_1(1),\tilde{A}_2(1)]]=[\tilde{A}_1(1), \tilde{A}_3(1)]=(-1/4)\tilde{A}_2(1),
    \\
    \ad(\tilde{A}_1(1))\ad(\tilde{A}_1(1))\tilde{A}_3(1)&=[\tilde{A}_1(1),[\tilde{A}_1(1),\tilde{A}_3(1)]]=[\tilde{A}_1(1), \tilde{A}_2(1)]=(-1/4)\tilde{A}_3(1)
    \end{align*}
    that $ B_{4,\sR}(\tilde{A}_1(1),\tilde{A}_1(1))=\tr( \ad(\tilde{A}_1(1))\ad(\tilde{A}_1(1)))=(-1/4) \times 2=-1/2 $.
    Hence we have $ k=1/4 $.

    Next, we will determine $ k' $. Similarly, it follows from
    \begin{align*}
    \tilde{A}_1(1)\tilde{A}_1(1)E_1&=0,
    \\
    \tilde{A}_1(1)\tilde{A}_1(1)E_2&=\tilde{A}_1(1)(-1/2 F_1(1))=-1/2(1,1)(E_2-E_3)=-1/2(E_2-E_3),
    \\
    \tilde{A}_1(1)\tilde{A}_1(1)E_3&=\tilde{A}_1(1)(1/2 F_1(1))=1/2(1,1)(E_2-E_3)=1/2(E_2-E_3),
    \\
    \tilde{A}_1(1)\tilde{A}_1(1)F_1(1)&=\tilde{A}_1(1)(E_2-E_3)=-F_1(1),
    \\
    \tilde{A}_1(1)\tilde{A}_1(1)F_2(1)&=\tilde{A}_1(1)(1/2F_3(1))=-1/4F_2(1),
    \\
    \tilde{A}_1(1)\tilde{A}_1(1)F_3(1)&=\tilde{A}_1(1)(-1/2F_2(1))=-1/4F_3(1).
    \end{align*}
    that $ \tr(\tilde{A}_1(1)\tilde{A}_1(1))=(-1/2)\times 2+(-1)+(-1/4)\times 2=-5/2 $. Hence we have $ k'=1/5 $.
\end{proof}
Moreover, note that by analogous argument as in $ (\mathfrak{f}_{4,\sR})^C $, any element $ \delta \in \mathfrak{f}_{4,\sR} $ is uniquely expressed by
\begin{align*}
\delta=\tilde{A}_1(r_1)+\tilde{A}_2(r_2)+\tilde{A}_3(r_3),\,\,r_i \in \R.
\end{align*}

Subsequently, we will prove the following theorem.

\begin{theorem}\label{theorem 3.8}
 In the root system $ \varDelta $ of Theorem {\rm \ref{theorem 3.6}},
    \begin{align*}
    \varPi=\{\alpha\}
    \end{align*}
    is a fundamental root system of $ (\mathfrak{f}_{4,\sR})^C $, where $ \alpha=(1/2)a$. The Dynkin diagram of $ (\mathfrak{f}_{4,\sR})^C $ is given by
\vspace{1mm}

     {\setlength{\unitlength}{1mm}
        \scalebox{0.9}
        {\begin{picture}(100,13)
            \put(0,9){}
            \put(70,10){\circle{2}} \put(69,6){$\alpha$}



            \end{picture}}
    }
\end{theorem}
\begin{proof}
It is clear that $ \varPi=\{\alpha\} $ is a fundamental root system of $ (\mathfrak{f}_{4,\sR})^C $, so that the Dynkin diagram is trivial.

\end{proof}

\section{The groups $ (E_{6,\sR})^C, E_{6,\sR} $ and the root system, the Dynkin diagram of the Lie algebra $ (\mathfrak{e}_{6,\sR})^C $ of the group $ (E_{6,\sR})^C $}

As in the previous section, we consider the following groups $ (E_{6,\sR})^C $ and $ E_{6,\sR} $ which are respectively defined by replacing $ \mathfrak{C} $ with $ \R $ in the groups $ {E_6}^C, E_6 $:
\begin{align*}
(E_{6,\sR})^C&=\left\lbrace \alpha \in \Iso_{C}((\mathfrak{J}_{\sR})^C)\relmiddle{|} \det\,\alpha X=\det\,X \right\rbrace
\\
&=\left\lbrace \alpha \in \Iso_{C}((\mathfrak{J}_{\sR})^C)\relmiddle{|} \alpha X \times \alpha Y={}^t\!\alpha^{-1}(X \times Y) \right\rbrace,
\\[1mm]
E_{6,\sR}&=\left\lbrace \alpha \in \Iso_{C}((\mathfrak{J}_{\sR})^C)\relmiddle{|}  \det\,\alpha X=\det\,X, \langle \alpha X,\alpha Y \rangle=\langle X,Y \rangle \right\rbrace
\\
&=\left\lbrace \alpha \in \Iso_{C}((\mathfrak{J}_{\sR})^C)\relmiddle{|}  \alpha X \times \alpha Y={}^t\!\alpha^{-1}(X \times Y), \langle \alpha X,\alpha Y \rangle=\langle X,Y \rangle \right\rbrace.
\end{align*}
The group $ E_{6,\sR} $ is a compact group as the closed subgroup of the unitary $ U(6)=U((\mathfrak{J}_{\sR})^C)=\{\alpha \in \Iso_C((\mathfrak{J}_{\sR})^C)\,|\,\langle \alpha X,\alpha Y\rangle=\langle X, Y \rangle\} $ and the group $ (E_{6,\sR})^C $ is the complexification of $ E_{6,\sR} $.

First, we will study the structure of the group $ (E_{6,\sR})^C $. Before that, we will make some preparations.

\noindent The Lie algebra $ (\mathfrak{e}_{6,\sR})^C $ of the group $ (E_{6,\sR})^C $ is given by
\begin{align*}
(\mathfrak{e}_{6,\sR})^C=\left\lbrace \phi \in \Hom_{C}((\mathfrak{J}_{\sR})^C) \relmiddle{|} (\phi X,X,X)=0 \right\rbrace.
\end{align*}


Then we will prove lemmas and theorem needed in the proof of Theorem \ref{theorem 4.3} below.

\begin{lemma}\label{lemma 4.1}
  The Lie algebra $ ((\mathfrak{e}_{6,\sR})^C)_E $ is isomorphic to the Lie algebra $ (\mathfrak{f}_{4,\sR})^C ${\rm :} $ ((\mathfrak{e}_{6,\sR})^C)_E \cong (\mathfrak{f}_{4,\sR})^C $, where $ ((\mathfrak{e}_{6,\sR})^C)_E:=\{\phi \in (\mathfrak{e}_{6,\sR})^C \,|\, \phi E=0 \} $.
\end{lemma}
\begin{proof}
As in the proof of $ ({E_6}^C)_E \cong {F_4}^C $ (\cite[Subsection 3.1 (p.276)]{iy4}, cf. \cite[Theorem 3.7.1]{iy0}), we can prove the isomorphism $ ((E_{6,\sR})^C)_E \cong (F_{4,\sR})^C $, where $ ((E_{6,\sR})^C)_E:=\{\alpha \in (E_{6,\sR})^C \,|\, \alpha E=E \} $.
Hence the required isomorphism follows from  the isomorphism $ ((E_{6,\sR})^C)_E \cong (F_{4,\sR})^C $.
\end{proof}

\begin{lemma}\label{lemma 4.2}
    For $ T \in (\mathfrak{J}_{\sR})^C $ such that $ \tr(T)=0 $, we have $ \tilde{T} \in  (\mathfrak{e}_{6,\sR})^C$,
    where $ \tilde{T}:(\mathfrak{J}_{\sR})^C \to (\mathfrak{J}_{\sR})^C $ is defined by $ \tilde{T}X=T \circ X $.
\end{lemma}
\begin{proof}
 For $ X \in (\mathfrak{J}_{\sR})^C $, it follows from \cite[Lemma 2.3.5]{iy0} and $ X \circ (X \times X)=(\det\,X)E $ that
 \begin{align*}
 (\tilde{T}X, X, X)&=(T \circ X, X, X)=(T \circ X, X \times X)
 \\
 &=((1/2)(TX+XT),X \times X)=((1/2)TX,X \times X)+((1/2)XT,X \times X)
 \\
 &=(T,(1/2)X(X \times X))+(T,(1/2)(X \times X)X)
 \\
 &=(T, X \circ (X \times X))=(T,(\det\,X)E)
 \\
 &=(\det\,X)(T,E)=(\det\,T)\tr(T)
 \\
 &=0.
 \end{align*}
 Hence we have $ \tilde{T} \in (\mathfrak{e}_{6,\sR})^C $.
\end{proof}

\begin{theorem}\label{theorem 4.3}
Any element $ \phi \in (\mathfrak{e}_{6,\sR})^C $ is uniquely expressed by
\begin{align*}
\phi=\delta+\tilde{T},\,\,\delta \in (\mathfrak{f}_{4,\sR})^C, T \in ((\mathfrak{J}_{\sR})^C)_0,
\end{align*}
where $ ((\mathfrak{J}_{\sR})^C)_0:=\{ X \in (\mathfrak{J}_{\sR})^C \,|\, \tr(X)=0\} $.

In particular, we have $ \dim_C((\mathfrak{e}_{6,\sR})^C )=3+(6-1)=8 $.
\end{theorem}
\begin{proof}
  For $ \phi \in (\mathfrak{e}_{6,\sR})^C $, set $ T:=\phi E $, then we have $ T \in (\mathfrak{J}_{\sR})^C $ and $ \tr(T)=0 $. Indeed, it is clear $ T \in (\mathfrak{J}_{\sR})^C $
  and it follows that $\tr(T)=(T,E,E)=(\phi E,E,E)=0 $.

  \noindent Here, set $ \delta:=\phi - \tilde{T} $, then we have $ \delta \in (\mathfrak{e}_{6,\sR})^C $ (Lemma \ref{lemma 4.2}). In addition, it follows from $ \delta E=(\phi -\tilde{T})E=\phi E-\tilde{T}E=T-T=0 $
  that $ \delta \in (\mathfrak{f}_{4,\sR})^C $ (Lemma \ref{lemma 4.1}).
   Hence we have $ \phi=\delta+\tilde{T} $.

   In order to prove the uniqueness of the expression, it is sufficient to show that
   \begin{align*}
   \delta+\tilde{T}=0, \delta \in  (\mathfrak{f}_{4,\sR})^C, T \in (\mathfrak{J}_{\sR})^C, \,\,{\text{then}}\,\, \delta=0, T=0.
   \end{align*}
   Indeed, if we apply it on $ E $, then we have $ T=0 $, so $ \delta=0 $.

Thus, from Theorem \ref{theorem 3.5}, we have $ \dim_C((\mathfrak{e}_{6,\sR})^C )=3+(6-1)=8 $.
\end{proof}

We define a space $ (EIV_{\sR})^C $ by
\begin{align*}
(EIV_{\sR})^C=\left\lbrace X \in (\mathfrak{J}_{\sR})^C \relmiddle{|}\det\,X=1 \right\rbrace.
\end{align*}

Here, as that in the proof of \cite[Proposition 3.8.2]{iy0}, we can prove that any element $ X \in (\mathfrak{J}_{\sR})^C $ can be transformed to a diagonal form
by some element $ \alpha \in (E_{6,\sR})_0 $:
\begin{align*}
\alpha X=\begin{pmatrix}
\xi_1 & 0 & 0 \\
0 & \xi_2 & 0 \\
0 & 0 & \xi_3
\end{pmatrix},\,\, \xi_i \in C,
\end{align*}
where the group $ (E_{6,\sR})_0 $ is a connected component containing the unit element $ 1 $.
Moreover, we can choose $ \alpha \in ((E_{6,\sR})^C)_0 $ so that two of $ \xi_1,\xi_2,\xi_3 $ are non-negative real numbers.
We use this fact in the proof of proposition below.

Then we will prove the following proposition.

\begin{proposition}\label{proposition 4.4}
 The homogeneous space $ (E_{6,\sR})^C/(F_{4,\sR})^C $ is homeomorphic to the space $ (EIV_{\sR})^C ${\rm :} $ (E_{6,\sR})^C/(F_{4,\sR})^C \simeq (EIV_{\sR})^C $.

 In particular, the group $ (E_{6,\sR})^C $  is connected.
\end{proposition}
\begin{proof}
  First, the group $ (E_{6,\sR})^C $  acts on  $ (EIV_{\sR})^C $, obviously. We will prove that the action of $ (E_{6,\sR})^C $ on $ (EIV_{\sR})^C $ is transitive.

  For a given $ X \in (EIV_{\sR})^C $, $ X $ can be transformed to a diagonal form by some $ \alpha \in (E_{6,\sR})_0 \subset (E_{6,\sR})^C $ as mentioned above:
  \begin{align*}
  \alpha X=\begin{pmatrix}
  \xi_1 & 0 & 0 \\
  0 & \xi_2 & 0 \\
  0 & 0 & \xi_3
  \end{pmatrix}=:X',\,\, \xi_i \in C.
  \end{align*}
  Moreover, we can choose $ \alpha \in ((E_{6,\sR})^C)_0 $ so that two of $ \xi_1,\xi_2,\xi_3 $ are non-negative real numbers.
  Hence, from $ \det\,(\alpha X)=\det \,X=1 $, that is, $ \xi_1\xi_2\xi_3=1 $, we have $ \xi_i >0, i=1,2,3 $.

  Let the element $ (s/2)(E_1-E_2)^\sim, (t/2)(E_2-E_3)^\sim \in (\mathfrak{e}_{6,\sR})^C, s,t \in \R $ (Theorem \ref{theorem 4.3}). We denote $ \exp((s/2)(E_1-E_2)^\sim), \exp((t/2)(E_2-E_3)^\sim) \in ((E_{6,\sR})^C)_0$ by $ \alpha_{12}(s), \alpha_{23}(t) $, respectively. Moreover, the explicit form of the actions of $ \alpha_{12}(s), \alpha_{23}(t) $ to $ (\mathfrak{J}_{\sR})^C $ are respectively given as follows:
  \begin{align*}
  \alpha_{12}(s)X=
  \begin{pmatrix}
  e^s\xi_1 & x_3 & e^{s/2}x_2 \\
  x_3 & e^{-s}\xi_2 & e^{-s/2}x_1 \\
  e^{s/2}x_2 & e^{-s/2}x_1 & \xi_3
  \end{pmatrix},\quad
  \alpha_{23}(t)X=
  \begin{pmatrix}
  \xi_1 & e^{t/2}x_3 & e^{-t/2}x_2 \\
  e^{t/2}x_3 & e^t\xi_2 & x_1 \\
  e^{-t/2}x_2 & x_1 & e^{-t}\xi_3
  \end{pmatrix}.
  \end{align*}
  Then, apply $ \alpha_{12}(s) $ on $ X' $, then we have $ \alpha_{12}(s)(X')=\diag(e^s\xi_1, e^{-s}\xi_2, \xi_3) $. Since we can choose $ s_0 \in \R $ such that $ e^{s_0}\xi_1=1 $, together with $ \det(\alpha_{12}(s)(X'))=1 $, we have $ \alpha_{12}(s_0)(X')=\diag(1, \xi, 1/\xi)=:X'' $.
  In addition,  apply $ \alpha_{23}(t) $ on $ X'' $, then we have $ \alpha_{23}(t)(X'')=\diag(1, e^t\xi, 1/(e^t\xi)) $. As in the case above, since we can choose $ t_0 \in \R $ such that $ e^{t_0}\xi=1 $, we have $ \alpha_{23}(t_0)(X'')=\diag(1, 1, 1)=E $. This shows the transitivity of action to $ (EIV_{\sR})^C $ by the group $ (E_{6,\sR})^C $. The isotropy subgroup of $ (E_{6,\sR})^C $ at $ E $ is the group $ (F_{4,\sR})^C $ (\cite[Subsection 3.1 (p.276)]{iy4}, cf. \cite[Theorem 3.7.1]{iy0}).

  Thus we have the required homoeomorphim
  \begin{align*}
  (E_{6,\sR})^C/(F_{4,\sR})^C \simeq (EIV_{\sR})^C.
  \end{align*}

  Moreover, since $ (F_{4,\sR})^C $ is connected because of $ (F_{4,\sR})^C  \cong SO(3,C) $ (Theorem \ref{theorem 3.2}) and $ (EIV_{\sR})^C=((E_{6,\sR})^C)_0E $ is connected,
  we have that the group $ (E_{6,\sR})^C $ is also connected.
\end{proof}

Now, we will determine the structure of the group $ (E_{6,\sR})^C $.

\begin{theorem}\label{theorem 4.5}
  The group $ (E_{6,\sR})^C $ is isomorphic to the group $ SL(3,C) ${\rm :} $ (E_{6,\sR})^C \cong SL(3,C) $.
\end{theorem}
\begin{proof}
    We define a mapping $ f_{6,C}: SL(3,C) \to (E_{6,\sR})^C $ by
    \begin{align*}
    f_{6,C}(A)X=AX{}^t\!A, \,\,X \in (\mathfrak{J}_{\sR})^C.
    \end{align*}

First, we will prove that $ f_{6,C} $ is well-defined. Since it is clear $ f_{6,C}(A) \in \Iso_{C}((\mathfrak{J}_{\sR})^C)$, we will show $ \det(f_{6,C}(A)X)=\det\,X $. Certainly, we have the following
\begin{align*}
\det(f_{6,C}(A)X)=\det(AX{}^t\!A)=(\det\,A)(\det\,X)(\det{}^t\!A)=\det\,X.
\end{align*}
Subsequently, we will prove that $ f_{6,C} $ is a homomorphism. For $ A,B \in SL(3,C) $, it follows that
\begin{align*}
f_{6,C}(AB)X=(AB)X{}^t(AB)=A(BX{}^t\!B){}^t\!A=f_{6,C}(A)f_{6,C}(B)X,\,\,X \in (\mathfrak{J}_{\sR})^C,
\end{align*}
that is, $ f_{6,C}(AB)=f_{6,C}(A)f_{6,C}(B) $.

Next, we will determine $ \Ker\,f_{6,C} $. As in the case $ \Ker\,F_{4,C} $, we have the following
\begin{align*}
\Ker\,f_{6,C}&=\left\lbrace A \in SL(3,C) \relmiddle{|} f_{6,C}(A)=E \right\rbrace
\\
&=\left\lbrace A \in SL(3,C) \relmiddle{|} AX{}^t\!A=X \,\,{\text{for all}}\,\,X \in (\mathfrak{J}_{\sR})^C \right\rbrace
\\
&=\left\lbrace  E \right\rbrace.
\end{align*}

Finally, we will prove that $ f_{6,C} $ is surjective. Since the group $ (E_{6,\sR})^C $ is connected (Proposition \ref{proposition 4.4}) and $ \Ker\,f_{6,C} $ is discrete, together with $ \dim_C((\mathfrak{e}_{6,\sR})^C)=8=\dim_C(\mathfrak{sl}(3,C)) $ (Theorem \ref{theorem 4.3}), we see that $ f_{6,C} $ is surjective.

Therefore we have the required isomorphism
\begin{align*}
(E_{6,\sR})^C \cong SL(3,C).
\end{align*}
\end{proof}

The following corollary is the direct result of Theorem \ref{theorem 4.5}.

\begin{corollary}\label{corollary 4.6}
  The Lie algebra $ (\mathfrak{e}_{6,\sR})^C $ is isomorphic to the Lie algebra $ \mathfrak{sl}(3,C) ${\rm :} $ (\mathfrak{e}_{6,\sR})^C \cong \mathfrak{sl}(3,C) $.
\end{corollary}
\begin{proof}
   We define a mapping $ f_{6,C_*}: \mathfrak{sl}(3,C) \to (\mathfrak{e}_{6,\sR})^C $ by
  \begin{align*}
  f_{6,C_*}(S)X=SX+X{}^t\!S, \,\,X \in (\mathfrak{J}_{\sR})^C.
  \end{align*}
  This mapping induces the required isomorphism.
\end{proof}

Let $ \alpha \in (E_{6,\sR})^C $. Then, as in the proof of \cite[Lemma 3.2.1]{iy7}, we can prove $ {}^t\!\alpha^{-1} \in (E_{6,\sR})^C $. Hence we define an involutive automorphism $ \lambda $ of $ (E_{6,\sR})^C $ by
\begin{align*}
  \lambda(\alpha)={}^t\!\alpha^{-1},\,\, \alpha \in (E_{6,\sR})^C.
\end{align*}

We consider the following subgroup $ ((E_{6,\sR})^C)^{\tau\lambda} $ of $ (E_{6,\sR})^C $:
\begin{align*}
((E_{6,\sR})^C)^{\tau\lambda}=\left\lbrace \alpha \in (E_{6,\sR})^C \relmiddle{|} \tau\lambda(\alpha)\tau=\alpha \right\rbrace.
\end{align*}

Then we have the following proposition.

\begin{proposition}\label{proposition 4.7}
    The group $ ((E_{6,\sR})^C)^{\tau\lambda} $ coincides with the group $ E_{6,\sR} ${\rm :} $ ((E_{6,\sR})^C)^{\tau\lambda}=E_{6,\sR}  $.
\end{proposition}
\begin{proof}
Let $ \alpha \in  E_{6,\sR} $. Then it follows from $ ({}^t\!\alpha^{-1}X, Y)=(X,\alpha^{-1}Y), X, Y \in (\mathfrak{J}_{\sR})^C $ that
\begin{align*}
({}^t\!\alpha^{-1}X, Y)&=\langle \tau{}^t\!\alpha^{-1}X, Y \rangle,
\\
(X,\alpha^{-1}Y)&=\langle \tau X , \alpha^{-1} Y \rangle=\langle \alpha \tau X , \alpha \alpha^{-1} Y \rangle=\langle \alpha \tau X , Y \rangle,
\end{align*}
that is, $ \langle \tau{}^t\!\alpha^{-1}X, Y \rangle=\langle \alpha \tau X , Y \rangle $ for all $  X, Y \in (\mathfrak{J}_{\sR})^C $.
Hence we have $ \tau{}^t\!\alpha^{-1}=\alpha \tau $, that is, $ \tau\lambda(\alpha)\tau=\alpha $. Thus we obtain $ E_{6,\sR} \subset ((E_{6,\sR})^C)^{\tau\lambda} $.

Conversely, let $ \beta \in ((E_{6,\sR})^C)^{\tau\lambda} $. For $ ({}^t\!\beta^{-1}X, Y)=(X,\beta^{-1}Y), X, Y \in (\mathfrak{J}_{\sR})^C $, set $ \beta^{-1}Y:=Y' $ and $ X:=\tau X' $, then we have $ ({}^t\!\beta^{-1}\tau X', \beta Y')=(\tau X',Y'), X', Y' \in (\mathfrak{J}_{\sR})^C  $, that is, $ \langle \tau {}^t\!\beta^{-1}\tau X', \beta Y' \rangle =\langle X', Y'\rangle $. Hence, from $ \tau\lambda(\beta)\tau=\beta $, we have $ \langle \beta X', \beta Y' \rangle =\langle X', Y'\rangle $. Thus we obtain $ ((E_{6,\sR})^C)^{\tau\lambda} \subset E_{6,\sR} $. With above, the proof of this proposition is completed.
\end{proof}

We will prove lemma needed in the proof of theorem below.

\begin{lemma}\label{lemma 4.8}
The mapping $ f_{6,C}: SL(3,C) \to (E_{6,\sR})^C $ satisfies the following relational formula{\rm :}
\begin{align*}
\tau\lambda(f_{6,C}(A))\tau=f_{6,C}(\tau{}^t\!A^{-1}).
\end{align*}
\end{lemma}
\begin{proof}
First, we will show $ {}^t\!f_{6,C}(A)=f_{6,C}({}^t\!A) $. Indeed, it follows that
\begin{align*}
({}^t\!f_{6,C}(A)X, Y)=(X,f_{6,C}(A)Y)=(X,AY{}^t\!A)=({}^t\!AXA,Y)=(f_{6,C}({}^t\!A)X,Y), \,\,X,Y \in (\mathfrak{J}_{\sR})^C.
\end{align*}
Hence we have
\begin{align*}
\tau\lambda\left(f_{6,C}(A)\right) \tau X&=\tau\left({}^t\!{f_{6,C}(A)}^{-1}\right) \tau X=\tau\left( f_{6,C}({}^t\!A^{-1})\right) \tau X
\\
&=\tau\left({}^t\!A^{-1}(\tau X)A^{-1}\right)=(\tau{}^t\!A^{-1})X{\,}^t\!(\tau{}^t\!A^{-1})
\\
&=f_{6,C}(\tau{}^t\!A^{-1})X, \,\,X \in (\mathfrak{J}_{\sR})^C,
\end{align*}
   that is, $ \tau\lambda(f_{6,C}(A))\tau=f_{6,C}(\tau{}^t\!A^{-1}) $.
\end{proof}

Now, we will determine the structure of the group $ E_{6,\sR} $.

\begin{theorem}\label{theorem 4.9}
The group $ E_{6,\sR} $ is isomorphic to the group $ SU(3) ${\rm :} $ E_{6,\sR} \cong SU(3) $.
\end{theorem}
\begin{proof}
Let the group $ SU(3)=\left\lbrace Q \in M(3,C) \relmiddle{|}(\tau {}^t\!Q)Q=E, \det\, Q=1\right\rbrace $ and the group $ ((E_{6,\sR})^C)^{\tau\lambda} $ as $ E_{6,\sR} $ (Proposition \ref{proposition 4.7}).
Then we define a mapping $ f_6: SU(3) \to ((E_{6,\sR})^C)^{\tau\lambda} $ by the restriction of the mapping $ f_{6,C} $: $ f_6(Q)=f_{6,C}(Q) $.

First, we will prove that $ f_6 $ is well-defined. However, since $ f_6 $ is the restriction of the mapping $ f_{6,C} $, we easily see $ f_6(Q) \in (E_{6,\sR})^C $. Moreover, it follows from Lemma \ref{lemma 4.8} that
\begin{align*}
\tau \lambda(f_6(Q))\tau=\tau\lambda(f_{6,C}(Q))\tau =f_{6,C}(\tau{}^t\!Q^{-1})=f_{6,C}(Q).
\end{align*}
Hence we have $ f_6(Q) \in ((E_{6,\sR})^C)^{\tau\lambda} $, so that $ f_6 $ is well-defined. Subsequently, we will prove that $ f_6 $ is a homomorphism. However, since $ f_6 $ is the restriction of the mapping $ f_{6,C} $, it is clear.

Next, we will prove that $ f_6 $ is surjective. Let $ \alpha \in ((E_{6,\sR})^C)^{\tau\lambda}  \subset (E_{6,\sR})^C$. Then there exists $ A \in SL(3,C) $ such that $ \alpha=f_{6,C}(A) $ (Theorem \ref{theorem 4.5}). Moreover, since $ \alpha $ satisfies the condition $ \tau\lambda(\alpha)\tau=\alpha $, it follows from
\begin{align*}
\tau\lambda(\alpha)\tau=\tau \lambda(f_{6,C}(A))\tau=f_{6,C}(\tau {}^t\!A^{-1})
\end{align*}
that $\tau {}^t\!A^{-1}=A $, that is, $ (\tau {}^t\!A) A=E $, and together with $ A \in SL(3,C) $, we see $ A \in SU(3) $. Thus there exists $ Q \in SU(3) $ such that $ \alpha=f_{6,C}(Q)=f_6(Q) $, so that $ f_6 $ is surjective.

Finally, we will determine $ \Ker\,f_6 $. Using the result of $ \Ker\,f_{6,C} $, we easily see $ \Ker\,f_6=\{E\} $.

Therefore, from $ E_{6,\sR}=((E_{6,\sR})^C)^{\tau\lambda} $(Proposition \ref{proposition 4.7}), we have the required isomorphism
    \begin{align*}
    E_{6,\sR} \cong SU(3).
    \end{align*}
\end{proof}

We move the determination of the root system and the Dynkin diagram of the Lie algebra $ (\mathfrak{e}_{6,\sR})^C $.

We define a Lie subalgebra $ \mathfrak{h}_6 $ of $ (\mathfrak{e}_{6,\sR})^C $ by
\begin{align*}
\mathfrak{h}_6=\left\lbrace \phi_6=\tilde{T}_0 \relmiddle{|}
\begin{array}{l}
T_0=\tau_1E_1+\tau_2E_2+\tau_3E_3 \in ({\mathfrak{J}_{\sR}}^C)_0,
\\
\quad \tau_1+\tau_2+\tau_3=0, \tau_i \in C
\end{array}
\right\rbrace.
\end{align*}
Then $ \mathfrak{h}_6 $ is a Cartan subalgebra of $ ({\mathfrak{e}_{6,\sR}})^C $.

\begin{theorem}\label{theorem 4.10}
The roots $ \varDelta $ of $ (\mathfrak{e}_{6,\sR})^C $ relative to $ \mathfrak{h}_6 $ are given by
\begin{align*}
   \varDelta=\left\lbrace
   \pm\dfrac{1}{2}(\tau_2-\tau_3),
   \pm\dfrac{1}{2}(\tau_3-\tau_1),
   \pm\dfrac{1}{2}(\tau_1-\tau_2) \right\rbrace.
\end{align*}
\end{theorem}
\begin{proof}
    Let $ \phi_6=\tilde{T}_0 \in \mathfrak{h}_6 $ and $ \tilde{A}_1(1)+\tilde{F}_1(1) \in (\mathfrak{e}_{6,\sR})^C $. Then it follows that
    \begin{align*}
    [\phi_6,\tilde{A}_1(1)+\tilde{F}_1(1) ]
    &=[\tilde{T}_0,\tilde{A}_1(1)+\tilde{F}_1(1)]=[\tau_1\tilde{E}_1+\tau_2\tilde{E}_2+\tau_3\tilde{E}_3,\tilde{A}_1(1)+\tilde{F}_1(1) ]
    \\
    &=\tau_2((1/2)\tilde{F}_1(1)+(1/2)\tilde{A}_1(1))+\tau_3((-1/2)\tilde{F}_1(1)+(-1/2)\tilde{A}_1(1))
    \\
    &=(1/2)(\tau_2-\tau_3)(\tilde{A}_1(1)+\tilde{F}_1(1)),
    \end{align*}
that is, $ [\phi_6,\tilde{A}_1(1)+\tilde{F}_1(1) ]=(1/2)(\tau_2-\tau_3)(\tilde{A}_1(1)+\tilde{F}_1(1)) $. Hence we see that $ (1/2)(\tau_2-\tau_3) $ is a root of $ (\mathfrak{e}_{6,\sR})^C $ and $ \tilde{A}_1(1)+\tilde{F}_1(1) $ is a root vector associated with its root. Similarly, we have that $ (-1/2)(\tau_2-\tau_3) $ is a root of $ (\mathfrak{e}_{6,\sR})^C $ and $ \tilde{A}_1(1)-\tilde{F}_1(1) $ is a root vector associated with its root. The remainders of roots and root vectors associated with these roots are obtained as follows:
\begin{longtable}[c]{cc}
 \hspace{0mm} $ \text{roots}  $
& \hspace{5mm} $ \text{root vectors associated with roots} $ \cr
$ (1/2)(\tau_3-\tau_1) $ & $ \tilde{A}_2(1)+\tilde{F}_2(1) $ \cr
$ (-1/2)(\tau_3-\tau_1) $ & $ \tilde{A}_2(1)-\tilde{F}_2(1) $ \cr
$ (1/2)(\tau_1-\tau_2) $ & $ \tilde{A}_3(1)+\tilde{F}_3(1) $ \cr
$ (-1/2)(\tau_1-\tau_2) $ & $ \tilde{A}_3(1)-\tilde{F}_3(1) $.
\end{longtable}
Thus, since $(\mathfrak{e}_{6,\sR})^C $ is spanned by $ \mathfrak{h}_6 $ and the root vectors associated with roots above, the roots obtained above are all.
\end{proof}

In $ (\mathfrak{e}_{6,\sR})^C $, we define an inner product $ (\phi_1, \phi_2)_6 $ by
\begin{align*}
(\phi_1, \phi_2)_6=(\delta_1+\tilde{T}_1, \delta_2+\tilde{T}_2)_6:=(\delta_1, \delta_2)_4+(T_1,T_2).
\end{align*}

Here, we will determine the Killing form of $ (\mathfrak{e}_{6,\sR})^C $.

\begin{theorem}\label{theorem 4.11}
 The Killing form $ B_{6,\sR} $ of $ (\mathfrak{e}_{6,\sR})^C $ is given by
\begin{align*}
  B_{6,\sR}(\phi_1, \phi_2)
&=\dfrac{3}{2}(\phi_1, \phi_2)_6
\\
&=\dfrac{3}{2}(\delta_1, \delta_2)_4+\dfrac{3}{2}(T_1,T_2)
\\
&=6B_{4,\sR}(\delta_1, \delta_2)+\dfrac{3}{2}(T_1,T_2)
\\
&=\dfrac{6}{5}\tr(\phi_1\phi_2).
\end{align*}
\end{theorem}
\begin{proof}
 The Lie algebra $ (\mathfrak{e}_{6,\sR})^C $ is simple because of $ (\mathfrak{e}_{6,\sR})^C \cong \mathfrak{sl}(3,C) $ (Corollary \ref{corollary 4.6}). By the analogues argument as that in $ {\mathfrak{e}_6}^C $, we will determine the values $ k, k' \in C $
 such that
 \begin{align*}
 B_{6,\sR}(\phi_1, \phi_2)
 =k (\phi_1, \phi_2)_4=k' \tr(\phi_1\phi_2).
 \end{align*}
 First, in order to determine $ k $, let $ \phi_1=\phi_2=\tilde{E}_1-\tilde{E}_2 $. Then we have
 \begin{align*}
 (\tilde{E}_1-\tilde{E}_2)_6=(E_1-E_2,E_1-E_2)=2.
 \end{align*}
 On the other hand, it follows from
 \begin{align*}
   \ad(\tilde{E}_1-\tilde{E}_2) \ad(\tilde{E}_1-\tilde{E}_2)\tilde{A}_1(1)&=[\tilde{E}_1-\tilde{E}_2,[\tilde{E}_1-\tilde{E}_2,\tilde{A}_1(1)]]=[\tilde{E}_1-\tilde{E}_2,(-1/2)\tilde{F}_1(1)]
   \\
   &=(1/4)\tilde{A}_1(1),
   \\
   \ad(\tilde{E}_1-\tilde{E}_2) \ad(\tilde{E}_1-\tilde{E}_2)\tilde{A}_2(1)&=[\tilde{E}_1-\tilde{E}_2,[\tilde{E}_1-\tilde{E}_2,\tilde{A}_2(1)]]=[\tilde{E}_1-\tilde{E}_2,(-1/2)\tilde{F}_2(1)]
   \\
   &=(1/4)\tilde{A}_2(1),
   \\
    \ad(\tilde{E}_1-\tilde{E}_2) \ad(\tilde{E}_1-\tilde{E}_2)\tilde{A}_3(1)&=[\tilde{E}_1-\tilde{E}_2,[\tilde{E}_1-\tilde{E}_2,\tilde{A}_3(1)]]=[\tilde{E}_1-\tilde{E}_2,\tilde{F}_3(1)]=\tilde{A}_3(1),
   \\
   \ad(\tilde{E}_1-\tilde{E}_2) \ad(\tilde{E}_1-\tilde{E}_2)(\tilde{E}_i-\tilde{E}_{i+1})&=[\tilde{E}_1-\tilde{E}_2,[\tilde{E}_1-\tilde{E}_2,\tilde{E}_i-\tilde{E}_{i+1}]]=0\,\,(i=1,2),
   \\
   \ad(\tilde{E}_1-\tilde{E}_2) \ad(\tilde{E}_1-\tilde{E}_2)\tilde{F}_1(1)&=[\tilde{E}_1-\tilde{E}_2,[\tilde{E}_1-\tilde{E}_2,\tilde{F}_1(1)]]=[\tilde{E}_1-\tilde{E}_2,(-1/2)\tilde{A}_1(1)]
   \\
   &=(1/4)\tilde{F}_1(1),
   \\
   \ad(\tilde{E}_1-\tilde{E}_2) \ad(\tilde{E}_1-\tilde{E}_2)\tilde{F}_2(1)&=[\tilde{E}_1-\tilde{E}_2,[\tilde{E}_1-\tilde{E}_2,\tilde{F}_2(1)]]=[\tilde{E}_1-\tilde{E}_2,(-1/2)\tilde{A}_2(1)]
   \\
   &=(1/4)\tilde{F}_2(1),
   \\
   \ad(\tilde{E}_1-\tilde{E}_2) \ad(\tilde{E}_1-\tilde{E}_2)\tilde{F}_3(1)&=[\tilde{E}_1-\tilde{E}_2,[\tilde{E}_1-\tilde{E}_2,\tilde{F}_1(1)]]=[\tilde{E}_1-\tilde{E}_2,\tilde{A}_3(1)]=\tilde{F}_3(1)
 \end{align*}
 that $ B_{6,\sR}(\tilde{E}_1-\tilde{E}_2,\tilde{E}_1-\tilde{E}_2)=\tr(\ad(\tilde{E}_1-\tilde{E}_2) \ad(\tilde{E}_1-\tilde{E}_2))=(1/4+1/4+1)\times 2=3 $. Hence we have $ k=3/2 $.

 Next, we will determine $ k' $. Similarly, it follows from
 \begin{align*}
 (\tilde{E}_1-\tilde{E}_2)(\tilde{E}_1-\tilde{E}_2)E_1&=E_1,
 \\
 (\tilde{E}_1-\tilde{E}_2)(\tilde{E}_1-\tilde{E}_2)E_2&=E_2,
 \\
 (\tilde{E}_1-\tilde{E}_2)(\tilde{E}_1-\tilde{E}_2)E_3&=0,
 \\
 (\tilde{E}_1-\tilde{E}_2)(\tilde{E}_1-\tilde{E}_2)F_1(1)&=(1/4)F_(1),
 \\
 (\tilde{E}_1-\tilde{E}_2)(\tilde{E}_1-\tilde{E}_2)F_2(1)&=(1/4)F_2(1),
 \\
 (\tilde{E}_1-\tilde{E}_2)(\tilde{E}_1-\tilde{E}_2)F_3(1)&=0
 \end{align*}
 that $ \tr((\tilde{E}_1-\tilde{E}_2)(\tilde{E}_1-\tilde{E}_2))=1+1+(1/4)\times 2=5/2 $. Hence we have $ k'=6/5 $.
\end{proof}

Subsequently, we will prove the following theorem.

\begin{theorem}\label{theorem 4.12}
In the root system $ \varDelta $ of Theorem {\rm \ref{theorem 4.9}},
\begin{align*}
    \varPi=\{\alpha_1, \alpha_2\}
\end{align*}
is a fundamental root system of $ (\mathfrak{e}_{6,\sR})^C $, where $ \alpha=(1/2)(\tau_2-\tau_3), \alpha_2=(1/2)(\tau_3-\tau_1) $. The Dynkin diagram of $ (\mathfrak{e}_{6,\sR})^C $ is given by

     {\setlength{\unitlength}{1mm}
        \scalebox{0.9}
        {\begin{picture}(100,13)
            \put(0,9){}
            \put(70,10){\circle{2}} \put(69,6){$\alpha_1$}
            \put(71,10){\line(1,0){8}}
            \put(80,10){\circle{2}} \put(79,6){$\alpha_2$}



            \end{picture}}
    }
\end{theorem}
\begin{proof}
   The remaining positive root $ (-1/2)(\tau_1-\tau_2) $ is expressed by $ \alpha_1,\alpha_2 $ as follows:
   \begin{align*}
   (-1/2)(\tau_1-\tau_2)=(1/2)(\tau_2-\tau_3)+(1/2)(\tau_3-\tau_1)=\alpha_1+\alpha_2.
   \end{align*}
   Hence we see that $ \varPi=\{\alpha_1, \alpha_2\} $ is the fundamental root system of $ (\mathfrak{e}_{6,\sR})^C $.

Let the real part $ \mathfrak{h}_{6,\sR} $ of $ \mathfrak{h}_6 $:
   \begin{align*}
   \mathfrak{h}_{6,\sR}=\left\lbrace \phi=\tilde{T}_0 \relmiddle{|}
   \begin{array}{l}
   T_0=\tau_1E_1+\tau_2E_2+\tau_3E_3 \in (({\mathfrak{J}_{\sR}})^C)_0,
   \\
   \quad \tau_1+\tau_2+\tau_3=0, \tau_i \in \R
   \end{array}
   \right\rbrace.
   \end{align*}

The killing form $ B_{6,\sR} $ of $ (\mathfrak{e}_{6,\sR})^C $ is given by $ B_{6,\sR}(\phi_1, \phi_2)=(6/5)\tr(\phi_1\phi_2)$, so is on $ \mathfrak{h}_{6,\sR} $. Hence, for $ \phi:=(\tau_1E_1+\tau_2E_2+\tau_3E_3)^\sim,\phi':=({\tau_1}'E_1+{\tau_2}'E_2+{\tau_3}'E_3)^\sim \in \mathfrak{h}_{6,\sR} $, from Theorem \ref{theorem 4.11} we easily obtain
   \begin{align*}
   B_{6,\sR}(\phi,\phi')=\dfrac{3}{2}(\tau_1{\tau_1}'+\tau_2{\tau_2}+\tau_3{\tau_3}').
   \end{align*}

Now, the canonical elements $ \phi_{\alpha_1}, \phi_{\alpha_2} \in \mathfrak{h}_{6,\sR}$ corresponding to $ \alpha_1, \alpha_2 $ are determined as follows:
   \begin{align*}
   \phi_{\alpha_1}=\dfrac{1}{3}(E_2-E_3)^\sim,
   \\
   \phi_{\alpha_2}=\dfrac{1}{3}(E_3-E_1)^\sim.
   \end{align*}
   Indeed, let $ \phi=({\tau_1}'E_1+{\tau_2}'E_2+{\tau_3}'E_3)^\sim \in \mathfrak{h}_{6,\sR}$ and set $ \phi_{\alpha_1}:=(\tau_1E_1+\tau_2E_2+\tau_3E_3)^\sim $, then from $ B_{6,\sR}(\phi_{\alpha_1},\phi)=\alpha_1(\phi) $, we have
   \begin{align*}
   \tau_1=0,\,\,\tau_2=\dfrac{1}{3},\,\,\tau_3=-\dfrac{1}{3}.
   \end{align*}
   Hence we have $ \phi_{\alpha_1}=(1/3)(E_2-E_3)^\sim $, and similarly $ \phi_{\alpha_2}=(1/3)(E_3-E_1)^\sim $.

   With above, we see that
   \begin{align*}
   (\alpha_1,\alpha_1)&=B_{6,\sR}(\phi_{\alpha_1},\phi_{\alpha_1})=\dfrac{3}{2}\left(\dfrac{1}{3}\cdot \dfrac{1}{3}+\left( -\dfrac{1}{3}\right)\cdot \left( -\dfrac{1}{3}\right) \right)=\dfrac{1}{3},
   \\
   (\alpha_1,\alpha_2)&=B_{6,\sR}(\phi_{\alpha_1},\phi_{\alpha_2})=\dfrac{3}{2}\left( -\dfrac{1}{3}\right)\cdot \dfrac{1}{3} =-\dfrac{1}{6},
   \\
   (\alpha_2,\alpha_2)&=B_{6,\sR}(\phi_{\alpha_2},\phi_{\alpha_2})=\dfrac{3}{2}\left(\left( -\dfrac{1}{3}\right)\cdot \left( -\dfrac{1}{3}\right) +\dfrac{1}{3}\cdot \dfrac{1}{3}\right)=\dfrac{1}{3}
   \end{align*}
   Hence, since we have
   \begin{align*}
   \cos\theta_{12}&=\dfrac{(\alpha_1, \alpha_2)}{\sqrt{(\alpha_1,\alpha_1)(\alpha_2,\alpha_2)}}=-\dfrac{1}{2},
   \end{align*}
   we can draw the Dynkin diagram.
\end{proof}

\section{The group $ (E_{7,\sR})^C, E_{7,\sR} $ and the root system, the Dynkin diagram of the Lie algebra $ ({\mathfrak{e}_{7,\sR}})^C $ of the group $ (E_{7,\sR})^C $}

Let $ \mathfrak{P}^C $ be the Freudenthal $ C $-vector space. As in $ (\mathfrak{J}_{\sR})^C $, we consider a $ C $-vector space $ (\mathfrak{P}_{\sR})^C $ which is defined by replacing $ \mathfrak{C} $ with $ \R $ in $ \mathfrak{P}^C $:
\begin{align*}
(\mathfrak{P}_{\sR})^C:=(\mathfrak{J}_{\sR})^C \oplus (\mathfrak{J}_{\sR})^C \oplus C \oplus C .
\end{align*}

We consider the following groups $ (E_{7,\sR})^C, E_{7,\sR} $ which are respectively  defined by replacing $ \mathfrak{P}^C $ with $ (\mathfrak{P}_{\sR})^C $ in the groups $ (E_7)^C, E_7 $:
\begin{align*}
(E_{7,\sR})^C&=\left\lbrace \alpha \in \Iso_{C}((\mathfrak{P}_{\sR})^C)\relmiddle{|} \alpha (P\times Q)\alpha^{-1}=\alpha P \times \alpha Q \right\rbrace,
\\[1mm]
E_{7,\sR}&=\left\lbrace \alpha \in \Iso_{C}((\mathfrak{P}_{\sR})^C)\relmiddle{|} \alpha (P\times Q)\alpha^{-1}=\alpha P \times \alpha Q, \langle\alpha P,\alpha Q \rangle=\langle P,Q \rangle \right\rbrace.
\end{align*}
The group $ E_{7,\sR} $ is a compact Lie group as the closed subgroup of the unitary group $ U(14)=U((\mathfrak{P}_{\sR})^C)=\{ \alpha \in \Iso_{C}((\mathfrak{P}_{\sR})^C)\,|\, \langle\alpha P,\alpha Q \rangle=\langle P,Q \rangle\}$ and the group $ (E_{7,\sR})^C $ is its complexification of $ E_{7,\sR} $.

Here, as in the Lie algebras $ {\mathfrak{e}_7}^C, \mathfrak{e}_7 $ of the groups $ {E_7}^C, E_7 $, the Lie algebras $ (\mathfrak{e}_{7,\sR})^C, \mathfrak{e}_{7,\sR} $ of the groups $ (E_{7,\sR})^C, E_{7,\sR} $ are given by:
\begin{align*}
(\mathfrak{e}_{7,\sR})^C&=\left\lbrace  \varPhi(\phi,A,B,\nu) \in \Hom_{C}((\mathfrak{P}_{\sR})^C) \relmiddle{|}\phi \in (\mathfrak{e}_{6,\sR})^C, A,B \in (\mathfrak{J}_{\sR})^C, \nu \in C \right\rbrace ,
\\[1mm]
\mathfrak{e}_{7,\sR}&=\left\lbrace  \varPhi(\phi,A,-\tau A,\nu) \in \Hom_{C}((\mathfrak{P}_{\sR})^C) \relmiddle{|} \phi \in \mathfrak{e}_{6,\sR}, A\in (\mathfrak{J}_{\sR})^C, \nu \in i\R \right\rbrace ,
\end{align*}
respectively.

In particular, we have $ \dim_C((\mathfrak{e}_{7,\sR})^C)=8+6\times 2+1=21 $ and  $ \dim(\mathfrak{e}_{7,\sR})=8+12+1=21 $.

\vspace{2mm}

Hereafter, we will determine the type of the group $ (E_{7,\sR})^C $ as Lie algebras.

\noindent Let $ \mathfrak{sp}(3,\H^C) $ be the Lie algebra of the symplectic group
$ Sp(3,\H^C) $:
\begin{align*}
   \mathfrak{sp}(3,\H^C)=\left\lbrace  D \in M(3,\H^C) \relmiddle{|} D+D^*=0 \right\rbrace.
\end{align*}

Then we have the following lemma.

\begin{lemma}\label{lemma 5.1}
   Any element $ D \in \mathfrak{sp}(3,\H^C) $ is uniquely expressed by
   \begin{align*}
   D=B+L(e_2E)+sE,\,\,B \in \mathfrak{su}(3,\C^C),L \in \mathfrak{S}(3,\C^C),s \in Ce_1,
   \end{align*}
   where $ \mathfrak{S}(3,\C^C)=\{ L \in M(3,\C^C) \,|\, {}^t\!L=L \} $ and $ e_1,e_2 $ are the basis in $ \H^C:=\{1,e_1,e_2,e_3\}_{\rm span} $.
\end{lemma}
\begin{proof}
    First, we confirm that any element $ D \in \mathfrak{sp}(3,\H^C) $ is expressed by
    \begin{align*}
    D=\begin{pmatrix}
    b_1 & d_1 & d_2 \\
    -\ov{d}_1 & b_2 & d_3 \\
    -\ov{d}_2 & -\ov{d}_3 & b_3
    \end{pmatrix},\,\, b_i \in {\rm Im}(\H^C), d_i \in \H^C.
    \end{align*}
    Here, set $ b_i:=r_ie_1+c_ie_2, d_i:=p_i+q_ie_2, r_i \in C, c_i,p_i,q_i \in \C^C, i=1,2,3 $, then $ D $ is also expressed as follows:
    \begin{align*}
    D&=\begin{pmatrix}
    r_1e_1 & p_1 & p_2 \\
    -\ov{p}_1 & r_2e_1 & p_3 \\
    -\ov{p}_2 & -\ov{p}_3 & r_3e_1
    \end{pmatrix}+
    \begin{pmatrix}
    c_1e_2 & q_1e_2 & q_2e_2 \\
    q_1e_2 & c_2e_2 & q_3e_2 \\
    q_2e_2 & q_3e_2 & c_3e_2
    \end{pmatrix}\left( =:B'+L(e_2E)\right)
    \\
    &=\begin{pmatrix}
    r_1e_1-s & p_1 & p_2 \\
    -\ov{p}_1 & r_2e_1-s & p_3 \\
    -\ov{p}_2 & -\ov{p}_3 & r_3e_1-s
    \end{pmatrix}+
    \begin{pmatrix}
    c_1 & q_1 & q_2 \\
    q_1 & c_2 & q_3 \\
    q_2 & q_3 & c_3
    \end{pmatrix}(e_2E)+sE,
    \end{align*}
    where $ s:=(1/3)(r_1+r_2+r_3)e_1 \in Ce_1 $ is uniquely determined for $ B' \in \mathfrak{u}(3,\C^C) $.

    Hence we see that $ D \in \mathfrak{sp}(3,\H^C) $ is expressed by $ D=B+L(e_2E)+sE,\,\,B \in \mathfrak{su}(3,\C^C),L \in \mathfrak{S}(3,\C^C), s \in Ce_1 $. In order to prove the uniqueness of expression, it is sufficient to show that
    \begin{align*}
    D:=B'+L(e_2E)=0,\,\,{\text{then}}\,\,B'=0, L=0.
    \end{align*}
    However, it is clear. With above, this lemma is proved.
\end{proof}

Subsequently, we will prove the following lemma.

\begin{lemma}\label{lemma 5.2}
  The Lie algebra $ \mathfrak{su}(3,\C^C) $ is isomorphic to the Lie algebra $ \mathfrak{sl}(3,C) ${\rm :} $ \mathfrak{su}(3,\C^C) \cong \mathfrak{sl}(3,C) $.
\end{lemma}
\begin{proof}
 We define a mapping $ g: \mathfrak{su}(3,\C^C) \to \mathfrak{sl}(3,C)$ by
 \begin{align*}
 g(B)=\iota B-\ov{\iota}\,{}^t\!B,
 \end{align*}
 where $ \iota=(1/2)(1+ie_1) $ with $ \iota^2=\iota, {\ov{\iota}}^2=\ov{\iota}, \iota\ov{\iota}=0, \iota+\ov{\iota}=1 $. 

First, we will prove that $ g $ is well-defined. It follows from
\begin{align*}
\ov{g(B)}&=\ov{\iota B-\ov{\iota}\,{}^t\!B}=\ov{\iota}\ov{B}-\iota{}^t\ov{B}=\ov{\iota}(-{}^t\!B)-\iota(-B)=\iota B-\ov{\iota}\,{}^t\!B=g(B),
\\[2mm]
\tr(g(B))&=\iota\tr(B)-\ov{\iota}\tr({}^t\!B)=\iota\tr(B)-\ov{\iota}\tr(B)=0
\end{align*}
that $ g(B) \in \mathfrak{sl}(3,C) $. Hence $ g $ is well-defined. Subsequently, we will prove that $ g $ is a Lie-homomorphism. It follows that
\begin{align*}
[g(B_1),g(B_2)]&=[\iota B_1-\ov{\iota}\,{}^t\!B_1,\iota B_2-\ov{\iota}\,{}^t\!B_2 ]=\iota[B_1,B_2]+\ov{\iota}[{}^t\!B_1, {}^t\!B_2]
\\
&=\iota[B_1,B_2]-\ov{\iota}\,{}^t\![B_1,B_2]
\\
&=g([B_1,B_2]).
\end{align*}
Hence $ g $ is a Lie-homomorphism. 

Next, we will prove that $ g $ is injective. In order to prove it, it is sufficient to show that $ g(B)=0 $ implies $ B=0 $. From $ g(B)=0$, that is, $\iota B=\ov{\iota}\,{}^t\!B $, apply $ \iota $ on both sides, then we have $ \iota B=0 $. Moreover, take transpose on both sides, we have $ \iota {}^t\!B=\ov{\iota}B $ and apply $ \ov{\iota} $ on this both sides, then we have $ \ov{\iota}B=0 $. Thus, since $ \iota B=0 $ and $ \ov{\iota}B=0 $, we obtain $ B=0 $, so $ g $ is injective, and together with $ \dim_C(\mathfrak{su}(3,\C^C))=8=\dim_C(\mathfrak{sl}(3,C)) $, $ g $ is surjective. 

Therefore we have the required isomorphism
\begin{align*}
\mathfrak{su}(3,\C^C) \cong \mathfrak{sl}(3,C).
\end{align*}

\end{proof}

Now, we will prove the following theorem.

\begin{theorem}\label{theorem 5.3}
    The Lie algebra $ (\mathfrak{e}_{7,\sR})^C $ is isomorphic to the Lie algebra $ \mathfrak{sp}(3,\H^C) ${\rm :}$ (\mathfrak{e}_{7,\sR})^C \cong\mathfrak{sp}(3,\H^C) $.
\end{theorem}
\begin{proof}
Let the Lie algebra $ \mathfrak{sp}(3,\H^C):=\{ B +L(e_2E)+sE \,|\, B \in \mathfrak{su}(3,\C^C), L \in \mathfrak{S}(3,\C^C), s \in Ce_1\} $ (Lemma \ref{lemma 5.1}).
Then we define a mapping $ f_{7,C_*}: \mathfrak{sp}(3,\H^C) \to (\mathfrak{e}_{7,\sR})^C$ by
    \begin{align*}
    f_{7,C_*}( B +L(e_2E)+sE)=\varPhi(f_{6,C_*}(g(B)), -(\iota L+\ov{\iota}\,\ov{L}),\ov{\iota}\,L+\iota\,\ov{L}, 3(\iota-\ov{\iota})s),
    \end{align*}
where the mapping $ g: \mathfrak{su}(3,\C^C) \to \mathfrak{sl}(3,C) $ is defined in the proof of Lemma \ref{lemma 5.1} and $ f_{6,C_*} $ is defined in the proof of Corollary \ref{corollary 4.6}.

First, we will prove that $ f_{7,C_*} $ is well-defined. From Corollary \ref{corollary 4.6} and Lemma \ref{lemma 5.2}, we have $ f_{6,C_*}(g(B)) \in (\mathfrak{e}_{6,\sR})^C $, and in addition, it is easy to verify $ -(\iota L+\ov{\iota}\,\ov{L}),\ov{\iota}\,L+\iota\,\ov{L} \in (\mathfrak{J}_{\sR})^C $ and $ 3(\iota-\ov{\iota})s \in C $. Hence  $ f_{7,C_*} $ is well-defined. Subsequently,  we will prove that $ f_{7,C_*} $ is a Lie-homomorphism. It follows that
    \begin{align*}
   (1)&\quad  [f_{7,C_*}(B), f_{7,C_*}(B')]
    \\
    &=[\varPhi(f_{6,C_*}(g(B)),0,0,0), \varPhi(f_{6,C_*}(g(B')),0,0,0)]
    \\
    &=\varPhi([f_{6,C_*}(g(B)),f_{6,C_*}(g(B')],0,0,0)
    \\
    &=\varPhi(f_{6,C_*}([g(B),g(B')]),0,0,0)
    \\
    &=\varPhi(f_{6,C_*}(g([B,B'])),0,0,0)
    \\
    &=f_{7,C_*}([B,B']),
    \\[2mm]
    (2) &\quad  [f_{7,C_*}(B), f_{7,C_*}(L(e_2E))]
    \\
    &=[\varPhi(f_{6,C_*}(g(B)),0,0,0), \varPhi(0,-(\iota L+\ov{\iota}\,\ov{L}),\ov{\iota}\,L+\iota\,\ov{L},0)]
    \\
    &=\varPhi(0,f_{6,C_*}(g(B))(-(\iota L+\ov{\iota}\,\ov{L})), -{}^t\!(f_{6,C_*}(g(B)))(\ov{\iota}\,L+\iota\,\ov{L}),0 )\,\,(-{}^t\!(f_{6,C_*}(g(B)))=f_{6,C_*}(g(\ov{B}))
    \\
    &=\varPhi(0,-(\iota(BL+L{}^t\!B)+\ov{\iota}(\ov{BL+L{}^t\!B})), \ov{\iota}(BL+L{}^t\!B)+\iota(\ov{BL+L{}^t\!B}),0)
    \\
    &=f_{7,C_*}((BL+L{}^t\!B)(e_2E))
    \\
    &=f_{7,C_*}([B,L(e_2E)]),
    \\[2mm]
    (3) &\quad  [f_{7,C_*}(B), f_{7,C_*}(sE)]
    \\
    &=[\varPhi(f_{6,C_*}(g(B)),0,0,0), \varPhi(0,0,0,3(\iota-\ov{\iota})s)]
    \\
    &=0
    \\
    &=f_{7,C_*}([B,sE)]).
    \end{align*}
    Here, in order to show $ [f_{7,C_*}(L(e_2E), f_{7,C_*}(L'(e_2E))]=f_{7,C_*}([L(e_2E), L'(e_2E)]) $ as $ (4) $, we use the following claims (i), (ii) and definition (iii):
    \begin{align*}
    &({\rm i})\,\, [L(e_2E), L'(e_2E))]=-L\ov{L'}+L'\ov{L}, \,\,{\text{then}}\,\,-L\ov{L'}+L'\ov{L}-sE \in \mathfrak{su}(3,\C^C), s:=(1/3)\tr(-L\ov{L'}+L'\ov{L}).
    \\
    &(\rm ii)\,\,\text{For $ B:=-L\ov{L'}+L'\ov{L}-sE \in \mathfrak{su}(3,\C^C), s:=(1/3)\tr(-L\ov{L'}+L'\ov{L}), g(B) $ is given by}
    \\
    & \hspace{40mm}g(B)=-\iota L\ov{L'}+\iota L'\ov{L}+\ov{\iota}\,\ov{L'}L-\ov{\iota}\,\ov{L}L'-(\iota-\ov{\iota})sE,
    \\
    &\qquad \text{where the mapping $ g $ is defined in the proof of Lemma \ref{lemma 5.2}}.
    \\
    &(\rm iii)\,\, \text{For $ A,B \in (\mathfrak{J}_{\sR})^C $, as that in \cite[Definition (p.76)]{iy0}, an element $ A \vee B \in (\mathfrak{e}_{6,\sR})^C$ is defined by }
    \\
    & \hspace{40mm} A \vee B=[\tilde{A},\tilde{B}]+\left(A \circ B-\dfrac{1}{3}(A,B)E \right)^\sim,
    \\
    &\qquad \text{where for $ A \in (\mathfrak{J}_{\sR})^C $, the $ C $-linear transformation $ \tilde{A}:(\mathfrak{J}_{\sR})^C \to (\mathfrak{J}_{\sR})^C $ is defined by $ \tilde{A}X=A \circ X $. }
    \end{align*}
    Now, we begin the proof of $ (4) $.
    \begin{align*}
    (4) &\quad  [f_{7,C_*}(L(e_2E), f_{7,C_*}(L'(e_2E))]
    \\
    &=[\varPhi(0,-(\iota L+\ov{\iota}\,\ov{L}),\ov{\iota}\,L+\iota\,\ov{L},0), \varPhi(0,-(\iota L'+\ov{\iota}\,\ov{L'}),\ov{\iota}\,L'+\iota\,\ov{L'},0)]
    \\
    &=\varPhi(-2(\iota L+\ov{\iota}\,\ov{L})\vee (\ov{\iota}\,L'+\iota\,\ov{L'})+2(\iota L'+\ov{\iota}\,\ov{L'})\vee (\ov{\iota}\,L+\iota\,\ov{L}),0,0,
    \\
    &\hspace{60mm}-(\iota L+\ov{\iota}\,\ov{L},\ov{\iota}\,L'+\iota\,\ov{L'})+(\ov{\iota}\,L+\iota\,\ov{L},\iota L'+\ov{\iota}\,\ov{L'} ),
    \end{align*}
    where using the formula of the definition $ {\rm (iii)} $, the explicit form of the action to $ (\mathfrak{J}_{\sR})^C $ of  $ -2(\iota L+\ov{\iota}\,\ov{L})\vee (\ov{\iota}\,L'+\iota\,\ov{L'})+2(\iota L'+\ov{\iota}\,\ov{L'})\vee (\ov{\iota}\,L+\iota\,\ov{L})
    $ above is expressed by
    \begin{align*}
    &\quad (-2(\iota L+\ov{\iota}\,\ov{L})\vee (\ov{\iota}\,L'+\iota\,\ov{L'})+2(\iota L'+\ov{\iota}\,\ov{L'})\vee (\ov{\iota}\,L+\iota\,\ov{L}) )X, \,\,X \in (\mathfrak{J}_{\sR})
    \\
    &=(-\iota L\ov{L'}+\iota L'\ov{L}-\ov{\iota}\,\ov{L}L'+\ov{\iota}\,\ov{L'}L)X+X(-\iota \ov{L'}L+\iota \ov{L}L'-\ov{\iota}\,L'\ov{L}+\ov{\iota}\,L\ov{L'})
    \\
    &\hspace{51mm}+((2/3)\tr(\iota L\ov{L'}+\ov{\iota}\ov{L}L')-(2/3)\tr(\iota L'\ov{L}+\ov{\iota}\ov{L'}L))X
    \\
    &=(-\iota L\ov{L'}+\iota L'\ov{L}-\ov{\iota}\,\ov{L}L'+\ov{\iota}\,\ov{L'}L)X+X(-\iota \ov{L'}L+\iota \ov{L}L'-\ov{\iota}\,L'\ov{L}+\ov{\iota}\,L\ov{L'})-2(\iota-\ov{\iota})sX
    \\
    &=(-\iota L\ov{L'}+\iota L'\ov{L}-\ov{\iota}\,\ov{L}L'+\ov{\iota}\,\ov{L'}L-(\iota-\ov{\iota})s)X+X(-\iota \ov{L'}L+\iota \ov{L}L'-\ov{\iota}\,L'\ov{L}+\ov{\iota}\,L\ov{L'}-(\iota-\ov{\iota})s)
    \\
    &=g(B)X+X\,{}^t\!g(B)\,\,(\text{claim (ii)})
    \\
    &=f_{6,C_*}(g(B))X,
    \end{align*}
note that in the calculations above, the calculations from the first line to second line are omitted because these are long, 
    and as for $ -(\iota L+\ov{\iota}\,\ov{L},\ov{\iota}\,L'+\iota\,\ov{L'})+(\ov{\iota}\,L+\iota\,\ov{L},\iota L'+\ov{\iota}\,\ov{L'} ) $,  using $ (X,Y)=\tr(X\circ Y), X, Y \in (\mathfrak{J}_{\sR})^C $, it follows that
    \begin{align*}
    -(\iota L+\ov{\iota}\,\ov{L},\ov{\iota}\,L'+\iota\,\ov{L'})+(\ov{\iota}\,L+\iota\,\ov{L},\iota L'+\ov{\iota}\,\ov{L'} )&=-\tr(\iota L\ov{L'}+\ov{\iota}\,\ov{L}L')+\tr(\ov{\iota}L\ov{L'}+\iota\ov{L}L')
    \\
    &=-\tr((\iota-\ov{\iota})L\ov{L'}-(\iota-\ov{\iota})\ov{L}L')
    \\
    &=(\iota-\ov{\iota})\tr(-L\ov{L'}+\ov{L}L')
    \\
    &=(\iota-\ov{\iota})\tr(-L\ov{L'}+L'\ov{L})
    \\
    &=3(\iota-\ov{\iota})s.
    \end{align*}
    Hence we have
    \begin{align*}
    [f_{7,C_*}(L(e_2E), f_{7,C_*}(L'(e_2E))]=\varPhi(f_{6,C_*}(g(B)),0,0,3(\iota-\ov{\iota})s).
    \end{align*}
    On the other hand, it follows that
    \begin{align*}
    f_{7,C_*}([L(e_2E),L'(e_2E)])&=f_{7,C_*}(-L\ov{L'}+L'\ov{L})
    \\
    &=f_{7,C_*}((-L\ov{L'}+L'\ov{L}-sE)+sE)\,\,(\text{claim (i)})
    \\
    &=\varPhi(f_{6,C_*}(g(B)),0,0,3(\iota-\ov{\iota})s).
    \end{align*}
    With above, we obtain
    \begin{align*}
    [f_{7,C_*}(L(e_2E), f_{7,C_*}(L'(e_2E))]=f_{7,C_*}([L(e_2E, L'(e_2E))]).
    \end{align*}
 Further, the proof of a Lie-homomorphism is continued. It follows that
    \begin{align*}
    (5)&\quad [f_{7,C_*}(L(e_2E), f_{7,C_*}(sE))]
    \\
    &=[\varPhi(0,-(\iota L+\ov{\iota}\,\ov{L}),\ov{\iota}\,L+\iota\,\ov{L},0), \varPhi(0,0,0,3(\iota-\ov{\iota})s)]
    \\
    &=\varPhi(0,(-2/3)3(\iota-\ov{\iota})s(-(\iota L+\ov{\iota}\,\ov{L})),(2/3)3(\iota-\ov{\iota})s(\ov{\iota}\,L+\iota\,\ov{L}),0)
    \\
    &=\varPhi(0,2(\iota sL-\ov{\iota}s\ov{L}),-2(\ov{\iota}sL-\iota s\ov{L}),0)
    \\
    &=\varPhi(0,2(\iota sL+\ov{\iota}\,\ov{sL}),-2(\ov{\iota}sL+\iota \ov{sL}),0)
    \\
   &=\varPhi(0,-(\iota (-2sL)+\ov{\iota}(\ov{-2sL})),\ov{\iota}(-2sL)+\iota (\ov{-2sL}),0)
   \\
   &=f_{7,C_*}((-2sL)(e_2E))
   \\
   &=f_{7,C_*}([L(e_2E),sE]),
   \\[2mm]
   (6)&\quad [f_{7,C_*}(sE), f_{7,C_*}(s'E))]
   \\
   &=[\varPhi(0,0,0,3(\iota-\ov{\iota})s),\varPhi(0,0,0,3(\iota-\ov{\iota})s') ]
   \\
   &=0
   \\
   &=f_{7,C_*}([sE,s'E]).
   \end{align*}

Consequently, the proof of homomorphism is completed.

 Next, we will prove that $ f_{7,C_*} $ is injective. Since we easily see $ \Ker\,f_{7,C_*}=\{0\}$, it is clear. Finally, we will prove that  $ f_{7,C_*} $ is surjective. Since $ f_{7,C_*} $ is injective and together with $ \dim(\mathfrak{sp}(3,\H^C))=21=\dim_C((\mathfrak{e}_{7,\sR})^C) $ (in the beginning of this section (p.18)), $ f_{7,C_*} $ is surjective.

   Therefore we have the required isomorphism
   \begin{align*}
   (\mathfrak{e}_{7,\sR})^C \cong \mathfrak{sp}(3,\H^C).
   \end{align*}
\end{proof}

We define a $ C $-linear transformation $ \lambda $ of $ (\mathfrak{P}_{\sR})^C $ by
\begin{align*}
\lambda(X,Y,\xi,\eta)=(Y,-X,\eta,-\xi).
\end{align*}
Then we have $ \lambda \in E_{7,\sR} \subset (E_{7,\sR})^C $  and $ \lambda^2=-1,\lambda^4=1 $. Hereafter, the element $ \lambda $ plays important role.

We define a space $ (\mathfrak{M}_{\sR})^C $ by
\begin{align*}
(\mathfrak{M}_{\sR})^C&=\left\lbrace P \in (\mathfrak{P}_{\sR})^C \relmiddle{|} P \times P=0,P \not=0 \right\rbrace
\\
&=\left\lbrace P=(X,Y,\xi,\eta) \in (\mathfrak{P}_{\sR})^C \relmiddle{|}\begin{array}{l}
X \vee Y=0, X \times X=\eta Y,
\\
Y \times Y=\xi X, (X,Y)=3\xi\eta,
\\
P\not=0
\end{array} \right\rbrace.
\end{align*}

Then we have the following proposition.

\begin{proposition}\label{proposition 5.4}
    For $ \alpha \in (E_{7,\sR})^C $, we have $ \alpha (\mathfrak{M}_{\sR})^C=(\mathfrak{M}_{\sR})^C $.
\end{proposition}
\begin{proof}
    We can prove this proposition as that in the proof of \cite[Proposition 4.2.2 (1)]{iy0}.
\end{proof}
Moreover, $ \alpha \in (E_{7,\sR})^C $ leaves the alternative inner product invariant: $ \{\alpha P, \alpha Q \}=\{ P,Q \}, P,Q \in (\mathfrak{P}_{\sR})^C $. This proof can be also proved as that in the proof of \cite[Proposition 4.2.2 (2)]{iy0}.
\vspace{2mm}

We consider the following subgroup $ ((E_{7,\sR})^C)_{\dot{1},\underset{\dot{}}{1}} $ of $ (E_{7,\sR})^C $:
\begin{align*}
((E_{7,\sR})^C)_{\dot{1},\text{\d{$ 1 $}}} =\left\lbrace \alpha \in (E_{7,\sR})^C \relmiddle{|}\alpha \dot{1}=\dot{1},\alpha\text{\d{$ 1 $}}=\text{\d{$ 1 $}} \right\rbrace.
\end{align*}

\begin{proposition}\label{proposition 5.5}
   The group $ ((E_{7,\sR})^C)_{\dot{1},\underset{\dot{}}{1}} $ is isomorphic to the group $ (E_{6,\sR})^C ${\rm :} $ ((E_{7,\sR})^C)_{\dot{1},\underset{\dot{}}{1}} \cong (E_{6,\sR})^C $.
\end{proposition}
\begin{proof}
  We can prove this proposition as that in the proof of \cite[Proposition 4.4.1]{iy2}.
\end{proof}

Subsequently, we consider the following subgroup $ ((E_{7,\sR})^C)_{\underset{\dot{}}{1}} $ of $ (E_{7,\sR})^C $:
\begin{align*}
((E_{7,\sR})^C)_{\text{\d{$ 1 $}}} =\left\lbrace \alpha \in (E_{7,\sR})^C \relmiddle{|}\alpha\underset{\dot{}}{1}=\text{\d{$ 1 $}} \right\rbrace.
\end{align*}

\begin{lemma}\label{lemma 5.6}
   The Lie algebra $ ((\mathfrak{e}_{7,\sR})^C)_{\underset{\dot{}}{1}}  $ of the group $ ((E_{7,\sR})^C)_{\underset{\dot{}}{1}} $ is given by
   \begin{align*}
   ((\mathfrak{e}_{7,\sR})^C)_{\underset{\dot{}}{1}}
   &=\left\lbrace \varPhi(\phi,A,B,\nu) \in (\mathfrak{e}_{7,\sR})^C \relmiddle{|}  \varPhi(\phi,A,B,\nu)\underset{\dot{}}{1}=0 \right\rbrace
   \\
   &=\left\lbrace \varPhi(\phi,0,B,0) \relmiddle{|} \phi \in (\mathfrak{e}_{6,\sR})^C, B \in (\mathfrak{J}_{\sR})^C \right\rbrace.
   \end{align*}

   In particular, we have $ \dim_C(((\mathfrak{e}_{7,\sR})^C)_{\underset{\dot{}}{1}})=8+6=14 $.
\end{lemma}
\begin{proof}
    Let $ \varPhi(\phi,A,B,\nu) \in (\mathfrak{e}_{7,\sR})^C $. Then, since it follows that
    \begin{align*}
    \varPhi(\phi,A,B,\nu)\underset{\dot{}}{1}=(A,0,0,-\nu),
    \end{align*}
    we have $ A=0 $ and $ \nu=0 $. Thus we have the required result.
\end{proof}

Then we have the following proposition.

\begin{proposition}\label{proposition 5.7}
    The group $ ((E_{7,\sR})^C)_{\underset{\dot{}}{1}} $ is the semi-direct product of groups $ \exp(\varPhi(0,0,(\mathfrak{J}_{\sR})^C,0)) $ and $ (E_{6,\sR})^C ${\rm :} $ ((E_{7,\sR})^C)_{\underset{\dot{}}{1}}= \exp(\varPhi(0,0,(\mathfrak{J}_{\sR})^C,0)) \rtimes (E_{6,\sR})^C $, where the group $ \exp(\varPhi(0,0,(\mathfrak{J}_{\sR})^C,0)):=\{\exp(\varPhi(0,0,B,0))\,|\,\allowbreak B \in (\mathfrak{J}_{\sR})^C\} $.

    In particular,the group $ ((E_{7,\sR})^C)_{\underset{\dot{}}{1}} $ is connected.
\end{proposition}
\begin{proof}
    Let $ \alpha \in ((E_{7,\sR})^C)_{\underset{\dot{}}{1}} $ and set $ \alpha \dot{1}:=(X,Y,\xi,\eta) $. Then it follows from
    \begin{align*}
    1=\{\dot{1},\underset{\dot{}}{1}\}=\{\alpha \dot{1},\alpha \underset{\dot{}}{1}\}=\{(X,Y,\xi,\eta), (0,0,0,1)\}=\xi
    \end{align*}
    that $ \xi=1 $, that is, $ \alpha \dot{1}=(X,Y,1,\eta) $, and moreover $ \alpha \dot{1} $ is of form $ (Y \times Y,Y,1,(1/3)(Y \times Y,Y)) $ from $ \dot{1} \in (\mathfrak{M}_{\sR})^C $:
    \begin{align*}
    \alpha \dot{1}=(Y \times Y,Y,1,\dfrac{1}{3}(Y \times Y,Y)).
    \end{align*}

    Here, set $ \delta:=\varPhi(0,0,B,0) $. Then it follows that
    \begin{align*}
    \delta \dot{1}
    &=\exp(\varPhi(0,0,B,0)) \dot{1}=\left( \displaystyle{\sum\limits_{n=0}^{\infty}\dfrac{1}{n!}(\varPhi(0,0,B,0))^n}\right) \dot{1}
    =(B \times B,B,1,\dfrac{1}{3}(B\times B,B)),
    \end{align*}
    and moreover it is easy to verify that $ \delta \underset{\dot{}}{1}=\underset{\dot{}}{1} $.

    Hence we have
    \begin{align*}
    \delta^{-1}\alpha \dot{1}&=\delta^{-1}(Y \times Y,Y,1,\dfrac{1}{3}(Y \times Y,Y))=\dot{1},
    \\
    \delta^{-1}\alpha \underset{\dot{}}{1}&=\delta^{-1} \underset{\dot{}}{1}=\underset{\dot{}}{1}.
    \end{align*}
    Hence, since we have $ \delta^{-1}\alpha \in (E_{6,\sR})^C $ from Proposition \ref{proposition 5.5}, the group $ ((E_{7,\sR})^C)_{\underset{\dot{}}{1}} $ is of the form
    \begin{align*}
    ((E_{7,\sR})^C)_{\underset{\dot{}}{1}}= \exp(\varPhi(0,0,(\mathfrak{J}_{\sR})^C,0)) (E_{6,\sR})^C.
    \end{align*}

    Next, we will show that $ ((E_{7,\sR})^C)_{\underset{\dot{}}{1}} $ is the semi-direct product of groups $ \exp(\varPhi(0,0,(\mathfrak{J}_{\sR})^C,0)) $ and $ (E_{6,\sR})^C $. For $ \beta \in (E_{6,\sR})^C $, it is to verify that
    \begin{align*}
    \beta \exp(\varPhi(0,0,B,0))\beta^{-1}=\exp(\varPhi(0,0,{}^t\!\beta^{-1} B,0)).
    \end{align*}
    Indeed, for $ (X,Y,\xi,\eta) \in (\mathfrak{P}_{\sR})^C $, first it follows that
    \begin{align*}
    \beta \varPhi(0,0,B,0)\beta^{-1}(X,Y,\xi,\eta)
    &=\beta \varPhi(0,0,B,0)(\beta^{-1}X,{}^t\!\beta Y,\xi,\eta)
    \\
    &=\beta(2B\times {}^t\!\beta Y,\xi B, 0, (B,\beta^{-1}X))
    \\
    &=(\beta(2B\times {}^t\!\beta Y),{}^t\!\beta^{-1}\xi B, 0, (B,\beta^{-1}X))
    \\
    &=(2{}^t\!\beta^{-1}B\times Y,\xi\,{}^t\!\beta^{-1}B, 0, ({}^t\!\beta^{-1}B,X))
    \\
    &=\varPhi(0,0,{}^t\!\beta^{-1}B,0)(X,Y,\xi,\eta),
    \end{align*}
    that is, $  \beta \varPhi(0,0,B,0)\beta^{-1}=\varPhi(0,0,{}^t\!\beta^{-1}B,0) $. Using this formula, we have the following
    \begin{align*}
    \beta \exp(\varPhi(0,0,B,0))\beta^{-1}&=\beta\left(\displaystyle{\sum\limits_{n=0}^{\infty}\dfrac{1}{n!}(\varPhi(0,0,B,0))^n} \right)\beta^{-1}
    \\
    &=\displaystyle{\sum\limits_{n=0}^{\infty}\dfrac{1}{n!}(\beta\varPhi(0,0,B,0)\beta^{-1})^n}
    \\
    &=\displaystyle{\sum\limits_{n=0}^{\infty}\dfrac{1}{n!}(\varPhi(0,0,{}^t\!\beta^{-1}B,0))^n}
    \\
    &=\exp(\varPhi(0,0,{}^t\!\beta^{-1} B,0)).
    \end{align*}
    Hence, for $ \alpha \in ((E_{7,\sR})^C)_{\underset{\dot{}}{1}} $, it follows from $ \alpha=\delta\beta \in \exp(\varPhi(0,0,(\mathfrak{J}_{\sR})^C,0)) (E_{6,\sR})^C=((E_{7,\sR})^C)_{\underset{\dot{}}{1}} $ that
    \begin{align*}
    \alpha\exp(\varPhi(0,0,B,0))\alpha^{-1}&=\delta\beta\exp(\varPhi(0,0,B,0))\beta^{-1}\delta^{-1}
    \\
    &=\delta\exp(\varPhi(0,0,{}^t\!\beta^{-1} B,0))\delta^{-1} \in  \exp(\varPhi(0,0,(\mathfrak{J}_{\sR})^C,0)).
    \end{align*}
    Thus the group $ \exp(\varPhi(0,0,(\mathfrak{J}_{\sR})^C,0)) $ is a normal subgroup of $ ((E_{7,\sR})^C)_{\underset{\dot{}}{1}} $.
In addition, we have a split exact sequence
    \begin{align*}
    1 \longrightarrow \exp(\varPhi(0,0,(\mathfrak{J}_{\sR})^C,0)) \overset{j}{\longrightarrow} ((E_{7,\sR})^C)_{\underset{\dot{}}{1}} \overset{\overset{p}{\scalebox{1.0}{$\longrightarrow$}}}{\underset{s}{\longleftarrow}} (E_{6,\sR})^C \longrightarrow 1.
    \end{align*}
Indeed, first we define a mapping $ j $ by $ j(\delta)=\delta $. Then it is clear that $ j $ is a injective homomorphism. Subsequently, we define a mapping $ p $ by $ p(\alpha):=p(\delta\beta)=\beta, \delta \in \exp(\varPhi(0,0,(\mathfrak{J}_{\sR})^C,0)), \beta \in (E_{6,\sR})^C $. Then, since  it follows that
    \begin{align*}
    p(\alpha_1\alpha_2)&=p((\delta_1\beta_1)(\delta_2\beta_2))=p(\delta_1(\beta_1\delta_2{\beta_1}^{-1})\beta_1\beta_2)\,\,(\beta_1\delta_2{\beta_1}^{-1}=:{\delta_2}' \in \exp(\varPhi(0,0,(\mathfrak{J}_{\sR})^C,0)))
    \\
    &=p((\delta_1{\delta_2}')(\beta_1\beta_2))=\beta_1\beta_2
    \\
    &=p(\alpha_1)p(\alpha_2),
    \end{align*}
$ p $ is a homomorphism. Let $ \beta \in (E_{6,\sR})^C $, then there exists $ \alpha \in ((E_{7,\sR})^C)_{\underset{\dot{}}{1}} $ such that $ \alpha=\delta\beta $ for $ \delta \in \exp(\varPhi(0,0,(\mathfrak{J}_{\sR})^C,0)) $. This implies that $ p $ is surjective. Finally, we define a mapping $ s $ by $ s(\beta):=s(\delta^{-1}\alpha)=\alpha $. Then, as in the mapping $ p $, we easily see that $ s $ is a homomorphism and $ ps=1 $. With above, the short sequence is a split exact sequence.

Thus the group $ ((E_{7,\sR})^C)_{\underset{\dot{}}{1}} $ is the semi-direct product of groups $ \exp(\varPhi(0,0,(\mathfrak{J}_{\sR})^C,0)) $ and $ (E_{6,\sR})^C $:
    \begin{align*}
    ((E_{7,\sR})^C)_{\underset{\dot{}}{1}}= \exp(\varPhi(0,0,(\mathfrak{J}_{\sR})^C,0)) \rtimes (E_{6,\sR})^C.
    \end{align*}

Finally, the connectedness of the group $  ((E_{7,\sR})^C)_{\underset{\dot{}}{1}} $ follows from the connectedness of the groups $ \exp(\varPhi(0,0,(\mathfrak{J}_{\sR})^C,0)) $ and $ (E_{6,\sR})^C $ are connected (Proposition \ref{proposition 4.4}).
\end{proof}

Now, we will prove the connectedness of the group  $ (E_{7,\sR})^C $.

\begin{theorem}\label{theorem 5.8}
The homogeneous space $ (E_{7,\sR})^C/((E_{7,\sR})^C)_{\underset{\dot{}}{1}} $ is homeomorphic to the space $ (\mathfrak{M}_{\sR})^C ${\rm :} $ (E_{7,\sR})^C \allowbreak /((E_{7,\sR})^C)_{\underset{\dot{}}{1}} \simeq (\mathfrak{M}_{\sR})^C $.

  In particular, the group $ (E_{7,\sR})^C $ is connected.
\end{theorem}
\begin{proof}
Obviously, the group $ (E_{7,\sR})^C $ acts on the space $ (\mathfrak{M}_{\sR})^C $, so we will prove that this action is transitive. In order to prove the transitivity, it is sufficient to show that any element $ P \in (\mathfrak{M}_{\sR})^C $ can be transformed to $ \underset{\dot{}}{1} \in (\mathfrak{M}_{\sR})^C $ by some $ \alpha \in (E_{7,\sR})^C $.

    Case (i) where $ P=(X,Y,\xi,\eta), \eta\not=0 $.

From the definition of $ (\mathfrak{M}_{\sR})^C $, $ P $ is of the form
    \begin{align*}
    P=(X, \dfrac{1}{\eta} (X \times X), \dfrac{1}{3\eta}(X,X \times X),\eta).
    \end{align*}
Well, for $ \varPhi(0,A,0,\nu) \in (\mathfrak{e}_{7,\sR})^C, \nu\not=0 $, we will compute $ \varPhi(0,A,0,\nu)^n\underset{\dot{}}{1} $:
    \begin{align*}
    \varPhi(0,A,0,\nu)^n\underset{\dot{}}{1}
    =\begin{pmatrix}
    \left(\dfrac{3}{\nu} \right)\left(\left( -\dfrac{1}{2}\right) (-\nu)^n+\dfrac{1}{2} \left( -\dfrac{\nu}{3}\right)^n  \right)A
    \\[2mm]
     \left(\dfrac{3}{\nu} \right)^2\left( \left( -\dfrac{1}{2}\right)  \left( -\dfrac{\nu}{3}\right)^n+\dfrac{1}{4}(-\nu)^n+\dfrac{1}{4}\left( \dfrac{\nu}{3}\right)^n  \right)(A \times A)
     \\[2mm]
     \left(\dfrac{3}{\nu} \right)^3\left(\dfrac{1}{24}\nu^n-\dfrac{1}{24}(-\nu)^n+\dfrac{1}{16}\left(-\dfrac{\nu}{3} \right)^n-\dfrac{1}{16}\left(\dfrac{\nu}{3} \right)^n  \right)(A,A \times A)
     \\[2mm]
     (-\nu)^n
    \end{pmatrix}.
    \end{align*}
    Then, using the formula above,  we have the following result
    \begin{align*}
    (\exp(\varPhi(0,A,0,\nu))\underset{\dot{}}{1}&=\left(\displaystyle{\sum\limits_{n=0}^{\infty}\dfrac{1}{n!}\varPhi(0,A,0,\nu)^n} \right)\underset{\dot{}}{1}
    \\[2mm]
    &=\begin{pmatrix}
    \left(\dfrac{3}{\nu} \right)\left( -\dfrac{1}{2}e^{-\nu}+\dfrac{1}{2}e^{-\frac{\nu}{3}} \right)A
    \\[2mm]
     \left(\dfrac{3}{\nu} \right)^2\left( \left( -\dfrac{1}{2}\right) e^{-\frac{\nu}{3}}+\dfrac{1}{4}e^{-\nu}+\dfrac{1}{4}e^{\frac{\nu}{3}}  \right)(A \times A)
     \\[2mm]
     \left(\dfrac{3}{\nu} \right)^3\left(\dfrac{1}{24}e^{\nu}-\dfrac{1}{24}e^{-\nu}+\dfrac{1}{16}e^{-\frac{\nu}{3}}-\dfrac{1}{16}e^{\frac{\nu}{3}}  \right)(A,A \times A)
     \\[2mm]
     e^{-\nu}
    \end{pmatrix}.
    \end{align*}
Hence, for a given $ P=(X, (1/\eta) (X \times X), (1/3\eta)(X,X \times X),\eta) \in (\mathfrak{M}_{\sR})^C $, we choose $ \nu \in C $ such that $ \eta=e^{-\nu} $ and construct $ A:=((3/\nu)((1/2)e^{-\nu})+(1/2)e^{-\nu/3})^{-1}X \in (\mathfrak{J}_{\sR})^C $ using $ X $ and $ \nu $.
Then we have
    \begin{align*}
    (\exp(\varPhi(0,A,0,\nu))\underset{\dot{}}{1}=(X, \dfrac{1}{\eta} (X \times X), \dfrac{1}{3\eta}(X,X \times X),\eta).
    \end{align*}
    Thus there exists $ (\exp(\varPhi(0,A,0,\nu)) \in ((E_{7,\sR})^C)_0 $ such that $ (\exp(\varPhi(0,A,0,\nu))^{-1}(X,(1/\eta) (X \times X), (1/3\eta)\allowbreak(X, X \times X),\eta)=\underset{\dot{}}{1} $.

    \noindent For $ \varPhi(0,A,0,\nu) \in (\mathfrak{e}_{7,\sR})^C $, if $ \nu=0 $, we easily see
    \begin{align*}
    (\exp(\varPhi(0,A,0,0))\underset{\dot{}}{1}=(A,A \times A, \dfrac{1}{3}(A,A \times A),1).
    \end{align*}
    Hence, for a given $ P=(X, (1/\eta) (X \times X), (1/3\eta)(X,X \times X),\eta) \in (\mathfrak{M}_{\sR})^C $, we choose $ X $ as $ A $, then we have
    \begin{align*}
    (\exp(\varPhi(0,X,0,0))\underset{\dot{}}{1}=(X, \dfrac{1}{\eta} (X \times X), \dfrac{1}{3\eta}(X,X \times X),\eta).
    \end{align*}
Thus, as in the case above, there exists $ (\exp(\varPhi(0,A,0,0)) \in ((E_{7,\sR})^C)_0 $ such that $ (\exp(\varPhi(0,X,0,0))^{-1}(X,\allowbreak  (1/\eta) (X \times X), (1/3\eta)(X,X \times X),\eta)=\underset{\dot{}}{1} $.
    \vspace{1mm}

    Case (ii) where $ P=(X,Y,\xi,0), \xi\not=0 $.

Let the element $ \lambda \in (E_{7,\sR})^C $ be the $ C $-linear transformation of $ (\mathfrak{P}_{\sR})^C $. Then, since we have $ \lambda(X,Y,\xi,0)\allowbreak =(Y,-X,0,\xi), \xi\not=0 $, this case is reduced to (i).
    \vspace{1mm}

     Case (iii) where $ P=(X,Y,0,0), Y\not=0 $.

Note that $ Y $ satisfies the condition $ Y \times Y=0 $.
For a given $ P=(X,Y,0,0), Y\not=0 $, we choose $ \tau Y $, then we have
     \begin{align*}
     \exp(\varPhi(0,\tau Y,0,0))(X,Y,0,0)
     &=(X,Y+2\tau Y \times X, (\tau Y,Y)+(\tau Y,\tau Y \times X),0)
     \\
     &=(X,Y+2\tau Y \times X, (\tau Y,Y)+(X,\tau Y \times \tau Y ),0)
     \\
     &=(X,Y+2\tau Y \times X, (\tau Y,Y)+(X,\tau (Y \times Y) ),0)
     \\
     &=(X,Y+2\tau Y \times X, (\tau Y,Y),0).
     \end{align*}
Here, we confirm that $ (\tau Y,Y)=0 $ if and only if $ Y=0 $, so that we see  $ (\tau Y,Y) \not=0 $ from $ Y \not=0 $. Hence this case is reduced to (ii).
     \vspace{1mm}

     Case (iv) where $ P=(X,0,0,0), X\not=0 $.

As in the case (ii), let the element $ \lambda \in (E_{7,\sR})^C $ again. Then, since we have $ \lambda(X,0,0,0)=(0,-X,0,0), X \allowbreak \not=0 $, this case is reduced to (iii).

With above, the proof of transitive is completed.

Thus it follows from the transitivity that $ (\mathfrak{M}_{\sR})^C=((E_{7,\sR})^C)_0 \text{\d{$ 1 $}} $, so that $ (\mathfrak{M}_{\sR})^C $ is connected.
The group $ (E_{7,\sR})^C $ acts transitively on $ (\mathfrak{M}_{\sR})^C $ and the isotropy subgroup of $ (E_{7,\sR})^C $ at $ \text{\d{$ 1 $}} $ is $ ((E_{7,\sR})^C)_{\underset{\dot{}}{1}} $.

Therefore we have the homeomorphism
\begin{align*}
     (E_{7,\sR})^C/((E_{7,\sR})^C)_{\underset{\dot{}}{1}} \simeq (\mathfrak{M}_{\sR})^C.
\end{align*}

Finally, the connectedness of $ (E_{7,\sR})^C $ follows from the connectedness of $ ((E_{7,\sR})^C)_{\underset{\dot{}}{1}} $ and $ (\mathfrak{M}_{\sR})^C $ (Proposition \ref{proposition 5.7}).
\end{proof}
\vspace{2mm}

We will determine the center of the group $ (E_{7,\sR})^C $. Before that, for $ \theta \in C^*:=\{ \theta \in C\,|\, \theta \not=0\} $, we define a $ C $-linear transformation $ \phi(\theta) $ of $ (\mathfrak{M}_{\sR})^C $ by
\begin{align*}
\phi(\theta)(X,Y,\xi,\eta)=(\theta^{-1}X, \theta Y, \theta^3 \xi,\theta^{-3} \eta).
\end{align*}
Then we have $ \phi(\theta) \in (E_{7,\sR})^C $.

\begin{theorem}\label{theorem 5.9}
   The center $ z((E_{7,\sR})^C) $ of the group $ (E_{7,\sR})^C $ is the cyclic group of order two:
   \begin{align*}
   z((E_{7,\sR})^C) = \{ 1, -1 \} \cong \Z_2.
   \end{align*}
\end{theorem}
\begin{proof}
 Let $\alpha \in z((E_{7,\sR})^C)$. From the commutativity with $\beta \in (E_{6,\sR})^C \subset (E_{7,\sR})^C$, we have $\beta\alpha\dot{1} = \alpha\beta\dot{1} = \alpha\dot{1}$. Set $\alpha\dot{1} = (X, Y, \xi, \eta) \in (\mathfrak{M}_{\sR})^C$, then from $(\beta X, {}^t\!\beta^{-1} Y, \xi, \eta) = (X, Y, \xi, \eta)$, we have
 \begin{align*}
 \beta X = X, \;\; {}^t\!\beta^{-1} Y = Y  \quad \text{for all} \quad \beta \in (E_{6,\sR})^C.
 \end{align*}
 Hence we have $X = Y = 0$. Indeed, note that $ z((E_{6,\sR})^C)=z(SL(3,C))(=\{1, \omega 1, \omega^2 1 \}) $ follows from Theorem \ref{theorem 4.5}. By letting $ \omega 1 \in z((E_{6,\sR})^C) , \omega \in C, \omega^3=1, \omega \not=1$ as $ \beta $, we easily obtain $ X=Y=0 $.
 Thus $\alpha\dot{1}$ is of the form
 \begin{align*}
 \alpha\dot{1} = (0, 0, \xi, \eta).
 \end{align*}
 Since $\alpha\dot{1} \in (\mathfrak{M}_{\sR})^C$, we have $\xi\eta = 0$, that is, $ \xi=0 $ or $ \eta=0 $. Here, suppose $\xi = 0$, then we have $\alpha\dot{1} = (0, 0, 0, \eta), \eta \neq 0$. Also from the commutativity with $\phi(\theta) \in (E_{7,\sR})^C$, we have
   \begin{align*}
   (0, 0, 0, \theta^{-3}\eta) &= \phi(\theta)(0, 0, 0, \eta) = \phi(\theta)\alpha\dot{1}=\alpha\phi(\theta)\dot{1}
   \\
   &= \alpha(0, 0, \theta^3, 0)=\theta^3\alpha\dot{1}
   \\
   &= (0, 0, 0, \theta^3\eta),
   \end{align*}
 that is, $\theta^{-3}\eta = \theta^3\eta$ for all $\theta \in C^*$. This is a contradiction. Hence, since we see $\xi \neq 0$, we have $\eta = 0$, that is, $\alpha\dot{1} = (0, 0, \xi, 0)$.

 \noindent Subsequently, set $ \alpha\text{\d{$ 1 $}}=(Z,W,\rho,\zeta) $, then by an argument similar to that in the case where $ \alpha \dot{1} $ above, we have $\alpha\underset{\dot{}}{1}= (0, 0, 0, \zeta)$. In addition, from $\{ \alpha\dot{1}, \alpha\underset{\dot{}}{1} \} = 1$, we have $\xi\zeta = 1$. Hence we have
 \begin{align*}
 \alpha\dot{1} = (0, 0, \xi, 0), \quad \alpha\underset{\dot{}}{1} = (0, 0, 0, \xi^{-1}).
 \end{align*}
 Moreover, from the commutativity with $\lambda \in (E_{7,\sR})^C$, we have
 \begin{align*}
 (0, 0, 0, -\xi)
 & = \lambda(0, 0, \xi, 0) = \lambda\alpha\dot{1} = \alpha\lambda\dot{1}
 \\
 & = \alpha(0, 0, 0, -1)
 \\
 &= (0, 0, 0, - \xi^{-1}),
 \end{align*}
 that is, $ \xi = \xi^{-1} $. Hence we have $ \xi=1 $ or $ \xi=-1 $.

 In the case where $\xi =1$. From $ \alpha\dot{1}=\dot{1} $ and $ \alpha\underset{\dot{}}{1}=\underset{\dot{}}{1} $, we have $\alpha \in (E_{6,\sR})^C$ (Proposition \ref{proposition 5.5}), so that we see $\alpha \in z((E_{6,\sR})^C) = \{1, \omega 1, \omega^2 1 \}  $ as a necessary condition. Then $ \alpha $ is expressed by
 \begin{align*}
 \alpha = \begin{pmatrix}
 \omega'1 & 0 & 0 & 0 \cr
 0 & {\omega'}^{-1}1 & 0 & 0 \cr
 0 & 0 & 1 & 0 \cr
 0 & 0 & 0 & 1
 \end{pmatrix}, \,\, \omega': = 1, \omega \;\text{or} \; \omega^2
 \end{align*}
 as the element of $ (E_{7,\sR})^C $.
 Again from the commutativity with $\lambda$, we have
 \begin{align*}
 (0, \omega'X, 0, 0) &= - \lambda(\omega'X, 0, 0, 0) = - \lambda\alpha(X, 0, 0, 0)
 \\
 &= - \alpha\lambda(X, 0, 0, 0) = \alpha(0, X, 0, 0)
 \\
 &= (0, {\omega'}^{-1}X, 0, 0),
 \end{align*}
 for all $X \in (\mathfrak{J}_{\sR})^C$, that is, $\omega' = {\omega'}^{-1}$. This implies  $\omega' = 1$. Hence we have $\alpha = 1$.

 In the case where $\xi = - 1$, we have $- \alpha \in z((E_{6,\sR})^C)$, so that we have $- \alpha = 1$ by the similar argument as above.

 With above, we obtain $ z((E_{7,\sR})^C) \subset \{ 1, -1 \} $, and vice versa.

 Therefore we have the required result
 \begin{align*}
 z((E_{7,\sR})^C) = \{ 1, -1 \} \cong \Z_2
 \end{align*}
\end{proof}

Now, we will determine the structure of the group $ (E_{7,\sR})^C $.

\begin{theorem}\label{theorem 5.10}
  The group $ (E_{7,\sR})^C $ is isomorphic to the group $ Sp(3,\H^C) ${\rm :} $ (E_{7,\sR})^C \cong Sp(3,\H^C) $.
\end{theorem}
\begin{proof}
Since the group $ (E_{7,\sR})^C $ is connected (Theorem \ref{theorem 5.8}) and its associated Lie algebra $ (\mathfrak{e}_{7,\sR})^C $ is isomorphic to the Lie algebra $ \mathfrak{sp}(3,\H^C) $ (Theorem \ref{theorem 5.3}), the group $ (E_{7,\sR})^C $ is isomorphic to either of the following groups
  \begin{align*}
  Sp(3,\H^C), \quad Sp(3,\H^C)/\Z_2.
  \end{align*}
Therefore, from Theorem \ref{theorem 5.9}, $ (E_{7,\sR})^C $ have to be isomorphic to $ Sp(3,\H^C) $:
  \begin{align*}
  (E_{7,\sR})^C \cong Sp(3,\H^C).
  \end{align*}
\end{proof}



Now, in order to determine the structure of the group $ E_{7,\sR} $, we make some preparations below.

First, we consider the following subgroup $ ((E_{7,\sR})^C)^{\tau\lambda} $ of $ (E_{7,\sR})^C $:
\begin{align*}
((E_{7,\sR})^C)^{\tau\lambda}:=\left\lbrace \alpha \in (E_{7,\sR})^C \relmiddle{|} \tau\lambda\alpha\lambda^{-1}\tau=\alpha \right\rbrace.
\end{align*}

Then we have the following proposition.

\begin{proposition}\label{proposition 5.11}
    The group $ ((E_{7,\sR})^C)^{\tau\lambda} $ coincides with the group $ E_{7,\sR} ${\rm :} $ ((E_{7,\sR})^C)^{\tau\lambda}=E_{7,\sR}  $.
\end{proposition}
\begin{proof}
    Let $ \alpha \in ((E_{7,\sR})^C)^{\tau\lambda} $. It follows from $ \langle P,Q\rangle =\{\tau \lambda P,Q \}, P,Q \in (\mathfrak{P}_{\sR})^C $ that
    \begin{align*}
    \langle \alpha P, \alpha Q \rangle=\{\tau\lambda \alpha P,\alpha Q \}=\{\alpha \tau\lambda P,\alpha Q\}=\{\tau\lambda P,Q\}=\langle P,Q \rangle.
    \end{align*}
    Hence we have $ \alpha \in E_{7,\sR}$.

    Conversely, let $ \beta \in E_{7,\sR} $. It follows from $ (P,Q)=\{\lambda P,Q\} $ that
    \begin{align*}
    \langle \beta P, \beta Q \rangle&=\{\tau\lambda \beta P,\beta Q\}=\{\lambda\tau \beta P,\beta Q\}=(\tau\beta P,\beta Q),
    \\[1mm]
    \langle P,Q \rangle&=\{\tau\lambda P,Q\}=\{\beta\tau\lambda P,\beta Q\}=\{\lambda\lambda^{-1}\beta\tau\lambda P,\beta Q\}=(\lambda^{-1}\beta\tau\lambda P,\beta Q).
    \end{align*}
    Hence we have $ \tau\beta=\lambda^{-1}\beta\tau\lambda $, that is, $ \tau\lambda\beta=\beta\tau\lambda $, so $ \beta \in ((E_{7,\sR})^C)^{\tau\lambda} $.

    With above, this proposition is proved.
\end{proof}

Hence the following corollary follows from Proposition \ref{proposition 5.11} above.

\begin{corollary}\label{corollary 5.12}
The Lie algebra $ ((\mathfrak{e}_{7,\sR})^C)^{\tau\lambda} $ of the group $ ((E_{7,\sR})^C)^{\tau\lambda} $ coincides with the Lie algebra $ \mathfrak{e}_{7,\sR} $ of the group $ E_{7,\sR} ${\rm :}$ ((\mathfrak{e}_{7,\sR})^C)^{\tau\lambda}=\mathfrak{e}_{7,\sR} $.
\end{corollary}
\begin{proof}
Since the group $ Sp(3,C) $ is simply connected (\cite[Theorem 6.13]{iy11}) and together with $ Sp(3,\H^C) \cong Sp(3,C) $ (\cite[Theorem 3.6 (2)]{iy11}), $ (E_{7,\sR})^C $ is simply connected (Theorem \ref{theorem 5.10}). Hence $ ((E_{7,\sR})^C)^{\tau\lambda} $ is connected, so that $  E_{7,\sR} $ is also connected (Proposition \ref{proposition 5.11}). Thus we have the required result.
\end{proof}

We consider the following subgroup $ (E_{7,\sR})_{\dot{1}} $ of $ E_{7,\sR} $:
\begin{align*}
(E_{7,\sR})_{\dot{1}} =\left\lbrace \alpha \in E_{7,\sR} \relmiddle{|}\alpha \dot{1}=\dot{1} \right\rbrace.
\end{align*}

\begin{proposition}\label{proposition 5.13}
   The group $ (E_{7,\sR})_{\dot{1}} $ is isomorphic to the group $ E_{6,\sR} ${\rm :} $ (E_{7,\sR})_{\dot{1}} \cong E_{6,\sR} $.
\end{proposition}
\begin{proof}
  We can prove this proposition as that in the proof of \cite[Theorem 4.7.2]{iy0}.
\end{proof}

As for the Lie algebra $ \mathfrak{sp}(3) $, we have the following lemma as that in Lemma \ref{lemma 5.1}.

\begin{lemma}\label{lemma 5.14}
   Any element $ D \in \mathfrak{sp}(3) $ is uniquely expressed by
   \begin{align*}
      D=B+L(e_2E)+sE,\,\,B \in \mathfrak{su}(3),L \in \mathfrak{S}(3,\C),s \in \R e_1,
   \end{align*}
   where $ \mathfrak{S}(3,\C)=\{ L \in M(3,\C) \,|\, {}^t\!L=L \} $ and $ e_1,e_2 $ are the basis in $ \H:=\{1,e_1,e_2,e_3\}_{\rm span} $.
\end{lemma}
\begin{proof}
We can prove this lemma by replacing $ \H^C , \C^C $ with $ \H , \C $ in the proof of Lemma \ref{lemma 5.1}, respectively.
\end{proof}

We will prove the following theorem.

\begin{theorem}\label{theorem 5.15}
   The Lie algebra $ \mathfrak{e}_{7,\sR} $ is isomorphic to the Lie algebra $ \mathfrak{sp}(3) ${\rm :} $ \mathfrak{e}_{7,\sR} \cong\mathfrak{sp}(3) $.
\end{theorem}
\begin{proof}
   Let the Lie algebra $ \mathfrak{sp}(3):=\{ B +L(e_2E)+sE \,|\, B \in \mathfrak{su}(3), L \in \mathfrak{S}(3,\C), s \in \R e_1\} $ (Lemma \ref{lemma 5.14}) and the Lie algebra $ ((\mathfrak{e}_{7,\sR})^C)^{\tau\lambda} $ as $ \mathfrak{e}_{7,\sR} $.
   Then we define a mapping $ f_{7,\sR_*}: \mathfrak{sp}(3) \to (\mathfrak{e}_{7,\sR})^{\tau\lambda}$ by
  \begin{align*}
  f_{7,\sR_*}( B +L(e_2E)+sE)=\varPhi(f_{6,C_*}(g(B)), -(\iota L+\ov{\iota}\,\ov{L}),\ov{\iota}\,L+\iota\,\ov{L}, 3(\iota-\ov{\iota})s).
  \end{align*}
  Note that this mapping is the restriction of the mapping $ f_{7,C_*} $ which is defined in the proof of Theorem \ref{theorem 5.3} and moreover $ g(B) \in \mathfrak{su}(3) \subset \mathfrak{sl}(3,C) $.

 First, we will prove that $ f_{7,\sR_*} $ is well-defined. From Theorem \ref{theorem 4.9}, we have $ f_{6_*}(B) \in \mathfrak{e}_{6,\sR} $, and in addition, it is easy to verify $ -(\iota L+\ov{\iota}\,\ov{L}),\ov{\iota}\,L+\iota\,\ov{L} \in (\mathfrak{J}_{\sR})^C, \ov{\iota}\,L+\iota\,\ov{L}=-\tau(-(\iota L+\ov{\iota}\,\ov{L}))  $ and $ 3(\iota-\ov{\iota})s \in \R $. Hence  $ f_{7,\sR_*} $ is well-defined. Subsequently,  we will prove that $ f_{7,\sR_*} $ is a Lie-homomorphism.
 However, since $ f_{7,\sR_*} $ is the restriction of the mapping $ f_{7,C_*} $, we can prove this by same argument as that in the proof of Theorem \ref{theorem 5.3}.

Next, we will prove that $ f_{7,\sR_*} $ is injective. Since we easily see $ \Ker\,f_{7,\sR_*}=\Ker\,f_{7,C_*}=\{0\}$, it is clear.

Finally, we will prove that  $ f_{7,C_*} $ is surjective. Let $ \varPhi \in ((\mathfrak{e}_{7,\sR})^C)^{\tau\lambda} \subset (\mathfrak{e}_{7,\sR})^C $. Then there exists $ B+L(e_2E)+sE \in \mathfrak{sp}(3,\H^C), B \in \mathfrak{su}(3,\C^C) , L \in \mathfrak{S}(3,\C^C), s \in Ce_1$ such that $ \varPhi=f_{7,C_*}(B+L(e_2E)+sE) $ (Theorem \ref{theorem 5.3}). Moreover, $ \varPhi $ satisfies the condition $ (\tau\lambda)\varPhi(\lambda^{-1}\tau)=\varPhi $, that is,  $ (\tau\lambda)f_{7,C_*}(B+L(e_2E)+sE)(\lambda^{-1}\tau)=f_{7,C_*}(B+L(e_2E)+sE) $, so it follows that
\begin{align*}
(\tau\lambda)f_{7,C_*}(B+L(e_2E)+sE)(\lambda^{-1}\tau)&=(\tau\lambda)\varPhi(f_{6,C_*}(g(B)), -(\iota L+\ov{\iota}\ov{L}), \ov{\iota}L+\iota\ov{L}, 3(\iota-\ov{\iota})s)(\lambda^{-1}\tau)
\\
&=\varPhi(-\tau\,{}^t(f_{6,C_*}(g(B)))\tau, -\tau(\ov{\iota}L+\iota\ov{L}), \tau(-(\iota L+\ov{\iota}\ov{L})),3\tau (\iota-\ov{\iota})s).
\end{align*}
Here, we use the following claim:
\begin{align*}
-\tau\,{}^t(f_{6,C_*}(g(B)))\tau=f_{6,C_*}(g(\tau B)).
\end{align*}
Indeed, first it follows from Lemma \ref{lemma 5.2} that
\begin{align*}
({}^tf_{6,C_*}(g(B))X,Y)&=(X,f_{6,C_*}(g(B))Y)=(X,g(B)Y+Y\,{}^t(g(B)))
\\
&=(X,g(B)Y)+(X,Y\,{}^t(g(B)))=({}^t(g(B))X,Y)+(g(B)X,Y)
\\
&=({}^t(g(B))X+g(B)X,Y)=(f_{6,C_*}({}^t(g(B)))X,Y), \;\; X,Y \in (\mathfrak{J}_{\sR})^C,
\end{align*}
that is, $ {}^tf_{6,C_*}(g(B))=f_{6,C_*}({}^t(g(B))) $, and moreover it follows that
\begin{align*}
\tau f_{6,C_*}(g(B))\tau X&=\tau(g(B)(\tau X)+(\tau X){}^t(g(B)))=(\tau g(B))X+X(\tau{}^t(g(B)))
\\
&=(\tau g(B))X+X{}^t(\tau g(B))=f_{6,C_*}(\tau g(B))X,\;\; X \in (\mathfrak{J}_{\sR})^C,
\end{align*}
that is, $ \tau f_{6,C_*}(g(B))\tau=f_{6,C_*}(\tau g(B)) $. In addition, it is easy to verify that $ -\tau\,{}^t(g(B))=g(\tau B)$. With above, the claim is shown. Hence $ B \in \mathfrak{su}(3) $ follows from $ -\tau\,{}^t(f_{6,C_*}(g(B)))\tau=f_{6,C_*}(g(B)) $.

\noindent As for the remaining parts, $ L \in \mathfrak{S}(3,\C) $ follows from $ -\tau(\ov{\iota}L+\iota\ov{L})=\ov{\iota}L+\iota\ov{L} $ and $ \tau(-(\iota L+\ov{\iota}\ov{L}))=\ov{\iota}L+\iota\ov{L} $ and we have $ s \in \R e_1 $ from $ 3\tau (\iota-\ov{\iota})s=3(\iota-\ov{\iota})s $. Thus there exists $ B+L(e_2E)+sE \in \mathfrak{sp}(3), B \in \mathfrak{su}(3) , L \in \mathfrak{S}(3,\C), s \in \Re_1$ such that $ \varPhi=f_{7,C_*}(B+L(e_2E)+sE) $. The proof of surjective is completed.


Therefore we have the required isomorphism
   \begin{align*}
      \mathfrak{e}_{7,\sR} \cong \mathfrak{sp}(3).
   \end{align*}
\end{proof}


Now, we will determine the structure of the group $ E_{7,\sR} $.

\begin{theorem}\label{theorem 5.16}
   The group $ E_{7,\sR} $ is isomorphic to the group $ Sp(3) ${\rm :} $ E_{7,\sR} \cong Sp(3) $.
\end{theorem}
\begin{proof}
First, it is to verify that $ z((E_{7,\sR})^C) \subset z(E_{7,\sR}) $, where both are the center of the groups $ (E_{7,\sR})^C $ and $ E_{7,\sR} $, respectively. Indeed, let $ \delta \in z((E_{7,\sR})^C) $. Then, since $ E_{7,\sR} \subset (E_{7,\sR})^C $, $ \delta $ commutes with any elements of $ E_{7,\sR} $. Hence we have $ \delta \in z(E_{7,\sR}) $, that is, $ z((E_{7,\sR})^C) \subset z(E_{7,\sR}) $. Here, since $ (E_{7,\sR})^C \cong Sp(3,\H^C) $(Theorem \ref{theorem 5.10}) and $ Sp(3,\H^C) $ is simply connected (\cite[Theorems 3.6 (2), 6.13]{iy11}), the group $  (E_{7,\sR})^C $ is also simply connected. Hence the group $ E_{7,\sR}=((E_{7,\sR})^C)^{\tau\lambda} $ is connected (Proposition \ref{proposition 5.11}). In addition, since the Lie algebra $ \mathfrak{e}_{7,\sR} $ of its group is isomorphic to the Lie algebra $ \mathfrak{sp}(3) $ (Theorem \ref{theorem 5.15}), the group $ E_{7,\sR} $ is isomorphic to either of the following groups
\begin{align*}
Sp(3),\quad Sp(3)/\Z_2.
\end{align*}
In the latter case, if $ E_{7,\sR} \cong Sp(3/\Z_2) $, then the center $ z(Sp(3)/\Z_2) $ is trivial, so that the center $ z(E_{7,\sR}) $ is also trivial. However, since $  z((E_{7,\sR})^C) \subset z(E_{7,\sR}) $ as mentioned above and $ z((E_{7,\sR})^C)=\{1,-1\} $ (Theorem \ref{theorem 5.9}), the elements $ 1, -1 $ at least belong to $z(E_{7,\sR}) $. Thus this case is impossible.

Therefore $ E_{7,\sR} $ has to be isomorphic to $ Sp(3) $:
\begin{align*}
       E_{7,\sR} \cong Sp(3).
    \end{align*}
\end{proof}



We move the determination of the root system and the Dynkin diagram of the Lie algebra $ (\mathfrak{e}_{7,\sR})^C $.
\vspace{1mm}

We define a Lie subalgebra $ \mathfrak{h} $ of $ (\mathfrak{e}_{7,\sR})^C $ by
\begin{align*}
\mathfrak{h}_7=\left\lbrace \varPhi_7=\varPhi(\phi,0,0,\nu) \relmiddle{|}
\begin{array}{l}
\phi=\tilde{T}_0 \in \mathfrak{h}_6,
\\
\quad T_0=\tau_1E_1+\tau_2E_2+\tau_3E_3 \in (({\mathfrak{J}_{\sC}})^C)_0,
\\
\qquad\quad \tau_1+\tau_2+\tau_3=0, \tau_i \in C,
\\
\nu \in C
\end{array}
\right\rbrace.
\end{align*}
Then the Lie subalgebra $ \mathfrak{h}_7 $ is a Cartan subalgebra of $ (\mathfrak{e}_{7,\sR})^C $.

\begin{theorem}\label{theorem 5.17}
    The roots $ \varDelta $ of $ (\mathfrak{e}_{7,\sR})^C  $ relative to $ \mathfrak{h}_7 $ are given by
    \begin{align*}
    \varDelta=\left\lbrace
    \begin{array}{l}
    \pm\dfrac{1}{2}(\tau_2-\tau_3),\pm\dfrac{1}{2}(\tau_3-\tau_1),\pm\dfrac{1}{2}(\tau_1-\tau_2)
    \\[2mm]
    \pm(\tau_1+\dfrac{2}{3}\nu), \pm(\tau_2+\dfrac{2}{3}\nu), \pm(\tau_3+\dfrac{2}{3}\nu),
    \\[2mm]
    \pm(\dfrac{1}{2}\tau_1-\dfrac{2}{3}\nu), \pm(\dfrac{1}{2}\tau_2-\dfrac{2}{3}\nu), \pm(\dfrac{1}{2}\tau_3-\dfrac{2}{3}\nu)
    \end{array}
    \right\rbrace .
    \end{align*}
\end{theorem}
\begin{proof}
    The roots of $ (\mathfrak{e}_{6,\sR})^C   $ are also the roots of $ (\mathfrak{e}_{7,\sR})^C  $. Indeed, let the root $ \alpha $ of $ (\mathfrak{e}_{6,\sR})^C  $ and its associated root vector $ S \in (\mathfrak{e}_{6,\sR})^C  \subset (\mathfrak{e}_{7,\sR})^C  $. Then we have
    \begin{align*}
    [\varPhi_7, S]&=[\varPhi(\phi,0,0,\nu), \varPhi(S,0,0,0)]
    \\
    &=\varPhi([\phi,S],0,0,0)
    \\
    &=\varPhi(\alpha(\phi)S,0,0,0)
    \\
    &=\alpha(\phi)\varPhi(S,0,0,0)
    \\
    &=\alpha(\varPhi)S.
    \end{align*}

    We will determine the remainders of roots. We will show a few examples.
    First, let $ \varPhi_7=\varPhi(\phi,0,\allowbreak 0,\nu) \in \mathfrak{h}_7 $ and $ \varPhi(0,E_1,0,0) \in (\mathfrak{e}_{7,\sR}) ^C$. Then it follows that
    \begin{align*}
    [\varPhi_7, \varPhi(0,E_1,0,0)]&=[\varPhi(\phi, 0,0,\nu), \varPhi(0,E_1,0,0)]
    \\
    &=\varPhi(0,(\phi+(2/3)\nu)E_1,0,0)
    \\
    &=\varPhi(0,(\tau_1+(2/3)\nu)E_1,0,0)
    \\
    &=(\tau_1+(2/3)\nu)\varPhi(0,E_1,0,0),
    \end{align*}
that is, $ [\varPhi_7, \varPhi(0,E_1,0,0)]=(\tau_1+(2/3)\nu)\varPhi(0,E_1,0,0) $. Hence we see that $ \tau_1+(2/3)\nu $ is a root and $ \varPhi(0,E_1,0,0) $ is an associated root vector. Next, let $ \varPhi(0, F_1(1),0,0) \in (\mathfrak{e}_{7,\sR})^C $. Then it follows that
    \begin{align*}
    [\varPhi_7, \varPhi(0,F_1(1),0,0)]&=[\varPhi(\phi, 0,0,\nu), \varPhi(0,F_1(1),0,0)]
    \\
    &=\varPhi(0,(\phi+(2/3)\nu)F_1(1),0,0)
    \\
    &=\varPhi(0,\phi F_1(1)+(2/3)\nu F_1(1), 0,0)
    \\
    &=\varPhi(0,(1/2)(\tau_2+\tau_3)F_1(1)+(2/3)\nu F_1(1),0,0)
    \\
    &=((1/2)(\tau_2+\tau_3)+(2/3)\nu)\varPhi(0,F_1(1),0,0),
    \end{align*}
that is, $ [\varPhi_7, \varPhi(0,F_1(1),0,0)]=(-(1/2)\tau_1+(2/3)\nu)\varPhi(0,F_1(1),0,0) $. Hence we see that $ -(1/2)\tau_1+(2/3)\nu $ is a root and $  \varPhi(0,F_1(1),0,0) $ is an associated root vector. The remainders of roots and these associated root vectors except ones of the Lie algebra $ (\mathfrak{e}_{6,\sR})^C $ and above are obtained as follows:
    \begin{longtable}[c]{ll}
        \hspace{3mm}
        $ \text{roots}  $
        & \hspace{-5mm}
        $ \text{associated root vectors} $
        \cr
        $ \tau_2+(2/3)\nu $
        & $ \varPhi(0,E_2,0,0) $
        \cr
        $ \tau_3+(2/3)\nu $
        & $ \varPhi(0,E_3,0,0) $
        \cr
        $ -(\tau_1+(2/3)\nu) $
        & $ \varPhi(0,0,E_1,0) $
        \cr
        $ -(\tau_2+(2/3)\nu) $
        & $ \varPhi(0,0,E_2,0) $
        \cr
        $ -(\tau_3+(2/3)\nu) $
        & $ \varPhi(0,0,E_3,0) $
        \cr
        $ (1/2)\tau_1-(2/3)\nu $
        & $ \varPhi(0,0,F_1(1),0) $
        \cr
        $ -(1/2)\tau_2+(2/3)\nu $
        & $ \varPhi(0,F_2(1),0,0) $
        \cr
        $ (1/2)\tau_2-(2/3)\nu $
        & $ \varPhi(0,0,F_2(1),0) $
        \cr
        $ (-1/2)\tau_3+(2/3)\nu $
        & $ \varPhi(0,F_3(1),0,0) $
        \cr
        $ (1/2)\tau_3-(2/3)\nu $
        & $ \varPhi(0,0,F_3(1),0) $
    \end{longtable}
    Thus, since $ (\mathfrak{e}_{7,\sR})^C $ is spanned by $ \mathfrak{h}_{7,\sR} $ and associated root vectors above, the roots obtained above are all.
\end{proof}

In $ (\mathfrak{e}_{7,\sR})^C $, we define an inner product $ (\varPhi_1, \varPhi_2)_7 $ by
\begin{align*}
(\varPhi_1, \varPhi_2)_7:=-2(\phi_1,\phi_2)_6-4(A_1,B_2)-4(A_2,B_1)-\dfrac{8}{3}\nu_1\nu_2.
\end{align*}

Here, we will determine the Killing form of $ (\mathfrak{e}_{7,\sR})^C $.

\begin{theorem}\label{theorem 5.18}
   The killing form $ B_{7,\sR} $ of $ (\mathfrak{e}_{7,\sR})^C $ is given by
    \begin{align*}
    B_{7,\sR}(\varPhi_1,\varPhi_2)&=-2(\varPhi_1,\varPhi_2)_7
    \\
    &=-4(\phi_1,\phi_2)_6+8(A_1,B_2)+8(A_2,B_1)+\dfrac{16}{3}\nu_1\nu_2
    \\
    &=\dfrac{8}{3}B_{6,\sR}(\phi_1,\phi_2)+8(A_1,B_2)+8(A_2,B_1)+\dfrac{16}{3}\nu_1\nu_2
    \\
    &=\dfrac{8}{5}\tr(\varPhi_1\varPhi_2).
    \end{align*}
\end{theorem}
\begin{proof}
    The Lie algebra $ (\mathfrak{e}_{7,\sR})^C $ is simple because of $ (\mathfrak{e}_{7,\sR})^C \cong \mathfrak{sp}(3,\H^C) $ (Corollary \ref{corollary 4.6}). By the analogues argument as that in $ {\mathfrak{e}_7}^C $, we will determine the values $ k, k' \in C $
    such that
    \begin{align*}
    B_{7,\sR}(\varPhi_1, \varPhi_2)
    =k (\varPhi_1, \varPhi_2)_7=k' \tr(\varPhi_1, \varPhi_2).
    \end{align*}
    First, in order to determine $ k $, let $ \varPhi_1=\varPhi_2=\varPhi(0,0,0,1)=:\varPhi_0 $. Then we have
    \begin{align*}
    (\varPhi_0, \varPhi_0)_7=-\dfrac{8}{3}.
    \end{align*}
     On the other hand, if follows from
     \begin{align*}
     (\ad \varPhi_0)(\ad \varPhi_0)\varPhi(\phi,A,B,\nu)&=[\varPhi_0,[\varPhi_0, \varPhi(\phi,A,B,\nu)]]
     \\
     &=[\varPhi_0,\varPhi(0,\dfrac{2}{3}A,-\dfrac{2}{3}B,0)]
     \\
     &=\varPhi(0,\dfrac{4}{9}A,\dfrac{4}{9}B,0).
     \end{align*}
     that $ B_{7,\sR}(\ad \varPhi_0,\ad \varPhi_0)=\tr((\ad \varPhi_0)(\ad \varPhi_0))=(4/9)\times 6 \times 2=16/3 $. Hence we have $ k=-2 $.

     Next, we will determine $ k' $. Similarly, it follows from
     \begin{align*}
     \varPhi_0\varPhi_0(X,Y,\xi,\eta)=\varPhi_0(-\dfrac{1}{3}X,\dfrac{1}{3}Y,\xi,-\eta)=(\dfrac{1}{9}X,\dfrac{1}{9}Y,\xi,\eta)
     \end{align*}
     that $ \tr(\varPhi_0\varPhi_0)=(1/9)\times 6\times 2+1+1=10/3 $. Hence we have $ k'=8/5 $.
\end{proof}

Subsequently, we have the following theorem.

\begin{theorem}\label{theorem 5.19}
    In the root system $ \varDelta $ of Theorem {\rm \ref{theorem 5.17}},
    \begin{align*}
    \varPi=\{\alpha_1, \alpha_2, \alpha_3 \}
    \end{align*}
    is a fundamental root system of $ (\mathfrak{e}_{7,\sR})^C $, where
    $ \alpha_1=(1/2)(\tau_3-\tau_1),
    \alpha_2=(-1/2)\tau_3+(2/3)\nu,
    \alpha_3=-(\tau_2+(2/3)\nu) $.
    The Dynkin diagram of $ (\mathfrak{e}_{7,\sR})^C $ is given by
    \vspace{-2mm}

    {
        \setlength{\unitlength}{1mm}
        \scalebox{1.0}
        
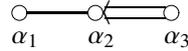
{\begin{picture}(100,20)
            \put(60,10){\circle{2}} \put(59,6){$\alpha_1$}
            \put(61,10){\line(1,0){8}}
            \put(70,10){\circle{2}} \put(69,6){$\alpha_2$}
            \put(70.7,9.2){$\langle$}
            \put(71.2,10.7){\line(1,0){8}}
            \put(71.2,9.3){\line(1,0){8}}
            \put(80,10){\circle{2}} \put(79,6){$\alpha_3$}

            \end{picture}}
    }
\end{theorem}
\begin{proof}
    The remaining positive roots are expressed by $ \alpha_1,\alpha_2,\alpha_3 $ as follows:
    \begin{align*}
    (-1/2)(\tau_2-\tau_3)&=\alpha_1+\alpha_2+\alpha_3,
    \\
    (1/2)(\tau_1-\tau_2)&=\alpha_2+\alpha_3,
    \\
    \tau_1+(2/3)\nu&=2\alpha_2+\alpha_3,
    \\
    \tau_3+(2/3)\nu&=2\alpha_1+2\alpha_2+\alpha_3,
    \\
    (-1/2)\tau_1+(2/3)\nu&=\alpha_1+\alpha_2,
    \\
    (-1/2)\tau_2+(2/3)\nu&=\alpha_1+2\alpha_2+\alpha_3.
    \end{align*}
    Hence we see that $ \varPi=\{\alpha_1, \alpha_2, \alpha_3 \} $ is a fundamental root system of $ (\mathfrak{e}_{7,\sR})^C $. Let the real part $ \mathfrak{h}_{7,\sR} $ of $ \mathfrak{h}_7 $:
    \begin{align*}
    \mathfrak{h}_{7,\sR}=\left\lbrace \varPhi=\varPhi(\phi,0,0,\nu) \relmiddle{|}
    \begin{array}{l}
    \phi=\tilde{T}_0 \in \mathfrak{h}_{6,\sR},
    \\
    \quad T_0=\tau_1E_1+\tau_2E_2+\tau_3E_3 \in (\mathfrak{J}_{\sR})_0,
    \\
    \qquad\quad \tau_1+\tau_2+\tau_3=0, \tau_i \in \R,
    \\
    \nu \in \R
    \end{array}
    \right\rbrace.
    \end{align*}
    Here, in $ (\mathfrak{e}_{7,\sR})^C $, the Killing form $ B_{7,\sR} $ of $ (\mathfrak{e}_{7,\sR})^C $ is given by
    \begin{align*}
    B_{7,\sR}(\varPhi,\varPhi')&=(8/3)B_{6,\sR}(\phi,\phi')+8(A,B')
    +8(A',B)+\dfrac{16}{3}\nu\nu'
    \\
    &=(8/5)\tr\,(\varPhi\varPhi'),
    \end{align*}
    so is on $ \mathfrak{h}_{7,\sR} $. Hence, for $  \varPhi=\varPhi(\phi,0,0,\nu),\varPhi'=\varPhi(\phi',0,0,\nu') \in \mathfrak{h}_{7,\sR} $, we easily obtain
    \begin{align*}
    B_{7,\sR}(\varPhi, \varPhi')=4(\tau_1{\tau_1}'+\tau_2{\tau_2}'+\tau_3{\tau_3}')+\dfrac{16}{3}\nu\nu'.
    \end{align*}
    Indeed, it follows from $ B_{6,\sR}(\phi, \phi')=(3/2)(\tau_1{\tau_1}'+\tau_2{\tau_2}'+\tau_3{\tau_3}')$ that
    \begin{align*}
    B_{7,\sR}(\varPhi, \varPhi')&=\dfrac{8}{3}B_{6,\sR}(\phi, \phi')+\dfrac{16}{3}\nu\nu'
    \\
    &=\dfrac{8}{3}\cdot\dfrac{3}{2} (\tau_1{\tau_1}'+\tau_2{\tau_2}'+\tau_3{\tau_3}')+\dfrac{16}{3}\nu\nu'
    \\
    &=4(\tau_1{\tau_1}'+\tau_2{\tau_2}'+\tau_3{\tau_3}')+\dfrac{16}{3}\nu\nu'.
    \end{align*}

    Now, the canonical elements $ \varPhi_{\alpha_1}, \varPhi_{\alpha_2}, \varPhi_{\alpha_3} \in \mathfrak{h}_{7,\sR} $ corresponding to $ \alpha_1,\alpha_2,\alpha_3 $ are determined as follows:
    \begin{align*}
    \varPhi_{\alpha_1}&=\dfrac{1}{8}\varPhi((-\tilde{E}_1+\tilde{E}_3,0,0,0),
    \\
    \varPhi_{\alpha_2}&=\dfrac{1}{24}\varPhi(\tilde{E}_1+\tilde{E}_2-2\tilde{E}_3,0,0,3)
    \\
    \varPhi_{\alpha_3}&=\dfrac{1}{24}\varPhi(2\tilde{E}_1-4\tilde{E}_2+2\tilde{E}_3,0,0,-3).
    \end{align*}
    Indeed, let $ \varPhi=\varPhi(\phi',0,0,\nu') \in \mathfrak{h}_{7,\sR}$ and set $ \varPhi_{\alpha_1}=\varPhi(\phi,0,0,\nu) $, then it follows from $ B_7(\varPhi_{\alpha_1},\varPhi)=\alpha_1(\varPhi) $ that
    \begin{align*}
    B_{7,\sR}(\varPhi_{\alpha_1},\varPhi)&=4(\tau_1{\tau_1}'+\tau_2{\tau_2}'+\tau_3{\tau_3}')+\dfrac{16}{3}\nu\nu'
    \\
    \alpha_1(\varPhi)&=-\dfrac{1}{2}{\tau_1}'+\dfrac{1}{2}{\tau_3}'
    \end{align*}
    Hence we have
    \begin{align*}
    \tau_1=-\dfrac{1}{8}, \tau_2=0, \tau_3=\dfrac{1}{8}, \nu=0,
    \end{align*}
    that is, $ \varPhi_{\alpha_1}=(1/8)\varPhi((-\tilde{E}_1+\tilde{E}_3,0,0,0) $. Note that $ \varPhi_{\alpha_1} $ can be also obtained by
    \begin{align*}
    \varPhi_{\alpha_1}=(B_{7,\sR}(\tilde{A}_2(1)+\tilde{F}_2(1),\tilde{A}_2(1)+\tilde{F}_2(1)))^{-1}[\tilde{A}_2(1)+\tilde{F}_2(1),\tilde{A}_2(1)+\tilde{F}_2(1))],
    \end{align*}
    where the elements $ \tilde{A}_2(1)+\tilde{F}_2(1),\tilde{A}_2(1)-\tilde{F}_2(1) $ are the root vectors associated with the roots $ \alpha_1, -\alpha_1 $.
    For the remainders of $ \varPhi_{\alpha_i},i=2,3$, as in the case above, we obtain the required results.

    With above, we see that
    \begin{align*}
    (\alpha_1,\alpha_1)&=B_{7,\sR}(\varPhi_{\alpha_1},\varPhi_{\alpha_1})=4\cdot\left( \left(-\dfrac{1}{8} \right)\cdot\left(-\dfrac{1}{8} \right)+\dfrac{1}{8} \cdot\dfrac{1}{8}\right) =\dfrac{1}{8},
    \\
    (\alpha_1,\alpha_2)&=B_{7,\sR}(\varPhi_{\alpha_1},\varPhi_{\alpha_2})=4\cdot\left( \left(-\dfrac{1}{8} \right)\cdot\dfrac{1}{24}+\dfrac{1}{8}\cdot\left(-\dfrac{2}{24}\right)\right)=-\dfrac{1}{16},
    \\
    (\alpha_1,\alpha_3)&=B_{7,\sR}(\varPhi_{\alpha_1},\varPhi_{\alpha_3})=4\cdot\left(\left( -\dfrac{1}{8}\right)\cdot\dfrac{2}{24} +\dfrac{1}{8} \cdot\dfrac{2}{24} \right)=0,
    \\
    (\alpha_2,\alpha_2)&=B_{7,\sR}(\varPhi_{\alpha_2},\varPhi_{\alpha_2})=4\cdot\left(\dfrac{1}{24}\cdot\dfrac{1}{24}+\cdot\dfrac{1}{24}\cdot \dfrac{1}{24}+\left( -\dfrac{2}{24}\right)\cdot\left(-\dfrac{2}{24} \right)  \right)=\dfrac{1}{8},
    \\
    (\alpha_2,\alpha_3)&=B_{7,\sR}(\varPhi_{\alpha_2},\varPhi_{\alpha_3})=4\cdot\left( \dfrac{1}{24}\cdot \dfrac{2}{24} +\dfrac{1}{24}\left( -\dfrac{4}{24} \right)+\left(-\dfrac{2}{24} \right) \cdot \dfrac{2}{24} \right)=-\dfrac{1}{8},
    \\
    (\alpha_3,\alpha_3)&=B_{7,\sR}(\varPhi_{\alpha_3},\varPhi_{\alpha_3})=4\cdot\left(\dfrac{2}{24}\cdot\dfrac{2}{24} +\left( -\dfrac{4}{24}\right) \cdot\left( -\dfrac{4}{24}\right)  + \dfrac{2}{24}\cdot\dfrac{2}{24} \right) = \dfrac{1}{4}.
    \end{align*}
    Hence, since we have
    \begin{align*}
    \cos\theta_{12}&=\dfrac{(\alpha_1, \alpha_2)}{\sqrt{(\alpha_1,\alpha_1)(\alpha_2,\alpha_2)}}=-\dfrac{1}{2},\quad
    \cos\theta_{13}=\dfrac{(\alpha_1, \alpha_3)}{\sqrt{(\alpha_1,\alpha_1)(\alpha_3,\alpha_3)}}=0,
    \\
    \cos\theta_{23}&=\dfrac{(\alpha_2, \alpha_3)}{\sqrt{(\alpha_2,\alpha_2)(\alpha_3,\alpha_3)}}=-\dfrac{1}{\sqrt{2}}, \quad (\alpha_3,\alpha_3)=\dfrac{1}{4} > \dfrac{1}{8}=(\alpha_2,\alpha_2),
    \end{align*}
    we can draw the Dynkin diagram.
\end{proof}

\section{ The root system, the Dynkin diagram of the Lie algebra $ ({\mathfrak{e}_{8,\sR}})^C $ of the group $ (E_{8,\sR})^C $ and the Lie groups $ (E_{8,\sR})^C, E_{8,\sR} $ }

We consider the following complex Lie algebra $ (\mathfrak{e}_{8,\sR})^C $ which is given by replacing $ \mathfrak{C} $ with $ \R $ in the complex Lie algebra $ {\mathfrak{e}_8}^C $:
\begin{align*}
(\mathfrak{e}_{8,\sR})^C&=(\mathfrak{e}_{7,\sR})^C \oplus (\mathfrak{P}_{\sR})^C \oplus (\mathfrak{P}_{\sR})^C \oplus C \oplus C \oplus C
\\
&=\left\lbrace R=(\varPhi,P,Q,r,s,t) \relmiddle{|}
\begin{array}{l}
\varPhi \in (\mathfrak{e}_{7,\sR})^C, P,Q \in (\mathfrak{P}_{\sR})^C, \\
r,s,t \in C
\end{array}
\right\rbrace
\end{align*}
In particular, we have $ \dim_C((\mathfrak{e}_{8,\sC})^C)=21+14\times 2+3=52 $. In $ (\mathfrak{e}_{8,\sR})^C $, we can define a Lie bracket and can prove its simplicity as that in $ {\mathfrak{e}_8}^C $.
In addition, we define a symmetric inner product $ (R_1,R_2)_8 $ by
\begin{align*}
(R_1,R_2)_8:=(\varPhi_1,\varPhi_2)_7-\{Q_1, P_2\}+\{P_1,Q_2\}-8r_1r_2-4t_1s_2-4s_1t_2,
\end{align*}
where $ R_i:=(\varPhi_i,P_i,Q_i,r_i,s_i,t_i) \in (\mathfrak{e}_{8,\sR})^C,i=1,2 $. Then $ (\mathfrak{e}_{8,\sR})^C $ leaves the symmetric inner product $ (R_1,R_2)_8 $ invariant (cf. \cite[Lemma 5.3.1]{iy0}).

Then we have the following theorem.

\begin{theorem}\label{theorem 6.1}
    The Killing form $ B_{8,\sR} $ of $ (\mathfrak{e}_{8,\sR})^C $ is given by
    \begin{align*}
    B_{8,\sR}(R_1,R_2)&=-\dfrac{9}{2}(R_1,R_2)_8
    \\
    &=-\dfrac{9}{2}(\varPhi_1,\varPhi_2)_7+\dfrac{9}{2}\{Q_1, P_2\}-\dfrac{9}{2}\{P_1,Q_2\}+36r_1r_2+18t_1s_2+18s_1t_2
    \\
    &=\dfrac{9}{4}B_{7,\sR}(\varPhi_1,\varPhi_2)+\dfrac{9}{2}\{Q_1, P_2\}-\dfrac{9}{2}\{P_1,Q_2\}+36r_1r_2+18t_1s_2+18s_1t_2.
    \end{align*}
\end{theorem}
\begin{proof}
    Since $ (\mathfrak{e}_{8,\sR})^C $ is simple, there exists $k \in C$ such that
    \begin{align*}
    B_{8,\sR}(R_1, R_2) = k(R_1, R_2)_8, \quad R_i \in (\mathfrak{e}_{8,\sR})^C.
    \end{align*}
    We will determine $k$. Let $ R_1 = R_2= (0, 0, 0, 1, 0, 0)=: \tilde{1}$. Then we have
    \begin{align*}
    (\tilde{1},\tilde{1})_8=-8.
    \end{align*}
    On the other hand, it follows from
    \begin{align*}
    (\ad\,\tilde{1})(\ad\,\tilde{1})R
    &=[\tilde{1}, [\tilde{1}, (\varPhi, P, Q, r, s, t)]\,],\,\, R \in (\mathfrak{e}_{8,\sR})^C
    \\
    &= [\tilde{1}, (0, P, - Q, 0, 2s, - 2t)]
    \\
    &= (0, P, Q, 0, 4s, 4t)
    \end{align*}
    that
    \begin{align*}
    B_{8,\sR}(\tilde{1}, \tilde{1})=\tr((\ad\tilde{1})(\ad\tilde{1}))= 14 \times 2 + 4 \times 2 = 36.
    \end{align*}
    Thus we have $k = - 9/2$.

    Therefore we have
    \begin{align*}
    B_{8,\sR}(R_1, R_2) = -\dfrac{9}{2}(R_1, R_2)_8
    \end{align*}
    and the remainders of formulas can be easily obtained.
\end{proof}


We move the determination of the root system and Dynkin diagram of the Lie algebra $ (\mathfrak{e}_{8,\sR})^C $.
\vspace{1mm}

We define a Lie subalgebra $ \mathfrak{h}_8 $ of $ (\mathfrak{e}_{8,\sR})^C $ by
\begin{align*}
\mathfrak{h}_8=\left\lbrace R_8=(\varPhi,0,0,r,0,0) \relmiddle{|}
\begin{array}{l}
\varPhi=\varPhi(\phi,0,0,\nu) \in \mathfrak{h}_7,
\\
\quad \phi=\tilde{T}_0 \in \mathfrak{h}_6,
\\
\qquad T_0=\tau_1E_1+\tau_2E_2+\tau_3E_3 \in (({\mathfrak{J}_{\sR}})^C)_0,
\\
\qquad\qquad \tau_1+\tau_2+\tau_3=0, \tau_i \in C,
\\
\quad \nu \in C,
\\
r \in C
\end{array}
\right\rbrace.
\end{align*}
Then the Lie subalgebra $ \mathfrak{h}_8 $ is a Cartan subalgebra of $ (\mathfrak{e}_{8,\sR})^C $.

\begin{theorem}\label{Theorem 6.2}
    The roots $ \varDelta $ of $ (\mathfrak{e}_{8,\sR})^C  $ relative to $ \mathfrak{h}_8 $ are given by
    \begin{align*}
    \varDelta=\left\lbrace
    \begin{array}{l}
    \pm\dfrac{1}{2}(\tau_2-\tau_3),\pm\dfrac{1}{2}(\tau_3-\tau_1),\pm\dfrac{1}{2}(\tau_1-\tau_2)
    \\[2mm]
    \pm(\tau_1+\dfrac{2}{3}\nu), \pm(\tau_2+\dfrac{2}{3}\nu), \pm(\tau_3+\dfrac{2}{3}\nu),
    \\[2mm]
    \pm(\dfrac{1}{2}\tau_1-\dfrac{2}{3}\nu), \pm(\dfrac{1}{2}\tau_2-\dfrac{2}{3}\nu), \pm(\dfrac{1}{2}\tau_3-\dfrac{2}{3}\nu),
    \\[2mm]
    \pm(\tau_1-\dfrac{1}{3}\nu+r),\pm(\tau_2-\dfrac{1}{3}\nu+r),\pm(\tau_3-\dfrac{1}{3}\nu+r),
    \\[2mm]
    \pm(-\dfrac{1}{2}\tau_1-\dfrac{1}{3}\nu+r),\pm(-\dfrac{1}{2}\tau_2-\dfrac{1}{3}\nu+r),\pm(-\dfrac{1}{2}\tau_3-\dfrac{1}{3}\nu+r),
    \\[2mm]
    \pm(\dfrac{1}{2}\tau_1+\dfrac{1}{3}\nu+r),\pm(\dfrac{1}{2}\tau_2+\dfrac{1}{3}\nu+r),\pm(\dfrac{1}{2}\tau_3+\dfrac{1}{3}\nu+r),
    \\[2mm]
    \pm(-\tau_1+\dfrac{1}{3}\nu+r),\pm(-\tau_2+\dfrac{1}{3}\nu+r),\pm(-\tau_3+\dfrac{1}{3}\nu+r),
    \\[2mm]
    \pm(\nu+r),\pm(-\nu+r),\pm 2r
    \end{array}
    \right\rbrace .
    \end{align*}
\end{theorem}
\begin{proof}
    The roots of $ (\mathfrak{e}_{7,\sR})^C $ are also the roots of $ (\mathfrak{e}_{8,\sR})^C $. Indeed, let the root $ \alpha $ of $ (\mathfrak{e}_{7,\sR})^C $ and its associated root vector $ \varPhi_s \in (\mathfrak{e}_{7,\sR})^C \subset (\mathfrak{e}_{8,\sR})^C $. Then we have
    \begin{align*}
    [R_8, \varPhi_s]&=[(\varPhi,0,0,r,0,0) (\varPhi_s,0,0,0,0,0)]
    \\
    &=([\varPhi,\varPhi_s],0,0,0,0,0)
    \\
    &=(\alpha(\varPhi)\varPhi_s,0,0,0,0,0)
    \\
    &=\alpha(\varPhi)(\varPhi_s,0,0,0,0,0)
    \\
    &=\alpha(R)\varPhi_s.
    \end{align*}

    We will determine the remainders of roots. We will show a few examples.
    First, let $ R_8=(\varPhi,0,0,r,0,0) \in \mathfrak{h}_8 $ and $ {R^-}_{\dot{E}_1}:=(0,\dot{E}_1,0,0,0,0) \in (\mathfrak{e}_{8,\sR})^C $, where $ \dot{E}_1:=(E_1,0,0,0) \in (\mathfrak{P}_{\sR})^C $. Then it follows that
    \begin{align*}
    [R_8,{R^-}_{\dot{E}_1}]&=[(\varPhi,0,0,r,0,0),(0,\dot{E}_1,0,0,0,0)]
    \\
    &=(0,\varPhi\dot{E}_1+r\dot{E}_1,0,0,0,0)
    \\
    &=(0,\varPhi(\phi,0,0,\nu)(E_1,0,0,0)+r(E_1,0,0,0),0,0,0,0)
    \\
    &=(0,(\phi-(1/3)\nu+r)E_1,0,0,0),0,0,0,0),\,\,(\phi=\tilde{T}_0)
    \\
    &=(0,((\tau_1-(1/3)\nu+r)E_1,0,0,0),0,0,0,0)
    \\
    &=(\tau_1-(1/3)\nu+r)(0,\dot{E}_1,0,0,0,0)
    \\
    &=(\tau_1-(1/3)\nu+r){R^-}_{\dot{E}_1},
    \end{align*}
    that is, $ [R_8,{R^-}_{\dot{E}_1}]=(\tau_1-(1/3)\nu+r){R^-}_{\dot{E}_1} $. Hence we see that $ \tau_1-(1/3)\nu+r $ is a root and $ (0,\dot{E}_1,0,0,0,0)$ is an associated root vector.
    Next, let $ {R^-}_{\dot{F}_1(1)}:=(0,\dot{F}_1(1),0,0,0,0) \in (\mathfrak{e}_{8,\sR})^C $, where $ \dot{F}_1(1):=(F_1(1),0,0,0) \in (\mathfrak{P}_{\sR})^C $. Then it follows that
    \begin{align*}
    [R_8,{R^-}_{\dot{F}_1(1)}]&=[(\varPhi,0,0,r,0,0),(0,\dot{F}_1(1),0,0,0,0)]
    \\
    &=(0,\varPhi\dot{F}_1(1)+r\dot{F}_1(1),0,0,0,0)
    \\
    &=(0,\varPhi(\phi,0,0,\nu)(F_1(1),0,0,0)+r(F_1(1),0,0,0),0,0,0,0)
    \\
    &=(0,(\phi-(1/3)\nu+r)F_1(1),0,0,0),0,0,0,0),\,\,(\phi=\tilde{T}_0)
    \\
    &=(0,(((1/2)(\tau_2+\tau_3)-(1/3)\nu+r)F_1(1),0,0,0),0,0,0,0)
    \\
    &=(0,(-(1/2)\tau_1-(1/3)\nu+r)\dot{F}_1(1),0,0,0,0)
    \\
    &=(-(1/2)\tau_1-(1/3)\nu+r){R^-}_{\dot{F}_1(1)},
    \end{align*}
    that is, $ [R_8,{R^-}_{\dot{F}_1(1)}]=(-(1/2)\tau_1-(1/3)\nu+r){R^-}_{\dot{F}_1(1)} $. Hence we see that $ -(1/2)\tau_1-(1/3)\nu+r $ is a root and $ (0,\dot{F}_1(1),0,0,0,0) $ is an associated root vector. The remainders of roots and these associated root vectors except ones of the Lie algebra $ (\mathfrak{e}_{8,\sR})^C $ and above are obtained as follows:
    \begin{longtable}[c]{ll}
        \hspace{8mm}
        $ \text{roots}  $
        & \hspace{-5mm}
        $ \text{associated root vectors} $
        \cr
        $ -(\tau_1-(1/3)\nu+r) $
        &
        $ (0,0,\underset{\dot{}}{E_1},0,0,0) $
        \cr
        $ \tau_2-(1/3)\nu+r $
        &
        $ (0,\dot{E}_2,0,0,0,0) $
        \cr
        $ -(\tau_2-(1/3)\nu+r) $
        &
        $ (0,0,\underset{\dot{}}{E_2},0,0,0) $
        \cr
        $ \tau_3-(1/3)\nu+r $
        &
        $ (0,\dot{E}_3,0,0,0,0) $
        \cr
        $ -(\tau_3-(1/3)\nu+r) $
        &
        $ (0,0,\underset{\dot{}}{E_3},0,0,0) $
        \cr
        $ -(-
        (1/2)\tau_1-(1/3)\nu+r) $
        &
        $ (0,0,\underset{\dot{}}{F_1}(1),0,0,0) $
        \cr
        $ -(1/2)\tau_2-(1/3)\nu+r $
        &
        $ (0,\dot{F_2}(1),0,0,0,0) $
        \cr
        $ -(-(1/2)\tau_2-(1/3)\nu+r) $
        &
        $ (0,0,\underset{\dot{}}{F_2}(1),0,0,0) $
        \cr
        $ -(1/2)\tau_3-(1/3)\nu+r $
        &
        $ (0,\dot{F_3}(1),0,0,0,0) $
        \cr
        $ -(-(1/2)\tau_3-(1/3)\nu+r) $
        &
        $ (0,0,\underset{\dot{}}{F_3}(1),0,0,0) $
        \cr
        $ (1/2)\tau_1+(1/3)\nu+r $
        &
        $ (0,\underset{\dot{}}{F_1}(1),0,0,0,0) $
        \cr
        $ -((1/2)\tau_1+(1/3)\nu+r) $
        &
        $ (0,0,\dot{F_1}(1),0,0,0) $
        \cr
        $ -\tau_1+(1/3)\nu+r $
        &
        $ (0,\underset{\dot{}}{E_1},0,0,0,0) $
        \cr
        $ -(-\tau_1+(1/3)\nu+r) $
        &
        $ (0,0,\dot{E_1},0,0,0) $
        \cr
        $ -\tau_2+(1/3)\nu+r $
        &
        $ (0,\underset{\dot{}}{E_2},0,0,0,0) $
        \cr
        $ -(-\tau_2+(1/3)\nu+r) $
        &
        $ (0,0,\dot{E_2},0,0,0) $
        \cr
        $ -\tau_3+(1/3)\nu+r $
        &
        $ (0,\underset{\dot{}}{E_3},0,0,0,0) $
        \cr
        $ -(-\tau_3+(1/3)\nu+r) $
        &
        $ (0,0,\dot{E_3},0,0,0) $
        \cr
        $ (1/2)\tau_2+(1/3)\nu+r $
        &
        $ (0,\underset{\dot{}}{F_2}(1),0,0,0,0) $
        \cr
        $ -((1/2)\tau_2+(1/3)\nu+r) $
        &
        $ (0,0,\dot{F_2}(1),0,0,0) $
        \cr
        $ (1/2)\tau_3+(1/3)\nu+r $
        &
        $ (0,\underset{\dot{}}{F_3}(1),0,0,0,0) $
        \cr
        $ -((1/2)\tau_3+(1/3)\nu+r) $
        &
        $ (0,0,\dot{F_3}(1),0,0,0) $
        \cr
        $ \nu+r $
        &
        $ (0,\dot{1},0,0,0,0) $
        \cr
        $ -(\nu+r) $
        &
        $ (0,0,\underset{\dot{}}{1},0,0,0) $
        \cr
        $ -\nu+r $
        &
        $ (0,\underset{\dot{}}{1},0,0,0,0) $
        \cr
        $ -(-\nu+r) $
        &
        $ (0,0,\dot{1},0,0,0) $
        \cr
        $ 2r $
        &
        $ (0,0,0,0,1,0)  $
        \cr
        $ -2r $
        &
        $ (0,0,0,0,0,1)  $,
        \cr
    \end{longtable}
where $ \dot{E_k}:=(E_k,0,0,0),\underset{\dot{}}{E_k}:=(0,E_k,0,0), \dot{F_k}(1):=(F_k(1),0,0,0),\underset{\dot{}}{F_k}(1):=(0,F_k(1),0,0),\dot{1}:=(0,0,1,\allowbreak 0), \underset{\dot{}}{1}:=(0,0,0,1) $ in $ (\mathfrak{P}_{\sR})^C $.

Therefore, since $ (\mathfrak{e}_{8,\sR})^C $ is spanned by $ \mathfrak{h}_8 $ and the associated root vectors above, the roots obtained above are all.
\end{proof}

Subsequently, we have the following theorem.

\begin{theorem}\label{theorem 6.3}
In the root system $ \varDelta $ of Theorem {\rm \ref{Theorem 6.2}},
    \begin{align*}
    \varPi=\{\alpha_1, \alpha_2, \alpha_3, \alpha_4 \}
    \end{align*}
    is a fundamental root system of $ (\mathfrak{e}_{8,\sR})^C $, where
    $ \alpha_1=-\tau_1+(1/3)\nu+r,
    \alpha_2=-2r,
    \alpha_3=(-1/2)\tau_2-(1/3)\nu+r,
    \alpha_4=(-1/2)\tau_3+(2/3)\nu $.
    The Dynkin diagram of $ (\mathfrak{e}_{8,\sR})^C $ is given by

    {\setlength{\unitlength}{1mm}
        \scalebox{1.0}
        {\setlength{\unitlength}{1mm}
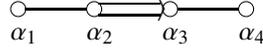
            \begin{picture}(100,20)
            \put(50,10){\circle{2}} \put(49,6){$\alpha_1$}
            \put(51,10){\line(1,0){8}}
            \put(60,10){\circle{2}} \put(59,6){$\alpha_2$}
            \put(60.6,10.7){\line(1,0){8}}
            \put(60.6,9.2){\line(1,0){8}}
            \put(68.2,9.1){$\rangle$}
            \put(70,10){\circle{2}} \put(69,6){$\alpha_3$}
            \put(71,10){\line(1,0){8}}
            \put(80,10){\circle{2}} \put(79,6){$\alpha_4$}
            \end{picture}}}
\end{theorem}
\begin{proof}
    The remainders of positive roots are expressed by $ \alpha_1,\alpha_2,\alpha_3,\alpha_4, $ as follows:
    \begin{align*}
    (-1/2)(\tau_2-\tau_3)&=\alpha_1+\alpha_2+\alpha_3+\alpha_4
    \\
    (1/2)(\tau_3-\tau_1)&=\alpha_1+\alpha_2+\alpha_3
    \\
    (1/2)(\tau_1-\tau_2)&=\alpha_2+2\alpha_3+\alpha_4
    \\
    \tau_1+(2/3)\nu&=\alpha_2+2\alpha_3+2\alpha_4
    \\
    -\tau_2-(2/3)\nu&=\alpha_2+2\alpha_3
    \\
    \tau_3+(2/3)\nu&=2\alpha_1+3\alpha_2+4\alpha_3+2\alpha_4
    \\
    (-1/2)\tau_1+(2/3)\nu&=\alpha_1+\alpha_2+\alpha_3+\alpha_4
    \\
    (-1/2)\tau_2+(2/3)\nu&=\alpha_1+2\alpha_2+3\alpha_3+2\alpha_4
    \\
    -\tau_1+(1/3)\nu-r&=\alpha_1+\alpha_2
    \\
    -\tau_2+(1/3)\nu-r&=\alpha_1+3\alpha_2+4\alpha_3+2\alpha_4
    \\
    \tau_3+(1/3)\nu+r&=\alpha_1+\alpha_2+2\alpha_3
    \\
    (1/2)\tau_1+(1/3)\nu-r&=\alpha_2+\alpha_3+\alpha_4
    \\
    (1/2)\tau_3+(1/3)\nu-r&=\alpha_1+2\alpha_2+2\alpha_3+\alpha_4
    \\
    (1/2)\tau_1+(1/3)\nu+r&=\alpha_3+\alpha_4
    \\
    (-1/2)\tau_2-(1/3)\nu-r&=\alpha_2+\alpha_3
    \\
    (1/2)\tau_3+(1/3)\nu+r&=\alpha_1+\alpha_2+2\alpha_3+\alpha_4
    \\
    -\tau_2+(1/3)\nu+r&=\alpha_1+2\alpha_2+4\alpha_3+2\alpha_4
    \\
    \tau_3-(1/3)\nu+r&=\alpha_1+2\alpha_2+2\alpha_3
    \\
    \nu+r&=\alpha_1+\alpha_2+2\alpha_3+2\alpha_4
    \\
    \nu-r&=\alpha_1+2\alpha_2+2\alpha_3+2\alpha_4.
    \end{align*}
    Hence we see that $ \varPi=\{\alpha_1, \alpha_2, \alpha_3, \alpha_4 \} $ is a fundamental root system of $ (\mathfrak{e}_{8,\sR})^C $. Let the real part $ \mathfrak{h}_{8,\sR} $ of $ \mathfrak{h}_8 $:

    \begin{align*}
    \mathfrak{h}_{8,\sR}=\left\lbrace R_8=(\varPhi,0,0,r,0,0) \relmiddle{|}
    \begin{array}{l}
    \varPhi=\varPhi(\phi,0,0,\nu) \in \mathfrak{h}_{7,\sR},
    \\
    \quad \phi=\tilde{T}_0 \in \mathfrak{h}_{6,\sR},
    \\
    \qquad T_0=\tau_1E_1+\tau_2E_2+\tau_3E_3 \in ({\mathfrak{J}_{\sR}})_0,
    \\
    \qquad\qquad \tau_1+\tau_2+\tau_3=0, \tau_i \in \R,
    \\
    \quad \nu \in \R,
    \\
    r \in \R
    \end{array}
    \right\rbrace.
    \end{align*}

    The Killing form $ B_{8,\sR} $ of $ (\mathfrak{e}_{8,\sR})^C $ is given by $ B_{8,\sR}(R_1,R_2)=(9/4)B_{7,\sR}(\varPhi_1,\varPhi_2)+(9/2)\{Q_1,P_2\}-(9/2)\{P_1,Q_2\}+36r_1r_2+18t_1s_2+18s_1t_2, R_i=(\varPhi_i,P_i,Q_i,r_i,s_i,t_i),i=1,2 $ (Theorem \ref{theorem 6.1}), so is on $ \mathfrak{h}_{8,\sR} $.
    Hence, for $ R=(\varPhi(\tau_1\tilde{E}_1+\tau_2\tilde{E}_2+\tau_3\tilde{E}_3,0,0,\nu),0,0,r,0,0), {R}'=({\varPhi}'({\tau_1}'\tilde{E}_1+{\tau_2}'\tilde{E}_2+{\tau_3}'\tilde{E}_3,0,0,\nu'),0,0,r'\allowbreak ,0,0)
    \in \mathfrak{h}_{8,\sR} $, we have
    \begin{align*}
    B_{8,\sR}(R, {R}')=9(\tau_1{\tau_1}'+\tau_2{\tau_2}'+\tau_3{\tau_3}')+12\nu\nu'+36rr'.
    \end{align*}
    Indeed, it follows from $ B_{7,\sR}(\varPhi, {\varPhi}')=4(\tau_1{\tau_1}'+\tau_2{\tau_2}'+\tau_3{\tau_3}')+(16/3)\nu\nu'$ that
    \begin{align*}
    B_{8,\sR}(R, {R}')&=\dfrac{9}{4}B_{7,\sR}(\varPhi, {\varPhi}')+36rr'
    \\
    &=\dfrac{9}{4}\cdot \left( 4(\tau_1{\tau_1}'+\tau_2{\tau_2}'+\tau_3{\tau_3}')+\dfrac{16}{3}\nu\nu'\right) +36rr'
    \\
    &=9(\tau_1{\tau_1}'+\tau_2{\tau_2}'+\tau_3{\tau_3}')+12\nu\nu'+36rr'.
    \end{align*}

    Now, the canonical elements $ R_{\alpha_1}, R_{\alpha_2}, R_{\alpha_3}, R_{\alpha_4} \in \mathfrak{h}_{8,\sR} $ corresponding to $ \alpha_1,\alpha_2,\alpha_3,\alpha_4 $ are determined as follows:
    \begin{align*}
    R_{\alpha_1}&=(\varPhi(\dfrac{1}{54}(-4\tilde{E}_1+2\tilde{E}_2+\tilde{E}_3),0,0,\dfrac{1}{36}),0,0,\dfrac{1}{36},0,0),
    \\
    R_{\alpha_2}&=(0,0,0,-\dfrac{2}{36},0,0),
    \\
    R_{\alpha_3}&=(\varPhi(\dfrac{1}{54}(\tilde{E}_1-2\tilde{E}_2+\tilde{E}_3),0,0,-\dfrac{1}{36}),0,0,\dfrac{1}{36},0,0),
    \\
    R_{\alpha_4}&=(\varPhi(\dfrac{1}{54}(\tilde{E}_1+\tilde{E}_2-2\tilde{E}_3),0,0,\dfrac{2}{36}),0,0,0,0,0).
    \end{align*}
 Indeed, let $ R=(\varPhi',0,0,r',0,0) \in (\mathfrak{e}_{8,\sR})^C $ and set $ R_{\alpha_1}=(\varPhi,0,0,r,0,0)\in \mathfrak{h}_{8,\sR}$, then from $ B_{8,\sR}(R_{\alpha_1},R)=\alpha_1(R) $, we have
    \begin{align*}
    \tau_1=-\dfrac{4}{54},\tau_2=\tau_3=\dfrac{2}{54}, \nu=\dfrac{1}{36}, r=\dfrac{1}{36}.
    \end{align*}
    Hence we have the required one as $ R_{\alpha_1} $. For the remainders of $ R_{\alpha_i},i=2,3,4,5,6 $, as in the case above, we obtain the required results.
    Note that $ R_{\alpha_1} $ is also given as follows:
    \begin{align*}
    R_{\alpha_1}=\left( B_{8,\sR}({\underset{\dot{}}{E_1}}^-,{\dot{E_1}}_-)\right)^{-1}[{\underset{\dot{}}{E_1}}^-,{\dot{E_1}}_-],
    \end{align*}
    where $ {\underset{\dot{}}{E_1}}^-:=(0,\dot{E_1},0,0,0,0),{\dot{E_1}}_-:=(0,0,\underset{\dot{}}{E_1},0,0,0) $ are the root vectors associated the roots $ \alpha_1,-\alpha_1 $, respectively. For the remainders of $ R_{\alpha_i}, i=2,3,4 $, as in the case above, we obtain the required results.

    With above, we see that
    \begin{align*}
    (\alpha_1,\alpha_1)&=B_{8,\sR}(R_{\alpha_1},R_{\alpha_1})
    =9\cdot\left( \left(-\dfrac{4}{54} \right)\cdot\left(-\dfrac{4}{54} \right)+\dfrac{2}{54}\cdot\dfrac{2}{54}+\dfrac{2}{54}\cdot\dfrac{2}{54}  \right)
    \\
    &\hspace{25mm}+12\cdot\dfrac{1}{36}\cdot\dfrac{1}{36}+36\cdot\dfrac{1}{36}\cdot\dfrac{1}{36}
    =\dfrac{1}{9},
    \\
    (\alpha_1,\alpha_2)&=B_{8,\sR}(R_{\alpha_1},R_{\alpha_2})=36\cdot\dfrac{1}{36} \cdot\left(-\dfrac{2}{36}\right) =-\dfrac{1}{18},
    \\
    (\alpha_1,\alpha_3)&=B_8(R_{\alpha_1},R_{\alpha_3})=9\cdot\left( \left(-\dfrac{4}{54} \right)\cdot\dfrac{1}{54} +\dfrac{2}{54}\cdot\left( -\dfrac{2}{54}\right) +\dfrac{2}{54}\cdot\dfrac{1}{54}  \right)
    \\
    &\hspace{25mm}+12\cdot\dfrac{1}{36}\cdot\left( -\dfrac{1}{36}\right) +36\cdot\dfrac{1}{36}\cdot\dfrac{1}{36}
    =0,
    \\
    (\alpha_1,\alpha_4)&=B_8(R_{\alpha_1},R_{\alpha_4})=9\cdot\left( \left(-\dfrac{4}{54} \right)\cdot\dfrac{1}{54} +\dfrac{2}{54}\cdot\dfrac{1}{54} +\dfrac{2}{54}\cdot\left( -\dfrac{2}{54}\right)   \right)
    \\
    &\hspace{25mm}+12\cdot\dfrac{1}{36}\cdot\left( -\dfrac{1}{36}\right)
    =0,
    \\
    (\alpha_2,\alpha_2)&=B_8(R_{\alpha_2},R_{\alpha_2})
    =36\cdot\left(-\dfrac{2}{36} \right)\cdot\left(-\dfrac{2}{36} \right)
    =\dfrac{1}{9},
    \\
    (\alpha_2,\alpha_3)&=B_8(R_{\alpha_2},R_{\alpha_3})=36\cdot\dfrac{1}{36} \cdot\left(-\dfrac{2}{36} \right)
    =-\dfrac{1}{18},
    \\
    (\alpha_2,\alpha_4)&=B_8(R_{\alpha_2},R_{\alpha_4})=0,
    \\
    (\alpha_3,\alpha_3)&=B_8(R_{\alpha_3},R_{\alpha_3})=9\cdot\left( \dfrac{1}{54}\cdot\dfrac{1}{54}+\left( -\dfrac{2}{54}\right) \cdot\left( -\dfrac{2}{54}\right) +\dfrac{1}{54}\cdot\dfrac{1}{54}  \right)
    \\
    &\hspace{25mm}+12\cdot\left( -\dfrac{1}{36}\right) \cdot\left( -\dfrac{1}{36}\right) +36\cdot\dfrac{1}{36}\cdot\dfrac{1}{36}
    =\dfrac{1}{18},
    \\
    (\alpha_3,\alpha_4)&=B_8(R_{\alpha_3},R_{\alpha_4})=9\cdot\left( \dfrac{1}{54} \cdot\dfrac{1}{54} +\left( -\dfrac{2}{54}\right) \cdot\dfrac{1}{54} +\dfrac{1}{54}\cdot\left( -\dfrac{2}{54}\right)   \right)
    \\
    &\hspace{25mm}+12\cdot\left( -\dfrac{1}{36}\right)\cdot\dfrac{2}{36}
    =-\dfrac{1}{36},
    \\
    (\alpha_4,\alpha_4)&=B_8(R_{\alpha_4},R_{\alpha_4})=9\cdot\left( \dfrac{1}{54} \cdot\dfrac{1}{54} +\dfrac{1}{54} \cdot\dfrac{1}{54} +\left( -\dfrac{2}{54}\right) \cdot\left( -\dfrac{2}{54}\right)   \right)
    \\
    &\hspace{25mm}+12\cdot\dfrac{2}{36}\cdot\dfrac{2}{36}
    =\dfrac{1}{18}.
    \end{align*}
    Hence, since we have
    \begin{align*}
    \cos\theta_{12}&=\dfrac{(\alpha_1, \alpha_2)}{\sqrt{(\alpha_1,\alpha_1)(\alpha_2,\alpha_2)}}=-\dfrac{1}{2},\quad
    \cos\theta_{13}=\dfrac{(\alpha_1, \alpha_3)}{\sqrt{(\alpha_1,\alpha_1)(\alpha_3,\alpha_3)}}=0,
    \\
    \cos\theta_{14}&=\dfrac{(\alpha_1, \alpha_4)}{\sqrt{(\alpha_1,\alpha_1)(\alpha_4,\alpha_4)}}=0,\quad
    \cos\theta_{23}=\dfrac{(\alpha_2, \alpha_3)}{\sqrt{(\alpha_2,\alpha_2)(\alpha_3,\alpha_3)}}=-\dfrac{1}{\sqrt{2}},
    \\
    \cos\theta_{24}&=\dfrac{(\alpha_2, \alpha_4)}{\sqrt{(\alpha_2,\alpha_2)(\alpha_4,\alpha_4)}}=0, \quad
    \cos\theta_{34}=\dfrac{(\alpha_3, \alpha_4)}{\sqrt{(\alpha_3,\alpha_3)(\alpha_4,\alpha_4)}}=-\dfrac{1}{2},
    \\
    (\alpha_3,\alpha_3)&=\dfrac{1}{18}<\dfrac{1}{9}=(\alpha_2,\alpha_2),
    \end{align*}
    we can draw the Dynkin diagram.
\end{proof}


We consider the following complex Lie group $ (E_{8,\sR})^C $ which is defined by replacing $ \mathfrak{C} $ with $ \R $ in the complex Lie group $ {E_8}^C $:
\begin{align*}
(E_{8,\sR})^C&=\Aut((\mathfrak{e}_{8,\sR})^C)
\\
&=\left\lbrace \alpha \in \Iso_{C}((\mathfrak{e}_{8,\sR})^C) \relmiddle{|}\alpha[R_1,R_2]=[\alpha R_1,\alpha R_2]\right\rbrace,
\end{align*}
where $ (\mathfrak{e}_{8,\sR})^C $ have been already defined in the beginning of this section.

The immediate aim is to prove the connectedness of the group $ (E_{8,\sR})^C $. In order to prove it, we will use the manner used in \cite{iy8}.

First, we will study the following subgroup $ ((E_{8,\sR})^C)_{\tilde{1},1^-,1_-} $ of $ (E_{8,\sR})^C $:
\begin{align*}
((E_{8,\sR})^C)_{\tilde{1},1^-,1_-}=\left\lbrace \alpha \in (E_{8,\sR})^C \relmiddle{|} \alpha \tilde{1}=\tilde{1},\alpha 1^-=1^-, \alpha 1_-=1_-\right\rbrace.
\end{align*}

Then we have the following proposition.

\begin{proposition}\label{proposition 6.4}
    The group $ ((E_{8,\sR})^C)_{\tilde{1},1^-,1_-} $ is isomorphic to the group $ (E_{7,\sR})^C ${\rm :} $ ((E_{8,\sR})^C)_{\tilde{1},1^-,1_-} \cong (E_{7,\sR})^C $.
\end{proposition}
\begin{proof}
    This proposition can be proved as in the proof of \cite[Proposition 5.7.1]{iy0} by replacing $ \mathfrak{C} $ with $ \R $.
\end{proof}

Subsequently, we will study the following subgroup $ ((E_{8,\sR})^C)_{1_-} $ of $ (E_{8,\sR})^C $:
\begin{align*}
((E_{8,\sR})^C)_{1_-}=\left\lbrace  \alpha \in (E_{8,\sR})^C \relmiddle{|}\alpha 1_-=1_- \right\rbrace.
\end{align*}

Then we have the following lemma.

\begin{lemma}\label{lemma 6.5}
    {\rm (1)} The Lie algebra $ (\mathfrak{e}_{8,\sR})^C $ of the group $ (E_{8,\sR})^C $ is given by
    \begin{align*}
    (\mathfrak{e}_{8,\sR})^C=
    \left\lbrace R=(\varPhi,P,Q,r,s,t) \relmiddle{|}
    \begin{array}{l}
    \varPhi \in (\mathfrak{e}_{7,\sR})^C, P,Q \in (\mathfrak{P}_{\sR})^C,
    r,s,t \in C
    \end{array}
    \right\rbrace.
    \end{align*}

    In particular, we have $ \dim_C((\mathfrak{e}_{8,\sR})^C)=21+14\times 2+3=52 $.
    \vspace{2mm}

    {\rm (2)} The Lie algebra $ ((\mathfrak{e}_{8,\sR})^C)_{1_-} $ of the group $ ((E_{8,\sR})^C)_{1_-} $ is given by
    \begin{align*}
    ((\mathfrak{e}_{8,\sR})^C)_{1_-}
    &=\left\lbrace R \in (\mathfrak{e}_{8,\sR})^C \relmiddle{|} [R, 1_-]=0 \right\rbrace
    \\
    &=\left\lbrace (\varPhi, 0, Q, 0, 0, t) \in (\mathfrak{e}_{8,\sR})^C  \relmiddle{|}
    \begin{array}{l}
    \varPhi \in (\mathfrak{e}_{7,\sR})^C,
    Q \in (\mathfrak{P}_{\sR})^C,
    t \in C
    \end{array} \right\rbrace .
    \end{align*}

    In particular, we have
    \begin{align*}
    \dim_C(((\mathfrak{e}_{8,\sR})^C)_{1_-}) = 21+14+1 = 36.
    \end{align*}
\end{lemma}
\begin{proof}
    (1) The Lie algebra of the group $ (E_{8,\sR})^C $ is $ \Der((\mathfrak{e}_{8,\sR})^C) \cong (\mathfrak{e}_{8,\sR})^C $.

    (2) Using the result of $ (1) $ above, by the straightforward computation, we can easily obtain the required results.
\end{proof}

We will prove the following proposition needed in the proof of connectedness.

\begin{proposition}\label{proposition 6.6}
    The group $((E_{8,\sR})^C)_{1_-}$ is a semi-direct product of groups $\exp(\ad(((\mathfrak{P}_{\sR})^C)_- \oplus C_- ))$ and $ (E_{7,\sR})^C ${\rm:}
    \begin{align*}
    ((E_{8,\sR})^C)_{1_-}=\exp(\ad(((\mathfrak{P}_{\sR})^C)_- \oplus C_- )) \rtimes (E_{7,\sR})^C.
    \end{align*}

    In particular, the group $((E_{8,\sR})^C)_{1_-}$ is connected.
\end{proposition}
\begin{proof}
     Let $((\mathfrak{P}_{\sR})^C)_- \oplus C_- = \{(0, 0, L, 0, 0, v) \, | \, L \in (\mathfrak{P}_{\sR})^C, v \in C\}$ be a Lie subalgebra of the Lie algebra
    $((\mathfrak{e}_{8,\sR})^C)_{1_-}$ (Lemma \ref{lemma 6.5} (2)). Since it follows from $[L_-, v_-] = 0$ that $\ad(L_-)$ commutes with $\ad(v_-)$, we have $\exp(\ad(L_- + v_-)) = \exp(\ad(L_-))\exp(\ad(v_-))$. Hence the group $\exp(\ad(((\mathfrak{P}_{\sR})^C)_- \oplus C_-))$ is the connected subgroup of the group $((E_{8,\sR})^C)_{1-}$.

    Now, let $\alpha \in ((E_{8,\sR})^C)_{1-}$ and set $ \alpha\tilde{1}:= (\varPhi, P, Q, r, s, t), \alpha1^-:= (\varPhi_1, P_1, Q_1, r_1, s_1, t_1) $. Then, it follows from
    \begin{align*}
    [\alpha\tilde{1}, 1_-]&= \alpha[\tilde{1}, 1_-] = -2\alpha 1_- = -21_-=(0,0,0,0,0,-2),
    \\
    [\alpha\tilde{1}, 1_-]&=[(\varPhi,P,Q,r,s,t),(0,0,0,0,0,1)]=(0,0,-P,s,0,-2r)
    \end{align*}
    that $ P=0, r=1, s=0 $. Similarly, from $ [\alpha1^-, 1_-] = \alpha[1^-, 1_-] = \alpha\tilde{1} $, we have $ \varPhi=0, P_1=-Q, s_1=r,  s=0,t=-2r_1$. Hence we have
    \begin{align*}
    \alpha \tilde{1}=(0,0,Q,1,0,t),\quad \alpha 1^-=(\varPhi_1,-Q,Q_1,-\dfrac{t}{2},1,t_1).
    \end{align*}
    Moreover, from $[\alpha\tilde{1}, \alpha1^-] = \alpha[\tilde{1}, 1^-] = 2\alpha1^-$, we obtain the following
    $$
    \varPhi_1 = \dfrac{1}{2}Q \times Q, \; Q_1 = - \dfrac{t}{2}Q - \dfrac{1}{3}\varPhi_1Q, \; t_1 = -\dfrac{t^2}{4} - \dfrac{1}{16}\{Q, Q_1\},
    $$
    that is,
    \begin{align*}
    \alpha 1^-=(\dfrac{1}{2}Q \times Q,-Q,-\dfrac{t}{2}Q -\dfrac{1}{6}(Q \times Q)Q ,-\dfrac{t}{2},1,-\dfrac{t^2}{4} + \dfrac{1}{96}\{Q, (Q \times Q)Q).
    \end{align*}
    On the other hand, set $ \delta:= \exp(\ad((v/2)_-))\exp(\ad(L_-)) \in \exp(\ad(((\mathfrak{P}_{\sR})^C)_- \oplus C_-)) $, then we have
    \begin{align*}
    \delta 1^- &= \exp\Big(\ad(\Big(\dfrac{v}{2}\Big)_-)\Big)\exp(\ad(L_-))1^-
    \\
    &=(\dfrac{1}{2}L \times L,- L,
    - \dfrac{t}{2}L - \dfrac{1}{6}(L \times L)L ,-\dfrac{v}{2}, 1,
    -\dfrac{v^2}{4} + \dfrac{1}{96}\{L, (L \times L)L\}).
    \end{align*}
Hence we see
\begin{align*}
(\delta^{-1}\alpha) 1^-=1^-.
\end{align*}
In addition, we can confirm
\begin{align*}
  (\delta^{-1}\alpha) \tilde{1} = \tilde{1},\,\, (\delta^{-1}\alpha) 1_- =1_-.
\end{align*}
    Indeed, it follows from
    \begin{align*}
    \delta \tilde{1}&=\exp\Big(\ad(\Big(\dfrac{v}{2}\Big)_-)\Big)\exp(\ad(L_-))\tilde{1}=\exp\Big(\ad(\Big(\dfrac{v}{2}\Big)_-)\Big)(0,0,L,1,0,0)=(0,0,L,1,0,v),
    \\
    \delta 1_-&=\exp\Big(\ad(\Big(\dfrac{v}{2}\Big)_-)\Big)\exp(\ad(L_-))1_-=\exp\Big(\ad(\Big(\dfrac{v}{2}\Big)_-)\Big)1_-=1_-
    \end{align*}
    that $ (\delta^{-1}\alpha)\tilde{1}=\delta^{-1}(0,0,P,1,0,t)=\tilde{1}, (\delta^{-1}\alpha)1_-=\delta^{-1}1_-=1_-$.

    \noindent Hence, $ \delta^{-1}\alpha \in ((E_{8,\sR})^C)_{\tilde{1},1^-,1_-} = (E_{7,\sR})^C $ follows from Proposition \ref{proposition 6.4},
    so that we obtain
    $$
    ((E_{8,\sR})^C)_{1_-} = \exp(\ad(((\mathfrak{P}_{\sR})^C)_- \oplus C_-))(E_{7,\sR})^C.
    $$

    Furthermore, for $\beta \in (E_{7,\sR})^C$, it is easy to verify that
    $$
    \beta(\exp(\ad(L_-)))\beta^{-1} = \exp(\ad(\beta L_-)),\quad \beta((\exp(\ad(v_-)))\beta^{-1} = \exp(\ad(v_-)).
    $$
    Indeed, for $(\varPhi, P, Q, r, s, t) \in (\mathfrak{e}_{8,\sR})^C$, by doing simple computation, we have the following
    \begin{align*}
    \beta \ad (L_-) \beta^{-1}(\varPhi, P, Q, r, s, t)&= \beta [L_-, \beta^{-1}(\varPhi, P, Q, r, s, t)]
    \\[1mm]
    &= [\beta L_-,\beta \beta^{-1}(\varPhi, P, Q, r, s, t)]\,(\beta \in (E_{7,\sR})^C \subset (E_{8,\sR})^C )
    \\[1mm]
    &= [\beta L_-,(\varPhi, P, Q, r, s, t)]
    \\[1mm]
    &= \ad (\beta L_-) (\varPhi, P, Q, r, s, t),
    \end{align*}
    that is, $\beta \ad (L_-) \beta^{-1}= \ad (\beta L_-)$.
    Hence we obtain the following
    \begin{align*}
    \beta(\exp(\ad(L_-)))\beta^{-1}&=\beta \, \Bigl(\displaystyle{\sum_{n=0}^{\infty}\dfrac{1}{n!}\ad (L_-)^n} \Bigr)\,\beta^{-1}
    \\[2mm]
    &=\displaystyle{\sum_{n=0}^{\infty}\dfrac{1}{n!}(\beta\ad (L_-)\beta^{-1})^n}\,\,\,(\, \beta \ad (L_-) \beta^{-1}= \ad (\beta L_-))
    \\[2mm]
    &=\displaystyle{\sum_{n=0}^{\infty}\dfrac{1}{n!}(\ad (\beta L_-))^n}
    \\[2mm]
    &= \exp(\ad(\beta L_-)).
    \end{align*}
As for another case, by the argument similar to above, we have $\beta((\exp(\ad(v_-)))\beta^{-1} = \exp(\ad(v_-)$.
    \vspace{1mm}

Hence this shows that $\exp(\ad(((\mathfrak{P}_{\sR})^C)_- \oplus C_-)) = \exp(\ad(((\mathfrak{P}_{\sR})^C)_-)\exp(\ad(C_-))$ is a normal \vspace{0.5mm}subgroup of the group$(
    (E_{8,\sR})^C)_{1_-}$.
In addition, we have a split exact sequence
    $$
    1 \to \exp(\ad((\mathfrak{P}_{\sR})^C)_- \oplus C_-)) \overset{j}{\longrightarrow} ((E_{8,\sR})^C)_{1_-} \overset{\overset{p}{\scalebox{1.0}{$\longrightarrow$}}}{\underset{s}{\longleftarrow}}  (E_{7,\sR})^C \to 1.
    $$
 Indeed, as that in the case where $ (E_{7,\sR})^C $, we can show this, however we rewrite the proof. First, we define a mapping $ j $ by $ j(\delta)=\delta $. Then it is clear that $ j $ is a injective homomorphism. Subsequently, we define a mapping $ p $ by $ p(\alpha):=p(\delta\beta)=\beta, \delta \in \exp(\ad((\mathfrak{P}_{\sR})^C)_- \oplus C_-)), \beta \in (E_{7,\sR})^C $. Then it follows that
    \begin{align*}
    p(\alpha_1\alpha_2)&=p((\delta_1\beta_1)(\delta_2\beta_2))=p(\delta_1(\beta_1\delta_2{\beta_1}^{-1})\beta_1\beta_2)\,\,(\beta_1\delta_2{\beta_1}^{-1}=:{\delta_2}' \in \exp(\ad((\mathfrak{P}_{\sR})^C)_- \oplus C_-))
    \\
    &=p((\delta_1{\delta_2}')(\beta_1\beta_2))=\beta_1\beta_2
    \\
    &=p(\alpha_1)p(\alpha_2),
    \end{align*}
so that $ p $ is a homomorphism. Moreover, let $ \beta \in (E_{7,\sR})^C $, then there exists $ \alpha \in ((E_{7,\sR})^C)_{\underset{\dot{}}{1}} $ such that $ \alpha=\delta\beta $ for $ \delta \in \exp(\ad((\mathfrak{P}_{\sR})^C)_- \oplus C_-)) $. This implies that $ p $ is surjective. Finally, we define a mapping $ s $ by $ s(\beta):=s(\delta^{-1}\alpha)=\alpha $. Then, as in the mapping $ p $, we easily see that $ s $ is a homomorphism and $ ps=1 $. With above, the short sequence is a split exact sequence.

Therefore the group  $((E_{8,\sR})^C)_{1_-}$ is a semi-direct product of $\exp(\ad(((\mathfrak{P}_{\sR})^C)_- \oplus C_-))$ and $(E_{7,\sR})^C$:
\begin{align*}
((E_{8,\sR})^C)_{1_-} =
    \exp(\ad(((\mathfrak{P}_{\sR})^C)_- \oplus C_-))\rtimes (E_{7,\sR})^C.
\end{align*}

Finally, the connectedness of the group  $((E_{8,\sR})^C)_{1_-}$ follows from the connectedness of the groups
$\exp(\ad(((\mathfrak{P}_{\sR})^C)_- \oplus C_-))$ and $(E_{7,\sR})^C$ (Theorem \ref{theorem 5.8}).
\end{proof}

Before proving the connectedness, we will make some preparations.

For $R \in (\mathfrak{e}_{8,\sR})^C$, we define a $C$-linear mapping $R \times R : (\mathfrak{e}_{8,\sR})^C \to (\mathfrak{e}_{8,\sR})^C $ by
$$
(R \times R)R_1 = [R, [R, R_1]\,] + \dfrac{1}{30}B_{8,\sR}(R, R_1)R,\,\, R_1 \in {\mathfrak{e}_8}^C,
$$
where $B_{8,\sR}$ is the Killing form of $ (\mathfrak{e}_{8,\sR})^C $.
Using this mapping, we define a space $\mathfrak{W}_{\sR}$ by
$$
\mathfrak{W}_{\sR} = \{R \in (\mathfrak{e}_{8,\sR})^C \, | \, R \times R = 0, R \not= 0\}.
$$

\begin{lemma}\label{lemma 6.7}
    For $R = (\varPhi, P, Q, r, s, t) \in (\mathfrak{e}_{8,\sR})^C$,
    $R \not=0$ belongs to $(\mathfrak{e}_{8,\sR})^C$ if and only if $R$ satisfies the following conditions
    \vspace{3mm}

    {\rm (1)} $2s\varPhi - P \times P = 0$ \quad {\rm (2)} $2t\varPhi + Q \times Q = 0$
    \vspace{1mm}

    {\rm (3)} $2r\varPhi + P \times Q = 0$ \quad {\rm (4)} $\varPhi P - 3rP - 3sQ = 0 $
    \vspace{1mm}

    {\rm (5)} $\varPhi Q + 3rQ - 3tP = 0 $ \quad {\rm (6)} $\{P, Q\} - 16(st + r^2) = 0$
    \vspace{1mm}

    {\rm (7)} $2(\varPhi P \times Q_1 + 2P \times \varPhi Q_1 - rP \times Q_1 - sQ \times Q_1) - \{P, Q_1\}\varPhi = 0$
    \vspace{1mm}

    {\rm (8)} $2(\varPhi Q \times P_1 + 2Q \times \varPhi P_1 + rQ \times P_1 - tP \times P_1) \!- \{Q, P_1\}\varPhi = 0$
    \vspace{1mm}

    {\rm (9)} $8((P \times Q_1)Q - stQ_1 - r^2Q_1 - \varPhi^2Q_1 + 2r\varPhi Q_1) + 5\{P, Q_1\}Q
    -2\{Q, Q_1\}P = 0$
    \vspace{1mm}

    \hspace*{-1.7mm}{\rm (10)} $8((Q \times P_1)P + stP_1 + r^2P_1 + \varPhi^2P_1 + 2r\varPhi P_1) \,+ \,5\{Q, P_1\}P
    -2\{P, Q_1\}Q= 0$
    \vspace{1mm}

    \hspace*{-1.7mm}{\rm (11)} $18(\ad\,\varPhi)^2\varPhi_1 + Q \times \varPhi_1P - P \times \varPhi_1Q) +   B_{7,\sR}(\varPhi, \varPhi_1)\varPhi = 0$
    \vspace{1mm}

    \hspace*{-1.7mm}{\rm (12)} $18(\varPhi_1\varPhi P -2\varPhi\varPhi_1P - r\varPhi_1P - s\varPhi_1Q) + B_{7,\sR}(\varPhi, \varPhi_1)P = 0$
    \vspace{1mm}

    \hspace*{-1.7mm}{\rm (13)} $18(\varPhi_1\varPhi Q -2\varPhi\varPhi_1Q + r\varPhi_1Q - t\varPhi_1P) + B_{7,\sR}(\varPhi, \varPhi_1)Q = 0,$

    \vspace{1mm}
    \noindent for all $\varPhi_1 \in(\mathfrak{e}_{7,\sR})^C, P_1, Q_1 \in (\mathfrak{P}_{\sR})^C$, where $B_{7,\sR}$ is the Killing form of the Lie algebra $(\mathfrak{e}_{7,\sR})^C$.
\end{lemma}

\begin{proof}
    For $R = (\varPhi, P, Q, r, s, t) \in (\mathfrak{e}_{8,\sR})^C$, by doing simple computation of $(R\, \times\, R)R_1=0$ for all $ R_1=(\varPhi_1, P_1, Q_1, r_1, $ $s_1, t_1) \in (\mathfrak{e}_{8,\sR})^C$, we have the required relational formulas above.
\end{proof}

\begin{proposition}\label{proposition 6.8}
    The group $((E_{8,\sR})^C)_0$ acts on $\mathfrak{W}_{\sR}$ transitively.
\end{proposition}
\begin{proof}
    Since $\alpha \in (E_{8,\sR})^C$ leaves the Killing form $B_{8,\sR}$ invariant $: B_{8,\sR}(\alpha R,$ $\alpha R') = B_8(R, R'), R, R' \in (\mathfrak{e}_{8,\sR})^C$, the group $(E_{8,\sR})^C$ acts on $\mathfrak{W}_{\sR}$.
    Indeed, let $ R \in \mathfrak{W}_{\sR} $. Then, for all $  R_1 \in (\mathfrak{e}_{8,\sR})^C $, it follows that \vspace{-1mm}
    \begin{align*}
    (\alpha R \times \alpha R)R_1 &= [\alpha R, [\alpha R, \alpha R_1]\,] +
    \dfrac{1}{30}
    B_{8,\sR}(\alpha R, R_1)\alpha R
    \\[0mm]
    &= \alpha[\,[R, [R, \alpha^{-1}R_1]\,] +
    \dfrac{1}{30}
    B_{8,\sR}(R, \alpha^{-1}R_1)\alpha R
    \\[0mm]
    &= \alpha((R \times R)\alpha^{-1}R_1
    \\[0mm]
    &= 0.
    \end{align*}
    \vspace{-5mm}

\noindent Hence the group $(E_{8,\sR})^C$ acts on $\mathfrak{W}_{\sR}$, that is, $ \alpha R \in (E_{8,\sR})^C $.

We will show that this action is transitive. First, for all $  R_1 \in (\mathfrak{e}_{8,\sR})^C $, it follows from
    \begin{align*}
    (1_- \times 1_-)R_1 &= [1_-,[1_-, (\varPhi_1, P_1, Q_1, r_1, s_1, t_1)] \, ] +
    \dfrac{1}{30}
    B_8(1_-, R_1)1_-
    \\[0mm]
    &= [1_-, (0, 0, P_1, -s_1, 0, 2r_1)] + 2s_11_-
    \\[0mm]
    &= (0, 0, 0, 0, -2s_1) + 2s_11_-
    \\[0mm]
    &= 0
    \end{align*}
    that $1_- \in \mathfrak{W}_{\sR}$. Then, any element $R \in \mathfrak{W}_{\sR}$ can be transformed to $1_- \in \mathfrak{W}_{\sR}$ by some $\alpha \in ((E_{8,\sR})^C)_0$. This will be proved the following.
    \vspace{1mm}

    Case (i) where $R = (\varPhi, P, Q, r, s,t), t \not= 0$.

    From Lemma \ref{lemma 6.7} (2),(5) and (6), we see
    $$
    \varPhi = -\dfrac{1}{2t}Q \times Q, \; P = \dfrac{r}{t}Q - \dfrac{1}{6t^2}(Q \times Q)Q, \; s = -\dfrac{r^2}{t} + \dfrac{1}{96t^3}\{Q, (Q \times Q)Q\}.
    $$
    Now, let $\varTheta:= \ad(0, P_1, 0, r_1, s_1, 0) \in \ad((\mathfrak{e}_{8,\sR})^C), r_1\not=0 $ (Lemma \ref{lemma 6.7} (1)), so we compute $\varTheta^n1_-$:
    \begin{align*}
    &\quad \varTheta^n1_-
    \\
    &= \begin{pmatrix} ((-2)^{n-1} + (-1)^n){r_1}^{n-2}P_1 \times P_1
    \vspace{0.5mm}\\
    \Big((-2)^{n-1} - \dfrac{1 + (-1)^{n-1}}{2}\Big){r_1}^{n-2}s_1P_1 + \Big(\dfrac{1 - (-2)^n}{6} + \dfrac{(-1)^n}{2}\Big){r_1}^{n-3}(P_1 \times P_1)P_1
    \vspace{0.5mm}\\
    ((-2)^n + (-1)^{n-1}){r_1}^{n-1}P_1
    \vspace{0.5mm}\\
    (-2)^{n-1}{r_1}^{n-1}s_1
    \vspace{0.5mm}\\
    -((-2)^{n-2} + 2^{n-2}){r_1}^{n-2}{s_1}^2 + \dfrac{2^{n-2} + (-2)^{n-2} - (-1)^{n-1}-1}{24}{r_1}^{n-4}\{P_1,(P_1 \times P_1)P_1\}
    \vspace{0.5mm}\\
    (-2)^n{r_1}^n \end{pmatrix}.
    \end{align*}
    Then, by doing straightforward computation, we have
    \begin{align*}
    &\quad (\exp\varTheta)1_- = \Big(\dsum_{n= 0}^\infty\dfrac{1}{n!}\varTheta^n \Big)1_-
    \vspace{0.5mm}\\
    &= \begin{pmatrix}   -\dfrac{1}{2{r_1}^2}(e^{-2r_1} -2e^{-r_1} + 1)P_1 \times P_1
    \vspace{0.5mm}\\
    \dfrac{s_1}{2{r_1}^2}(-e^{-2r_1} - e^{r_1} + e^{-r_1} + 1)P_1 + \dfrac{1}{6{r_1}^3}(-e^{-2r_1} + e^{r_1} + 3e^{-r_1} - 3)(P_1 \times P_1)P_1
    \vspace{0.5mm}\\
    \dfrac{1}{r_1}(e^{-2r_1} - e^{-r_1})P_1
    \vspace{0.5mm}\\
    \dfrac{s_1}{2r_1}(1 - e^{-2r_1})
    \vspace{0.5mm}\\
    -\dfrac{{s_1}^2}{4{r_1}^2}(e^{-2r_1} + e^{2r_1} -2) + \dfrac{1}{96{r_1}^4}(e^{2r_1} + e^{-2r_1} - 4e^{r_1} - 4e^{-r_1} + 6)\{P_1, (P_1 \times P_1)P_1\}
    \vspace{0.5mm}\\
    e^{-2r_1} \end{pmatrix}.
    \end{align*}
Here, in the case where $ t\not=1 $, for a given $ Q \in (\mathfrak{J}_{\sR})^C,r,t \in C $ in Case (i) above,
we can choose $ P_1 \in (\mathfrak{J}_{\sR})^C,r_1,s_1 \in C $ satisfying the conditions
\begin{align*}
\dfrac{1}{r_1}(e^{-2r_1} - e^{-r_1})P_1=Q, \;\; \dfrac{s_1}{2r_1}(1 - e^{-2r_1})=r, \;\; e^{-2r_1}=t.
\end{align*}
 Indeed,
we choose some one value of $ \sqrt{t} $ satisfying $ (\sqrt{t})^2=t $ and some one value of $ \log t $ for $ t \in C $, respectively. Then because of $ t-t^2\not=0, t-1\not=0 $, we can get
\begin{align*}
P_1=\dfrac{(\sqrt{t}+t)\log t}{2(t-t^2)}Q,\;\; r_1=-\dfrac{\log t}{2},\;\;  s_1=\dfrac{\log t}{t-1}r
\end{align*}
as $ P_1,r_1, s_1 $ satisfying the conditions above.
    Hence, by using $ P_1,r_1,s_1 $ obtained above, we obtain
    \begin{align*}
    (\exp\varTheta)1_- = \begin{pmatrix} -\dfrac{1}{2t}Q \times Q
    \vspace{1mm}\\
    \dfrac{r}{t}Q - \dfrac{1}{6t^2}(Q \times Q)Q
    \vspace{0mm}\\
    Q
    \vspace{0mm}\\
    r
    \vspace{0mm}\\
    -\dfrac{r^2}{t} + \dfrac{1}{96 t^3}\{Q, (Q \times Q)Q\}
    \vspace{0mm}\\
    t
    \end{pmatrix}
    =R.
    \end{align*}
    In the case where $ t=1 $. Then, for a given $ R=(\varPhi,P,Q,r,s,t),t\not=0 $, $ R $ is of the form
    \begin{align*}
    \varPhi=-\dfrac{1}{2}Q \times Q,\;\;P=rQ-\dfrac{1}{6}(Q\times Q)Q,\;\;s=-r^2+\dfrac{1}{96}\{Q,(Q \times Q)Q\},
    \end{align*}
    so let $\varTheta:= \ad(0, P_1, 0, 0, s_1, 0) \in \ad((\mathfrak{e}_{8,\sR})^C) $, we have
    \begin{align*}
    (\exp\varTheta)1_-
    = \Big(\dsum_{n= 0}^\infty\dfrac{1}{n!}\varTheta^n \Big)1_-
    = \begin{pmatrix}
    -\dfrac{1}{2}P_1 \times P_1
    \vspace{1mm}\\
    -s_1P_1+ \dfrac{1}{6}(P_1 \times P_1)P_1
    \vspace{0mm}\\
    -P_1
    \vspace{0mm}\\
    s_1
    \vspace{0mm}\\
    -{s_1}^2+\dfrac{1}{96}\{P_1, (P_1 \times P_1)P_1\}
    \vspace{0mm}\\
    1
    \end{pmatrix}.
    \end{align*}
    Here, as in the case where $ t\not=1 $, we choose $ P_1 \in (\mathfrak{J}_{\sR})^C, s_1 \in C $ satisfying $ -P_1=Q, s_1=r $. Hence, by using these $ P_1,s_1 $, we obtain
    \begin{align*}
    (\exp\varTheta)1_- = \begin{pmatrix} -\dfrac{1}{2}Q \times Q
    \vspace{1mm}\\
    rQ-\dfrac{1}{6}(Q \times Q)Q
    \vspace{0mm}\\
    Q
    \vspace{0mm}\\
    r
    \vspace{0mm}\\
    -r^2 + \dfrac{1}{96}\{Q, (Q \times Q)Q\}
    \vspace{0mm}\\
    1
    \end{pmatrix}
    =R.
    \end{align*}
    Thus $R$ is transformed to $1_-$ by $(\exp \varTheta)^{-1} \in ((E_{8,\sR})^C)_0$.
\vspace{2mm}

    Case (ii) where  $R = (\varPhi, P, Q, r, s, 0), s \not= 0$.

    Let $\varTheta:=\ad(0, 0, 0, 0, {\pi}/{2}, -{\pi}/{2})) \in (\mathfrak{e}_{8,\sR})^C$. Then we have
    $$
    (\exp\varTheta) R =(\varPhi, Q,-P, -r, 0, -s), \;\;\; -s \not= 0.
    $$
    Hence this case can be reduced to Case (i).
 \vspace{2mm}

    Case (iii) where  $R = (\varPhi, P, Q, r, 0, 0), r \not= 0$.

    From Lemma \ref{lemma 6.7} (2),(5) and (6), we have
    $$
    Q \times Q = 0, \;\; \varPhi Q = -3rQ, \;\; \{P, Q\} = 16r^2.
    $$
    Then, let $\ad(0, Q, 0, 0, 0, 0) \in \ad((\mathfrak{e}_{8,\sR})^C)$ (Lemma \ref{lemma 6.5} (1)), we have
    $$
    (\exp\ad(0, Q, 0, 0, 0, 0))R = (\varPhi, P + 2rQ, Q, r, -4r^2, 0), \;\; -4r^2 \not= 0.
    $$
    Hence this case can be reduced to Case (ii).
 \vspace{2mm}

    Case (iv) where $R = (\varPhi, P, Q, 0, 0, 0), Q \not= 0$.

    We can choose $P_1 \in (\mathfrak{P}_{\sR})^C$ such that $\{P_1, Q\} \not= 0$. Indeed, for all $ P_1 \in (\mathfrak{P}_{\sR})^C $, suppose $ \{P_1,Q\}=0 $. Then we have $ Q=0 $. This is contradiction. Hence there exists $ P_1 \in (\mathfrak{P}_{\sR})^C $ such that $\{P_1, Q\} \not= 0$.

    Now, let $\varTheta:= \ad(0, P_1, 0, 0, 0, 0) \in \ad((\mathfrak{e}_{8,\sR})^C)$ (Lemma \ref{lemma 6.5} (1)), we have
    \begin{align*}
     (\exp \varTheta)R
    = \left( \begin{array}{c}
    \varPhi+P_1 \times Q
    \vspace{1mm}\\
    \,P-\varPhi P_1+\dfrac{1}{2}(P_1 \times Q)P_1
    \vspace{1mm}\\
    Q
    \vspace{1mm}\\
    -\dfrac{1}{8}\{P_1, Q\}
    \vspace{1mm}\\
    \dfrac{1}{4}\{P_1, P\}+\dfrac{1}{8}\{P_1, -\varPhi P_1\}+\dfrac{1}{24}\{P_1,( P_1 \times Q)P_1\}
    \vspace{1mm}\\
    0
     \end{array} \right),\;\;  -\dfrac{1}{8}\{P_1, Q\} \not =0.
    \end{align*}
    Hence this case can be reduced to Case (iii).
\vspace{2mm}

    Case (v) where $R = (\varPhi, P, 0, 0, 0, 0), P \not= 0$.

    As in Case (iv), we choose $Q_1 \in (\mathfrak{P}_{\sC})^C$ such that $\{P, Q_1\} \not= 0$. Then, let $\varTheta:= \ad(0, 0, Q_1, 0, 0, 0) \in \ad((\mathfrak{e}_{8,\sR})^C)$ (Lemma \ref{lemma 6.5} (1)), we have
    \begin{align*}
    (\exp \varTheta)R
    =\left( \begin{array}{c}
    \varPhi-P \times Q_1
    \vspace{1mm}\\
     P
    \vspace{1mm}\\
     -\varPhi Q_1-\dfrac{1}{2}(P \times Q_1)Q_1
    \vspace{1mm}\\
    \dfrac{1}{8}\{P, Q_1\}
    \vspace{1mm}\\
    0
     \vspace{1mm}\\
     -\dfrac{1}{8}\{Q_1, -\varPhi Q_1\} -\dfrac{1}{24}\{Q_1, -(P \times Q_1)Q_1  \}
    \end{array}\right),\;\; \dfrac{1}{8}\{P, Q_1\} \not=0.
    \end{align*}
    Hence this case can be reduced to Case (iii).
\vspace{2mm}

    Case (vi) where $R = (\varPhi, 0, 0, 0, 0, 0), \varPhi \not= 0.$
    From  Lemma \ref{lemma 6.7} (9) or (10), we have $\varPhi^2=0$.
    We choose $P_1 \in (\mathfrak{P}_{\sC})^C$ such that $\varPhi P_1 \not = 0$. Indeed, for all $ P_1 \in (\mathfrak{P}_{\sR})^C $, suppose $ \varPhi P_1=0 $. Then we have $ \varPhi=0 $. This is contradiction. Hence there exists $ P_1 \in (\mathfrak{P}_{\sR})^C $ such that $ \varPhi P_1 \not= 0$.
    Then, let $ \varTheta:=\ad(0,P_1,0,0,0,0) \in \ad((\mathfrak{e}_{8,\sR})^C)$ (Lemma \ref{lemma 6.5} (1)), we have
    $$
    (\exp\ad(0, P_1, 0, 0, 0, 0))R = \Big(\varPhi, -\varPhi P_1, 0, 0, \dfrac{1}{8}\{\varPhi P_1, P_1\}, 0 \Big),\;\; -\varPhi P_1\not=0.
    $$
    Hence this case is also reduced to Case (v).

    With above, the proof of this proposition is completed.
\end{proof}

Now, we will prove the theorem as the aim of this section.

\begin{theorem}\label{theorem 6.9}
    The homogeneous space $(E_{8,\sR})^C/((E_{8,\sR})^C)_{1_-}$ is diffeomorphic to the space $\mathfrak{W}_{\sR}${\rm : } \\
    $(E_{8,\sR})^C/((E_{8,\sR})^C)_{1_-} \simeq \mathfrak{W}_{\sR}$.

    In particular, the group $(E_{8,\sR})^C$ is connected.
\end{theorem}
\begin{proof}
    Since the group $(E_{8,\sR})^C$ acts on the space $\mathfrak{W}_{\sR}$ transitively (Proposition \ref{proposition 6.8}), the former half of this theorem is proved.

    The latter half can be shown as follows. Since the group $((E_{8,\sR})^C)_{1_-}$ and the space $\mathfrak{W}_{\sR}=((E_{8,\sR})^C)_0 1_-$  are connected (Propositions \ref{proposition 6.6}, \ref{proposition 6.8}), $(E_{8,\sR})^C$ is also connected.
\end{proof}

We will determine the structure of the group $ (E_{8,\sR})^C $.

\begin{theorem}\label{Theorem 6.10}
    The group $ (E_{8,\sR})^C $ is isomorphic to the group $ {F_4}^C ${\rm :} $ (E_{8,\sR})^C \cong {F_4}^C $.
\end{theorem}
\begin{proof}
    Since the group $ (E_{8,\sR})^C $ is connected (Theorem \ref{theorem 6.9}) and the type of its associated Lie algebra $ (\mathfrak{e}_{8,\sR})^C $ is $ F_4 $ (Theorem \ref{theorem 6.3}) and the center $ z({F_4}^C) $ of the group $ {F_4}^C $ is trivial.

Therefore the group $ (E_{8,\sR})^C $ has be isomorphic to the group $ {F_4}^C $:
\begin{align*}
    (E_{8,\sR})^C \cong {F_4}^C.
\end{align*}
\end{proof}


We consider a real form $ E_{8,\sR} $ of $ (E_{8,\sR})^C $:
\begin{align*}
E_{8,\sR}=\left\lbrace \alpha \in (E_{8,\sR})^C \relmiddle{|} \langle \alpha R_1, \alpha R_2 \rangle=\langle R_1, R_2 \rangle \right\rbrace,
\end{align*}
where the Hermite inner product $ \langle R_1, R_2 \rangle $ is defined by $ (-1/15)B_{8,\sR}(R_1, \tau\lambda_\omega R_2) $: $  \langle R_1, R_2 \rangle =(-1/15)\allowbreak B_{8,\sR}(R_1, \tau\lambda_\omega R_2)$. Then this group is compact Lie group as the closed subgroup of the unitary group $ U(52)=\left\lbrace \alpha \in \Iso_C((\mathfrak{e}_{8,\sR})^C)\relmiddle{|} \langle \alpha R_1, \alpha R_2 \rangle=\langle R_1, R_2 \rangle \right\rbrace  $.
\vspace{2mm}

Now, we will determine the structure of the group $ E_{8,\sR} $.

\begin{theorem}\label{Theorem 10.0.1}
    The group $ E_{8,\sR} $ is isomorphic to the group $ F_4 ${\rm :} $ E_{8,\sR} \cong F_4 $.
\end{theorem}
\begin{proof}
  Since the complex Lie algebra $ (\mathfrak{e}_{8,\sR})^C $ of the group $ (E_{8,\sR})^C$ is the type of $ F_4 $ (Theorem \ref{theorem 6.3}), its real form is also the type $ F_4 $. Hence the real form of the complex Lie group $ E_{8,\sR} $ is isomorphic to one of the following groups
    \begin{align*}
    F_4, \,\,F_{4(4)}, \,\,F_{4(-20)}.
    \end{align*}
    However, $ E_{8,\sR} $ is compact.

    Therefore $ E_{8,\sR} $ have to be isomorphic to the compact Lie group $ F_4 $:
\begin{align*}
E_{8,\sR}  \cong F_4.
\end{align*}
\end{proof}

\end{document}